\documentclass[11pt,oneside,a4paper]{amsart}

\usepackage[
  paper=a4paper,
  headsep=17pt,footskip=15pt,text={143mm,237mm},centering
]{geometry}

\usepackage{amssymb,amsxtra,bbm}
\usepackage[mathcal]{euscript}

\usepackage{tikz}
\usetikzlibrary{
  cd,
  calc,
  positioning,
  arrows,
  decorations.pathreplacing,
  decorations.markings,
}
\usepackage[mode=buildnew]{standalone}
\usepackage{graphicx,xcolor}
\usepackage{float,array,multirow,calc,booktabs}
\usepackage{enumitem}
\setlist[enumerate,1]{label=\textup{(\arabic*)}}

\definecolor{todo-background-color}{gray}{0.95}
\usepackage[
  textwidth=1.1in,
  backgroundcolor=todo-background-color,
  bordercolor=black,
  linecolor=black,
  textsize=footnotesize
]{todonotes}

\usepackage[bookmarks]{hyperref}
\hypersetup{colorlinks=true,linkcolor=blue,citecolor=blue,urlcolor=blue}

\usepackage[nobysame]{amsrefs}
\def\MR#1{}

\iftrue
  \makeatletter
  \gdef\@settitle{%
    \vspace*{10pt}
    \begin{flushleft}%
      \Large\bfseries
      \strut\@title\strut
    \end{flushleft}%
  }
  \gdef\@setauthors{%
    \begingroup
    \def\thanks{\protect\thanks@warning}%
    \trivlist
    \raggedright
    \large \@topsep31\p@\relax
    \advance\@topsep by -\baselineskip
    \item\relax
    \author@andify\authors
    \def\\{\protect\linebreak}%
    \authors
    \ifx\@empty\contribs
    \else
      ,\penalty-3 \space \@setcontribs
      \@closetoccontribs
    \fi
    \normalfont
    \endtrivlist
    \endgroup
  }
  \gdef\@setaddresses{\par
    \nobreak \begingroup
    \small\raggedright
    \def\author##1{\nobreak\addvspace\smallskipamount}%
    \def\\{\unskip, \ignorespaces}%
    \interlinepenalty\@M
    \def\address##1##2{\begingroup
      \par\addvspace\bigskipamount\noindent
      \@ifnotempty{##1}{(\ignorespaces##1\unskip) }%
      {\ignorespaces##2}\par\endgroup}%
    \def\curraddr##1##2{\begingroup
      \@ifnotempty{##2}{\nobreak\noindent\curraddrname
        \@ifnotempty{##1}{, \ignorespaces##1\unskip}\/:\space
        ##2\par}\endgroup}%
    \def\email##1##2{\begingroup
      \@ifnotempty{##2}{\nobreak\noindent E-mail address%
        \@ifnotempty{##1}{, \ignorespaces##1\unskip}\/:\space
        \ttfamily##2\par}\endgroup}%
    \def\urladdr##1##2{\begingroup
      \def~{\char`\~}%
      \@ifnotempty{##2}{\nobreak\noindent\urladdrname
        \@ifnotempty{##1}{, \ignorespaces##1\unskip}\/:\space
        \ttfamily##2\par}\endgroup}%
    \addresses
    \endgroup
    \global\let\addresses=\@empty
  }
  \gdef\@setabstracta{%
      \ifvoid\abstractbox
    \else
      \skip@19pt \advance\skip@-\lastskip
      \advance\skip@-\baselineskip \vskip\skip@
      \box\abstractbox
      \prevdepth\z@ 
      \vskip-28pt
    \fi
  }
  \renewenvironment{abstract}{%
    \ifx\maketitle\relax
      \ClassWarning{\@classname}{Abstract should precede
        \protect\maketitle\space in AMS document classes; reported}%
    \fi
    \global\setbox\abstractbox=\vtop \bgroup
      \normalfont\small
      \list{}{\labelwidth\z@
        \leftmargin0pc \rightmargin\leftmargin
        \listparindent\normalparindent \itemindent\z@
        \parsep\z@ \@plus\p@
        
      }%
      \item[\hskip\labelsep\bfseries\abstractname.]%
  }{%
    \endlist\egroup
    \ifx\@setabstract\relax \@setabstracta \fi
  }

  \gdef\ps@headings{\ps@empty
    \def\@evenhead{%
      \setTrue{runhead}%
      \normalfont\scriptsize
      \rlap{\thepage}\hfill
      \def\thanks{\protect\thanks@warning}%
      \leftmark{}{}}%
    \def\@oddhead{%
      \setTrue{runhead}%
      \normalfont\scriptsize
      \def\thanks{\protect\thanks@warning}%
      \rightmark{}{}\hfill \llap{\thepage}}%
    \let\@mkboth\markboth
  }\ps@headings

  \gdef\section{\@startsection{section}{1}%
    \z@{-1.4\linespacing\@plus-.5\linespacing}{.8\linespacing}%
    {\normalfont\bfseries\large}}
  \gdef\subsection{\@startsection{subsection}{2}%
    \z@{-.8\linespacing\@plus-.3\linespacing}{.5\linespacing\@plus.2\linespacing}%
    {\normalfont\bfseries}}
  \gdef\subsubsection{\@startsection{subsubsection}{3}%
    \z@{.7\linespacing\@plus.2\linespacing}{-1.5ex}%
    {\normalfont\bfseries}}
  \gdef\@secnumfont{\bfseries}

  \renewcommand\contentsnamefont{\bfseries}
  \gdef\@starttoc#1#2{\begingroup
    \setTrue{#1}%
    \par\removelastskip\vskip\z@skip
    \@startsection{}\@M\z@{\linespacing\@plus\linespacing}%
      {.5\linespacing}{
        \contentsnamefont}{#2}%
    \ifx\contentsname#2%
    \else \addcontentsline{toc}{section}{#2}\fi
    \makeatletter
    \@input{\jobname.#1}%
    \if@filesw
      \@xp\newwrite\csname tf@#1\endcsname
      \immediate\@xp\openout\csname tf@#1\endcsname \jobname.#1\relax
    \fi
    \global\@nobreakfalse \endgroup
    \addvspace{32\p@\@plus14\p@}%
    \let\tableofcontents\relax
  }
  \gdef\contentsname{Contents}
  \gdef\l@section{\@tocline{2}{.5ex}{0mm}{5pc}{}}
  \gdef\l@subsection{\@tocline{2}{0pt}{2em}{5pc}{}}
  \makeatother
\fi 

\gdef\to{\mathchoice{\longrightarrow}{\rightarrow}{\rightarrow}{\rightarrow}}
\makeatletter
\newcommand{\shortxra}[2][]{\ext@arrow 0359\rightarrowfill@{#1}{#2}}
\gdef\longrightarrowfill@{\arrowfill@\relbar\relbar\longrightarrow}
\newcommand{\longxra}[2][]{\ext@arrow 0359\longrightarrowfill@{#1}{#2}}
\renewcommand{\xrightarrow}[2][]{\mathchoice{\longxra[#1]{#2}}%
  {\shortxra[#1]{#2}}{\shortxra[#1]{#2}}{\shortxra[#1]{#2}}}
\makeatother

\gdef\longtwoheadrightarrow{\relbar\joinrel\twoheadrightarrow}

\gdef\otimesover#1{\mathbin{\mathop{\otimes}_{#1}}}

\makeatletter
\gdef\Nopagebreak{\@nobreaktrue\nopagebreak}

\newtheoremstyle{mytheorem}
        {}{}              
        {\itshape}                      
        {}                              
        {\bfseries}                     
        {.}                
        {5\p@ plus\p@ minus\p@\relax}         
        {\thmname{#1} #2\thmnote{ {\the\thm@notefont(#3)}}}

\makeatother

\theoremstyle{plain}
\newtheorem*{conjecture*}{Conjecture}
\newtheorem*{assertion*}{Assertion}

\theoremstyle{mytheorem}
\newtheorem{theorem}{Theorem}[section]
\newtheorem{theoremalpha}{Theorem}
\newtheorem{proposition}[theorem]{Proposition}

\newtheorem{lemma}[theorem]{Lemma}

\newtheorem{conjecture}[theorem]{Conjecture}

\newtheorem{assertion}{Assertion}[section]

\theoremstyle{definition}
\newtheorem{definition}[theorem]{Definition}
\newtheorem{question}[theorem]{Question}
\newtheorem{example}[theorem]{Example}
\newtheorem{remark}[theorem]{Remark}

\newtheoremstyle{theorem-giventitle}
        {}{}              
        {\itshape}                      
        {}                              
        {\bfseries}                     
        {.}                             
        { }                             
        {\thmnote{\bfseries#3}}
\theoremstyle{theorem-giventitle}
\newtheorem{theorem-named}{}
\newtheorem{question-named}{}
\newtheorem{conjecture-named}{}
\newtheoremstyle{definition-giventitle}
        {}{}              
        {}                      
        {}                              
        {\bfseries}                     
        {.}                             
        {.7em}                             
        {\thmnote{\bfseries#3}}
\theoremstyle{definition-giventitle}
\newtheorem{step-named}{}
\newtheorem{case-named}{}
\newtheorem{remark-named}{}

\numberwithin{equation}{section}

\def\Z{\mathbb{Z}}
\def\Q{\mathbb{Q}}
\def\R{\mathbb{R}}
\def\C{\mathbb{C}}

{
  \foreach \l in {A,...,Z} {
    \expandafter\xdef\csname c\l\endcsname{\noexpand\mathcal{\l}}
    \expandafter\xdef\csname b\l\endcsname{\noexpand\mathbb{\l}}
  }
}

\def\Ker{\operatorname{Ker}}
\def\Coker{\operatorname{Coker}}

\def\Im{\operatorname{Im}}
\def\Hom{\mathrm{Hom}}

\def\sign{\operatorname{sign}}

\def\sm{\setminus}
\makeatletter
\def\tpmod#1{{\@displayfalse\pmod{#1}}}
\makeatother

\def\Bl{B\ell}
\def\lk{\mathrm{lk}}

\def\rhot{\rho^{(2)}}
\def\osigmat{\bar\sigma^{(2)}}

\def\spinc{spin$^c$}
\DeclareMathOperator\Spin{Spin}
\def\Spinc{\Spin^c}
\def\Char{\operatorname{Char}}

\def\mathbinover#1#2{\mathbin{\mathop{#1}\limits_{#2}}}

\def\cupover#1{\mathbinover{\cup}{#1}}

\def\Wh{\mathrm{Wh}}

\def\top{{\mathrm{top}}}
\def\smooth{{\mathrm{sm}}}

\def\gr{\mathrm{gr}}

\def\setminus{\smallsetminus}

\def\nstrut{{\vphantom{1}}}

\def\sbmatrix#1{\left[\begin{smallmatrix}#1\end{smallmatrix}\right]}

\begin{document}

\vspace*{0mm}

\title%
[Primary decomposition in the smooth concordance group]
{Primary decomposition in the smooth concordance group of
topologically slice knots}

\author{Jae Choon Cha}
\address{
  Center for Research in Topology\\
  POSTECH\\
  Pohang 37673\\
  Republic of Korea\\
  and\linebreak
  School of Mathematics\\
  Korea Institute for Advanced Study \\
  Seoul 02455\\
  Republic of Korea
}
\email{jccha@postech.ac.kr}

\gdef\subjclassname{\textup{2010} Mathematics Subject Classification}
\expandafter\let\csname subjclassname@1991\endcsname=\subjclassname
\expandafter\let\csname subjclassname@2000\endcsname=\subjclassname
\subjclass{%
}

\begin{abstract}
  We address primary decomposition conjectures for knot concordance groups,
  which predict direct sum decompositions into primary parts.  We show
  that the smooth concordance group of topologically slice knots has a large
  subgroup for which the conjectures are true and there are infinitely many
  primary parts each of which has infinite rank. This supports the conjectures
  for topologically slice knots.  We also prove analogues for the associated
  graded groups of the bipolar filtration of topologically slice knots.  Among
  ingredients of the proof, we use amenable $L^2$-signatures, Ozsv\'ath-Szab\'o
  $d$-invariants and N\'emethi's result on Heegaard Floer homology of Seifert
  3-manifolds. In an appendix, we present a general formulation of the
  notion of primary decomposition.
\end{abstract}

\maketitle

\thispagestyle{empty}

\section{Introduction and main results}
\label{section:introduction}

It is a major open problem to classify knots in 3-space modulo concordance. Our
understanding is far from complete, for both topological and smooth knot
concordance groups.  The sophistication of the smooth case beyond topological
concordance is measured by the smooth concordance group of topologically slice
knots, which has been actively investigated using modern smooth techniques.

In the study of knot concordance, the notion of \emph{primary decomposition}
first appeared in Jerome Levine's foundational work~\cites{Levine:1969-1,
Levine:1969-2}. Briefly, he constructed an algebraic concordance group of
Seifert matrices and proved that it is isomorphic to the knot concordance group
in high odd dimensions, while it gives algebraic invariants in the classical
dimension~\cite{Levine:1969-1}. He proved that a rational coefficient version of
the algebraic concordance group decomposes into a direct sum of certain
``primary parts'' indexed by irreducible factors of Alexander polynomials.  This
plays a crucial role in his well-known classification result that the algebraic
concordance group and high odd dimensional knot concordance groups are
isomorphic to $\Z^\infty\oplus(\Z_2)^\infty \oplus
(\Z_4)^\infty$~\cite{Levine:1969-2}.

For the low dimensional case, it is natural to ask whether the classical knot
concordance groups and related objects admit analogous primary decomposition,
and to study the structures via primary parts.  In an appendix, we formulate a
general notion of primary decomposition, which specializes to several specific
cases including concordance and rational homology cobordism, and discuss related
questions. We hope this is useful for future study as well. The appendix also
discusses known earlier results from the viewpoint of the general formulation.
In particular, for topological knot concordance, there were remarkable results
related to primary decomposition~\cites{Livingston:2002-2, Kim:2005-2,
Kim-Kim:2008-1, Cochran-Harvey-Leidy:2009-2, Cochran-Harvey-Leidy:2009-3,
Kim-Kim:2014-1}.

In this paper, we begin a detailed study of the smooth concordance group of
topologically slice knots~$\cT$, via primary decomposition. Precise statements
of our results are given in Sections~\ref{subsection:pd-top-slice}
and~\ref{subsection:pd-bipolar} below. Briefly, the conjectural primary
decompositions (see Question~\ref{question:pd-top-slice}) are direct sum
decompositions, along irreducible factors of Alexander polynomials, of the
quotient $\cT/\Delta$ where $\Delta$ is the subgroup generated by knots with
unit Alexander polynomials. (Taking the quotient by $\Delta$ is along the same
lines of ignoring units for factorizations in a ring.)  We show that the
conjectures hold for a large subgroup of~$\cT$, and that there are infinitely
many primary parts each of which has infinite rank (see
Theorems~\ref{theorem:main-top-slice}
and~\ref{theorem:main-independence-top-slice})\@. This provides evidence
supporting primary decomposition conjectures, and in addition reveals a rich
structure in $\cT/\Delta$, generalizing a result of Hedden, Livingston and
Ruberman~\cite{Hedden-Livingston-Ruberman:2010-1} that $\cT/\Delta$ has infinite
rank.

In proving this, the essential challenge is to find irreducible polynomials
$\lambda(t)$ and topologically slice knots (whose nontrivial linear combinations
are) not concordant to any knot whose Alexander polynomial is relatively prime
to~$\lambda(t)$.  It appears to be hard, if possible at all, to do this using
known smooth invariants such as those from gauge theory, Heegaard Floer homology
and Khovanov homology.  Our proof combines amenable $L^2$-signatures, which are
key ingredients of recent studies of \emph{topological} concordance, with
\emph{smooth} information from Heegaard Floer homology of infinitely many
branched covers of a knot.  It seems intriguing to study whether more recent
smooth invariants are useful in understanding primary decomposition, motivated
from the results and approaches of this paper.

We also prove results which support primary decomposition conjectures for the
bipolar filtration of topologically slice knots (see
Theorems~\ref{theorem:main-bipolar}
and~\ref{theorem:main-independence-bipolar} below).

\subsection{Primary decomposition for topologically slice knots}
\label{subsection:pd-top-slice}

In what follows, we state primary decomposition conjectures and main results
for~$\cT$.

For a knot $K$ in $S^3$, denote the Alexander polynomial by~$\Delta_K$, and
regard it as an element in the Laurent polynomial ring~$\Q[t^{\pm1}]$.  Then
$\Delta_K$ is well-defined up to associates.  Recall that $\lambda$ and $\mu \in
\Q[t^{\pm1}]$ are \emph{associates} if $\lambda= at^k \mu$ for some $a\in \Q\sm
\{0\}$ and $k\in \Z$. The standard involution on $\Q[t^{\pm1}]$ is defined by
$\bigl(\sum a_i t^i\bigr)^* = \sum a_i t^{-i}$.  We say that $\lambda$ and $\mu
\in \Q[t^{\pm1}]$ are \emph{$*$-associates} if $\lambda$ is an associate of
either $\mu$ or~$\mu^*$.  

Denote the smooth concordance class of a knot $K$ by~$[K]$.  Recall that $\Delta
= \{[K] \in \cT \mid \Delta_K$ is trivial$\}$.  For an irreducible $\lambda$ in
$\Q[t^{\pm1}]$, let
\begin{align*}
  \cT_\lambda &=\{[K] \in \cT \mid \Delta_K \text{ is an associate of }
  (\lambda\lambda^*)^k\text{ for some }k\ge 0\},
  \\
  \cT^\lambda &= \{[K] \in \cT \mid \Delta_K
  \text{ is relatively prime to }\lambda\}.
\end{align*}
We remark that the product $\lambda\lambda^*$ in the definition of $\cT_\lambda$
reflects the Fox-Milnor condition that $\Delta_K$ is an associate of $ff^*$ for
some $f\in\Q[t^{\pm1}]$ when $K$ is topologically slice. Note that $\cT_\lambda$
and $\cT^\lambda$ are subgroups containing $\Delta$ for every
irreducible~$\lambda$. Also, $\cT_\lambda=\cT_\mu$ and $\cT^\lambda=\cT^\mu$ if
$\lambda$ and $\mu$ are $*$-associates.

Primary decomposition for topologically slice knots concerns natural
homomorphisms
\[
  \Phi_L\colon \bigoplus_{[\lambda]} \cT_\lambda/\Delta \to \cT/\Delta
  \quad\text{and}\quad
  \Phi_R\colon \cT/\Delta \to \bigoplus_{[\lambda]} \cT/\cT^\lambda.
\]
Here the index $[\lambda]$ of the direct sums varies over the $*$-associate
classes of irreducibles $\lambda$ that arise as the factor of an Alexander
polynomial of a knot.  The homomorphism $\Phi_L$ is defined to be the sum of the
inclusions $\cT_\lambda/\Delta \hookrightarrow \cT/\Delta$.  Since $\Delta_K$ is
a product of finitely many irreducibles, the quotient epimorphisms $\cT/\Delta
\twoheadrightarrow \cT/\cT^\lambda$ induce a homomorphism into the direct sum,
which is our $\Phi_R$ above.  This formulation is influenced by earlier work in
the literature, particularly Levine~\cite{Levine:1969-1,Levine:1969-2} and
Cochran, Harvey and
Leidy~\cite{Cochran-Harvey-Leidy:2009-2,Cochran-Harvey-Leidy:2009-3}.

\begin{remark-named}[An informal remark]
  
  One might regard elements in $\cT_\lambda/\Delta \subset \cT/\Delta$ as
  ``$\lambda$-primary'', and $\cT/\Delta \to \cT/\cT^\lambda$ as ``forgetting
  those coprime to~$\lambda$'' or ``extracting the $\lambda$-primary component''
  of an element. Then, the surjectivity of $\Phi_L$ means the existence of a
  ``decomposition into a sum of primary elements,'' and the injectivity  of
  $\Phi_L$ means the uniqueness of such a decomposition.  Also, the injectivity
  of $\Phi_R$ means that ``primary components determine an element uniquely,''
  and the surjectivity of $\Phi_R$ means that ``every combination of primary
  components is realizable.''

\end{remark-named}

\begin{question}
  [Primary decomposition for topologically slice knots]
  \phantomsection\leavevmode\Nopagebreak
  \label{question:pd-top-slice}
  \begin{enumerate}[label=\textup{(\arabic*)}]
    \item\label{item:left-pd-top-slice} Left primary decomposability: is
    $\Phi_L$ an isomorphism?
    \item\label{item:right-pd-top-slice} Right primary decomposability: is
    $\Phi_R$ an isomorphism?
  \end{enumerate}
  More generally, (if they are not isomorphisms) what are their kernel and
  cokernel?

  We conjecture an affirmative answer to~\ref{item:right-pd-top-conc} and that
  $\Phi_L$ is injective at the least. In any case, it appears to be interesting
  to study $\cT_\lambda/\Delta$ and $\cT/\cT^\lambda$, which we call \emph{left}
  and \emph{right primary parts} (or \emph{primary factors}). This leads us
  particularly to the following.

  \begin{enumerate}[resume]
    \item Nontriviality of primary parts: are $\cT_\lambda/\Delta$ and
    $\cT/\cT^\lambda$ nonzero for each irreducible $\lambda$ that arises as the factor of an Alexander polynomial of a knot?
  \end{enumerate}
  More generally (if they are nontrivial), what are the isomorphism types of
  the primary parts $\cT_\lambda/\Delta$ and~$\cT/\cT^\lambda$?
  \begin{enumerate}[resume]
    \item Relationship of left and right primary parts: is the
    composition
    \[
      \cT_\lambda/\Delta \hookrightarrow \cT/\Delta \twoheadrightarrow \cT/\cT^\lambda
    \]
    an isomorphism?
  \end{enumerate}
\end{question}

We remark that Definition~\ref{definition:weak-pd} in the appendix generalizes
the notion of left and right primary decomposition to a broader context. Also,
regarding the relationship of
Question~\ref{question:pd-top-slice}\ref{item:left-pd-top-slice}
and~\ref{item:right-pd-top-slice}, see Lemma~\ref{lemma:surj-inj-Phi_L-Phi_R} in
the appendix.

The first main result of this paper, which is given as
Theorem~\ref{theorem:main-top-slice} below, says that there is a large subgroup
of $\cT$ for which the answers to the above questions are affirmative and many
primary parts of $\cT$ are highly nontrivial.  To state the result, we use the
following notation. For a subgroup $\cS$ of $\cT$ and an irreducible $\lambda\in
\Q[t^{\pm1}]$, let $\cS_\lambda$ be the subgroup of $[K]\in \cS$ with $\Delta_K$
a power of $\lambda\lambda^*$, and $\cS^\lambda$ be the subgroup of $[K]\in \cS$
with $\gcd(\Delta_K,\lambda)=1$.  That is,  $\cS_\lambda = \cS\cap \cT_\lambda$
and $\cS^\lambda = \cS\cap \cT^\lambda$.  Then one can ask
Question~\ref{question:pd-top-slice} for $\cS$ in place of~$\cT$. Let
\begin{equation}
  \Lambda=\{(m+1)t-m \mid m \text{ is a positive integer}\}.
  \label{equation:collection-Lambda}  
\end{equation}
Note that $\Lambda$ is an infinite collection of pairwise non-$*$-associate irreducibles $\lambda$ such that $\lambda\lambda^*$ is an Alexander polynomial of a knot.

\begin{theoremalpha}
  \label{theorem:main-top-slice}
  There is a subgroup $\cS$ in $\cT$ containing $\Delta$ which satisfies the
  following:
  \begin{enumerate}
    \item \label{item:nontrivial-primary-parts}
    For every $\lambda\in\Lambda$, $\cS_\lambda/\Delta\cong \Z^\infty$,
    $\cS/\cS^\lambda \cong \Z^\infty$ and the composition $\cS_\lambda/\Delta
    \hookrightarrow \cS/\Delta \twoheadrightarrow \cS/\cS^\lambda$ is an
    isomorphism.
    \item \label{item:left-primary-decomposablity}
    The inclusions $\cS_\lambda/\Delta \to \cS/\Delta$ induce an isomorphism
    $\bigoplus_{\lambda\in \Lambda} \cS_\lambda/\Delta \to \cS/\Delta$.
    \item\label{item:right-primary-decomposablity}
    The surjections $\cS/\Delta \to \cS/\cS^\lambda$ induce an isomorphism
    $\cS/\Delta \to \bigoplus_{\lambda\in\Lambda} \cS/\cS^\lambda$.
  \end{enumerate}
  \begin{equation}
    \label{equation:subgroup-diagram}
    \begin{tikzcd}[row sep=large,column sep=6ex]
      \llap{$\bigoplus_{\lambda\in\Lambda} \Z^\infty={}$}
      \bigoplus_{\lambda\in \Lambda} \cS_\lambda/\Delta \ar[r,"\cong"] \ar[d,hook]
      & \cS/\Delta \ar[r,"\cong"] \ar[d,hook]
      & \bigoplus_{\lambda\in \Lambda} \cS/\cS^\lambda \ar[d,hook]
      \rlap{${}=\bigoplus_{\lambda\in\Lambda} \Z^\infty$}
      \\
      \bigoplus_{[\lambda]} \cT_\lambda/\Delta \ar[r,"\Phi_L"']
      & \cT/\Delta \ar[r,"\Phi_R"']
      & \bigoplus_{[\lambda]} \cT/\cT^\lambda
    \end{tikzcd}
  \end{equation}
\end{theoremalpha}

From Theorem~\ref{theorem:main-top-slice} it follows that each of the primary
parts $\cT_\lambda /\Delta$ and $\cT/\cT^\lambda$ has a subgroup isomorphic to
$\Z^\infty$ for all~$\lambda\in\Lambda$.  An immediate
consequence is the main result of \cite{Hedden-Livingston-Ruberman:2010-1} that
$\cT/\Delta$ has a subgroup isomorphic to~$\Z^\infty$.

Note that
Theorem~\ref{theorem:main-top-slice}\ref{item:nontrivial-primary-parts}
implies the following.

\begin{theoremalpha}
  \label{theorem:main-independence-top-slice}
  For each $\lambda\in \Lambda$, there is an infinite collection of
  topologically slice knots $\{K_{\lambda,i}\}_{i=1}^\infty$ with
  $\Delta_{K_{\lambda,i}}$ a power of $\lambda\lambda^*$, such that any nontrivial
  linear combination of the $K_{\lambda,i}$ is not smoothly concordant to any
  knot $J$ with $\Delta_J$ relatively prime to~$\lambda$.
\end{theoremalpha}

Indeed, Theorem~\ref{theorem:main-independence-top-slice} is equivalent to
Theorem~\ref{theorem:main-top-slice}, by an elementary formal argument.

\begin{proof}
  [Proof that Theorem~\ref{theorem:main-independence-top-slice} implies
  Theorem~\ref{theorem:main-top-slice}]

  Suppose that Theorem~\ref{theorem:main-independence-top-slice} holds.
  Let $\cS$ be the subgroup in $\cT$
  which is generated by $\Delta$ and the two-parameter family
  $\{K_{\lambda,i}\}_{\lambda\in\Lambda, i\in \bN}$ given by
  Theorem~\ref{theorem:main-independence-top-slice}.
  
  We claim that, for each $\lambda\in \Lambda$, $\cS_\lambda$ is equal to the
  subgroup generated by $\Delta$ and the one-parameter family
  $\{K_{\lambda,i}\}_{i\in\bN}$. To see this, first observe that for
  irreducibles $\lambda$ and $\mu$ which are not $*$-associates,
  $\cS_\lambda\subset \cS^\mu$ and so the composition
  \begin{equation}
    \label{item:distinct-primary-composition}
    \cS_\lambda/\Delta \hookrightarrow \cS/\Delta \twoheadrightarrow \cS/\cS^\mu    
  \end{equation}
  is zero. If a linear combination $[J] + \sum_\mu \sum_i r_{\mu,i} [K_{\mu,i}]$
  ($[J]\in \Delta$, $r_{\mu,i} \in \Z$) lies in $\cS_\lambda$, then for each
  $\mu\ne \lambda$, the image of the linear combination in $\cS/\cS^\mu$, which
  is represented by $\sum_{i} r_{\mu,i} [K_{\mu,i}]$, should be zero by the
  observation. Therefore $r_{\mu,i}=0$ for all $i$ and $\mu\ne\lambda$, by the
  conclusion of Theorem~\ref{theorem:main-independence-top-slice} that the
  classes $[K_{\mu,i}]$ are linearly independent in $\cS/\cS^\mu$. This proves
  the claim.

  Fix $\lambda\in \Lambda$, and temporarily denote by $\Z^\infty$ the free
  abelian group generated by the collection~$\{K_{\lambda,i}\}_i$.  The
  assignment $K_{\lambda,i} \mapsto [K_{\lambda,i}]$ gives rise to an
  epimorphism $\Z^\infty \twoheadrightarrow \cS_\lambda/\Delta$ by the claim.
  Since $\cS$ is generated by $\{K_{\lambda,i}\}_{\lambda,i}$ and
  $[K_{\mu,i}]=0$ in $\cS/\cS^\lambda$ for $\mu\ne\lambda$, the composition
  \begin{equation}
    \label{equation:same-primary-composition}
    \Z^\infty \twoheadrightarrow \cS_\lambda/\Delta \hookrightarrow \cS/\Delta \twoheadrightarrow \cS/\cS^\lambda   
  \end{equation}
  is surjective. Moreover, by the linear independence of $[K_{\lambda,i}]$ in
  $\cS/\cS^\lambda$, \eqref{equation:same-primary-composition}~is an
  isomorphism.  From this, it follows that both $\Z^\infty \to
  \cS_\lambda/\Delta$ and $\cS_\lambda/\Delta \to \cS/\cS^\lambda$ are
  isomorphisms.  This shows that
  Theorem~\ref{theorem:main-top-slice}\ref{item:nontrivial-primary-parts}
  holds.

  Since \eqref{item:distinct-primary-composition} is zero for $\mu\ne\lambda\in
  \Lambda$, the composition of two horizontal arrows on the top row
  of~\eqref{equation:subgroup-diagram} is the direct sum of the isomorphisms
  $\cS_\lambda/\Delta \to \cS/\cS^\lambda$. So the top row composition in
  \eqref{equation:subgroup-diagram} is an isomorphism. Also, the first
  horizontal arrow in the top row is surjective by the definition of~$\cS$. From
  this, it follows that
  Theorem~\ref{theorem:main-top-slice}\ref{item:left-primary-decomposablity}
  and~\ref{item:right-primary-decomposablity} hold.
\end{proof}

\subsection{Primary decomposition for the bipolar filtration}
\label{subsection:pd-bipolar}

The method of this paper provides further information on the structure of~$\cT$.
To discuss this, we consider the \emph{bipolar filtration} of~$\cT$, which was
defined by Cochran, Harvey and Horn~~\cite{Cochran-Harvey-Horn:2012-1}.  It is a
descending filtration
\[
  \{0\} \subset \cdots \subset \cT_n \subset \cdots \subset \cT_1
  \subset \cT_0 \subset \cT,
\]
where the subgroup $\cT_n$ consists of concordance classes of certain knots
called \emph{$n$-bipolar} (see
Definition~\ref{definition:positivity-negativity}).  It is known that various
modern smooth invariants vanish for knots in the subgroups~$\cT_0$ and~$\cT_1$,
but the associated graded groups $\gr_n(\cT):=\cT_n/\cT_{n+1}$ are nontrivial
for all~$n$~\cites{Cochran-Harvey-Horn:2012-1,Cha-Kim:2017-1}. Indeed, the
abelian group $\gr_n(\cT)$ is known to have infinite rank for
$n=0$~\cite{Cochran-Horn:2012-1}, and for all $n\ge 2$~\cite{Cha-Kim:2017-1}.

As an attempt to understand the structure of the filtration, we formulate and
study the primary decomposition of $\gr_n(\cT)$.  For an $n$-bipolar knot $K$,
denote its class in $\gr_n(\cT)$ by~$[K]$. Similarly to the case of~$\cT$, for
an irreducible element $\lambda$ in $\Q[t^{\pm1}]$, consider the following
subgroups of~$\gr_n(\cT)$.
\begin{align*}
  \gr_n(\cT)_\lambda &:=\{[K] \in \gr_n(\cT) \mid
  \Delta_K = (\lambda\lambda^*)^k \text{ for some } k \ge 0\},
  \\
  \gr_n(\cT)^\lambda &:= \{[K] \in \gr_n(\cT) \mid \Delta_K
  \text{ is relatively prime to }\lambda\}.
\end{align*}
Also, let $\gr_n(\Delta) =\{[K] \in \gr_n(\cT) \mid \Delta_K \text{\space is
trivial}\}$.
The injections $\gr_n(\cT)_\lambda/\gr_n(\Delta) \hookrightarrow
\gr_n(\cT)/\gr_n(\Delta)$ and the surjections $\gr_n(\cT)/\gr_n(\Delta)
\twoheadrightarrow
\gr_n(\cT)/\gr_n(\cT)^\lambda$ induce homomorphisms
\begin{align*}
  \Phi^n_L\colon & \bigoplus\nolimits_{[\lambda]}
  \gr_n(\cT)_\lambda/\gr_n(\Delta) \to \gr_n(\cT)/\gr_n(\Delta),
  \\
  \Phi^n_R\colon & \gr_n(\cT)/\gr_n(\Delta)
  \to \bigoplus\nolimits_{[\lambda]} \gr_n(\cT)/\gr_n(\cT)^\lambda.
\end{align*}
The following is an analogue of Question~\ref{question:pd-top-slice}.

\begin{question}
  [Primary decomposition for the associated graded]
  \phantomsection\leavevmode\Nopagebreak
  \label{question:pd-bipolar}
  \begin{enumerate}[label=\textup{(\arabic*)}]
    \item
    Is $\Phi^n_L$ an isomorphism?
    \item
    Is $\Phi^n_R$ an isomorphism?
    \item
    Are $\gr_n(\cT)_\lambda/\gr_n(\Delta)$ and $\gr_n(\cT)/\gr_n(\cT)^\lambda$
    nontrivial for every irreducible $\lambda$ that arises as a factor of an Alexander polynomial of a knot?
    \item
    Is the following composition an isomorphism?
    \[
      \gr_n(\cT)_\lambda/\gr_n(\Delta) \hookrightarrow
      \gr_n(\cT)/\gr_n(\Delta) \twoheadrightarrow \gr_n(\cT)/\gr_n(\cT)^\lambda
    \]
  \end{enumerate}
\end{question}

The following result supports affirmative answers.  Similarly to the case of
$\cT$, for a subgroup $\cS$ in $\gr_n(\cT)$, let
$\cS_\lambda=\cS\cap\gr_n(\cT)_\lambda$ and
$\cS^\lambda=\cS\cap\gr_n(\cT)^\lambda$. Recall that the collection $\Lambda$ has been defined in \eqref{equation:collection-Lambda} above.

\begin{theoremalpha}
  \label{theorem:main-bipolar}
  Let $n\ge 2$.  Then there is a subgroup $\cS$ in $\gr_n(\cT)$ which contains
  $\gr_n(\Delta)$ such that
  \[
    \bigoplus_{\lambda\in\Lambda} \cS_\lambda/\gr_n(\Delta)
    \xrightarrow{\cong} \cS/\gr_n(\Delta)
    \xrightarrow{\cong} \bigoplus_{\lambda\in\Lambda} \cS / \cS^\lambda
  \]
  and $\cS_\lambda/\gr_n(\Delta) \cong \Z^\infty \cong \cS /
  \cS^\lambda$ for every $\lambda\in\Lambda$.
\end{theoremalpha}

By the same argument as the proof that Theorem~\ref{theorem:main-top-slice} is
equivalent to Theorem~\ref{theorem:main-independence-top-slice}, it is seen
that Theorem~\ref{theorem:main-bipolar} is equivalent to the following
statement:

\begin{theoremalpha}
  \label{theorem:main-independence-bipolar}
  Let $n\ge 2$.  Then for each $\lambda\in \Lambda$, there are infinitely many
  topologically slice $n$-bipolar knots $K_{\lambda,i}$
  \textup{($i=1,2,\ldots$)} with $\Delta_{K_{\lambda,i}}$ a power of
  $\lambda\lambda^*$, which are linearly independent
  in~$\gr_n(\cT)/\gr_n(\cT)^\lambda$.
\end{theoremalpha}

Also, Theorem~\ref{theorem:main-independence-bipolar} implies
Theorem~\ref{theorem:main-independence-top-slice}\@.  Indeed, if a linear
combination $\#_i\, a_i K_{\lambda,i}$ of the knots $K_{\lambda,i}$ in
Theorem~\ref{theorem:main-independence-bipolar} is smoothly concordant to a knot
$L$ with $\Delta_L$ relatively prime to $\lambda$, then the class $[L]$
automatically lies in the subgroup $\cT_n$ of $\cT$ since so are
$K_{\lambda,i}$, and thus $[L]$ is zero in the quotient
$\gr_n(\cT)/\gr_n(\cT)^\lambda$.  It follows that $a_i=0$ for all~$i$, by
Theorem~\ref{theorem:main-independence-bipolar}\@. Therefore, to obtain
Theorems~\ref{theorem:main-top-slice}, \ref{theorem:main-independence-top-slice}
and~\ref{theorem:main-bipolar}, it suffices to prove
Theorem~\ref{theorem:main-independence-bipolar}\@.

The remaining part of this paper is organized as follows.
Sections~\ref{section:setup-proof}--\ref{section:d-invariant-estimate} are
devoted to the proof of Theorem~\ref{theorem:main-independence-bipolar}\@.  In
the appendix, we discuss a general formulation of the notion of primary
decomposition.

\subsubsection*{Ingredients of the proof}

The proof of the above results uses (the ideas of) several results in the
literature.  To extract obstructions to smooth concordance, we combine the
Cheeger-Gromov $L^2$ $\rho$-invariants, or equivalently $L^2$-signature defects,
and the Ozsv\'ath-Szab\'o $d$-invariant defined from Heegaard Floer homology,
following the approach of~\cite{Cha-Kim:2017-1}, which was motivated by
earlier work of Cochran, Harvey and Horn~\cite{Cochran-Harvey-Horn:2012-1}.  The
amenable signature theorem developed in~\cites{Cha-Orr:2009-1,Cha:2010-1} and
Ozsv\'ath-Szab\'o's $d$-invariant inequality for definite
4-manifolds~\cite{Ozsvath-Szabo:2003-2} are among the key ingredients.  We
develop and use a localization technique inspired by work of Cochran, Harvey and
Leidy~\cites{Cochran-Harvey-Leidy:2009-2, Cochran-Harvey-Leidy:2009-3}, to
produce representations of the fundamental group to which our $\rho$-invariants
are associated. Also, to compute and estimate $d$-invariants of infinitely many
branched covers of the infinitely many topologically slice knots in
Theorems~\ref{theorem:main-independence-top-slice}
and~\ref{theorem:main-independence-bipolar}, we use N\'emethi's
work~\cite{Nemethi:2005-1} on Heegaard Floer homology of negative definite
plumbed 3-manifolds.

\subsubsection*{Acknowledgements}

The author is indebted to Min Hoon Kim, Se-Goo Kim and Taehee Kim for
discussions which led him to prove the results described in the introduction.
The author thanks Chuck Livingston for his extremely helpful comments. Many
results in this paper were obtained during the author's visit to Max Planck
Institute for Mathematics in 2017--18.  The author is grateful to the institute
and Peter Teichner for the invitation and hospitality.  This work was also
partly supported by NRF grant 2019R1A3B2067839.  Finally the author thanks an anonymous referee for comments which were very helpful in improving the exposition.

\section{The first step of the proof of
Theorem~\ref{theorem:main-independence-bipolar}}
\label{section:setup-proof}

As a preliminary of the proof of
Theorem~\ref{theorem:main-independence-bipolar}, we recall the definition of the
bipolar filtration, from~\cite{Cochran-Harvey-Horn:2012-1}. Let $M(K)$ be the
zero framed surgery manifold of a knot $K$ in~$S^3$.

\begin{definition}[{\cite[Definition~5.1]{Cochran-Harvey-Horn:2012-1}}]
  \label{definition:positivity-negativity}
  Let $n\ge 0$ be an integer.  A compact connected 4-manifold $V$ bounded by
  $M(K)$ is an \emph{$n$-negaton} if the following are satisfied.
  \begin{enumerate}
  \item The inclusion induces an isomorphism on $H_1(M(K))\to H_1(V)$,
    and a meridian of $K$ normally generates~$\pi_1(V)$.
  \item There are disjointly embedded closed connected surfaces $S_i$ in $V$
    which form a basis for $H_2(V)$ and have self-intersection $-1$, or
    equivalently, the normal bundle of $S_i$ has Euler class~$-1$.
  \item For each $i$, the image of $\pi_1(S_i)$ lies in the $n$th derived
    subgroup~$\pi_1(V)^{(n)}$.
  \end{enumerate}
  If there is an $n$-negaton bounded by $M(K)$, then $K$ is called
  \emph{$n$-negative}.  An \emph{$n$-positon} and an \emph{$n$-positive knot}
  are defined by replacing $-1$ by~$+1$ in condition (2) above.
  A knot $K$ is \emph{$n$-bipolar} if $K$ is $n$-positive and $n$-negative.
\end{definition}

Recall that $\cT$ is the smooth concordance group of topologically slice knots.
The \emph{bipolar filtration} $\{\cT_n\}_{n\ge 0}$ of $\cT$ is defined by
\[
  \cT_n = \{[K] \in \cT \mid K\text{ is $n$-bipolar}\}.
\]
Since an $(n+1)$-bipolar knot is $n$-bipolar, $\{\cT_n\}_{n\ge 0}$ is a
descending filtration.  It is an open problem whether $\bigcap_{n\ge 0}
\cT_n=\{0\}$.

\subsection{Construction of a family of knots~\texorpdfstring{$\{K_i\}$}{Ki}}
\label{subsection:construction-K_i}

We start the proof of Theorem~\ref{theorem:main-independence-bipolar} with a
construction of knots which will be shown to generate the promised infinite rank
free abelian subgroup.

Fix an integer $n\ge 2$, as in
Theorem~\ref{theorem:main-independence-bipolar}\@.  Also, fix another integer
$m\ge 1$.
Several objects we will use below depend on~$(n,m)$, but we omit it from
notation since $(m,n)$ is fixed in our arguments. Let
\[
  \lambda(t) = \lambda_m(t) = (m+1)t-m \in \Lambda.
\]
For each $(n,m)$, we will construct an infinite family of knots $\{K_i\}$
indexed by integers~$i>0$, whose Alexander polynomial is
$\lambda(t)\lambda(t^{-1})$.  The construction is similar
to~\cite[Section~2.2]{Cha-Kim:2017-1}.  (See also
\cite{Cochran-Harvey-Horn:2012-1} which influenced the construction
of~\cite{Cha-Kim:2017-1} and ours.)  Let $R(J,D)$ be the knot shown in
Figure~\ref{figure:R(J,D)}.  Here $J$ and $D$ are knots that will be specified
later.  (For now, ignore the circles $\alpha_J$ and~$\alpha_D$.)  The knot
$R(J,D)$ bounds an obviously seen Seifert surface of genus one, which consists
of a 0-handle and two 1-handles.  In Figure~\ref{figure:R(J,D)}, the two
1-handles are untwisted and cross each other $2m+1$ times.  So, $S=\sbmatrix{0 &
m+1\\m & 0}$ is a Seifert matrix. We remark that \cite{Cha-Kim:2017-1}
and~\cite{Cochran-Harvey-Horn:2012-1} use the particular case of $m=1$.

\begin{figure}[H]
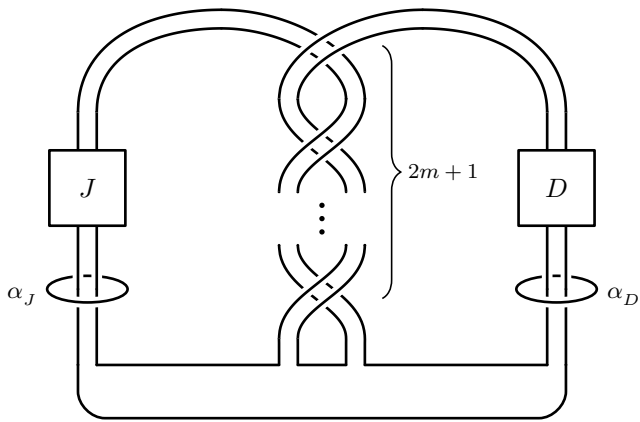

  \includestandalone{seed-knot}
  \caption{The knot $R(J,D)$.}
  \label{figure:R(J,D)}
\end{figure}

Since $tS-S^T$ presents the Alexander module of $R(J,D)$, a routine computation
shows that $R(J,D)$ has (integral) Alexander module
\begin{equation}
  \label{equation:integral-alexander-module}
  H_1(M(R(J,D));\Z[t^{\pm1}]) = \Z[t^{\pm1}]/\langle \lambda(t)\rangle
  \oplus \Z[t^{\pm1}]/\langle \lambda(t^{-1})\rangle,
\end{equation}
and the summands $\Z[t^{\pm1}]/\langle \lambda(t)\rangle$ and
$\Z[t^{\pm1}]/\langle \lambda(t^{-1})\rangle$ are equal to the subgroups
$\langle\alpha_J\rangle$ and $\langle\alpha_D\rangle$ generated by the loops
$\alpha_J$ and $\alpha_D$ shown in Figure~\ref{figure:R(J,D)}.  It follows that
$\Delta_{R(J,D)} = \lambda(t)\lambda(t^{-1})$.  Moreover, the Blanchfield
pairing
\[
  \Bl\colon \cA \times \cA
  \to \Q(t)/\Q[t^{\pm1}]
\]
on the Alexander module $\cA := H_1(M(R(J,D));\Q[t^{\pm1}])$ has exactly two
metabolizers, $\langle\alpha_J\rangle$ and~$\langle\alpha_D\rangle$.  Here, a
submodule $P\subset \cA$ is called a metabolizer if $P$ is equal to
$P^\perp:=\{x\in \cA \mid \Bl(P,x)=0\}$.  The above observations on $R(J,D)$
hold for any choice of $J$ and~$D$. 

As another ingredient of the construction of the promised knots $K_i$, we will
use a family of knots $\{J^i_k\}$ which was given in~\cite{Cha-Kim:2017-1}.  For
$k=0$, an explicit construction of such $\{J^i_0\}_i$ is given
in~\cite[Section~4]{Cha-Kim:2017-1}. We need that $\{J^i_0\}_i$ satisfies the
following conditions \ref{item:negativity-of-J_0},
\ref{item:signature-of-J_0-large-enough}
and~\ref{item:independence-of-signature-of-J_0} for some sequence of increasing
primes $\{p_i\}$.   For a knot $J$ and an integer $p$, let
$\sigma_J(\omega)\in \Z$ be the Levine-Tristram signature function of $J$ at
$\omega\in S^1\subset \C$, and let
\begin{equation}
  \label{equation:rho-J}
  \rho(J,\Z_d) = \frac1d \sum_{k=0}^{d-1}
  \sigma_J(e^{2\pi k\sqrt{-1}/d}).
\end{equation}

\begin{enumerate}[label=({J\arabic*})]
  \item\label{item:negativity-of-J_0} For each $i$, $J^i_0$ is
    $0$-negative.
  \item\label{item:signature-of-J_0-large-enough} For each $i$,
    $|\rho(J^i_0,\Z_{p_i})| > 69\,713\,280\cdot (6n+8m+86)$.
  \item\label{item:independence-of-signature-of-J_0} For $i<j$,
    $\rho(J^j_0,\Z_{p_i})=0$.
\end{enumerate}

For $k=0,\ldots,n-2$, $J^i_{k+1}$ is defined inductively by $J^i_{k+1} =
P_k(\eta_k,J^i_k)$, where $P_k(\eta_k,J^i_k)$ is the satellite knot shown in the
right of Figure~\ref{figure:stevedore-pattern}.  The left of
Figure~\ref{figure:stevedore-pattern} shows the pattern $P_k$ in the exterior of
an unknotted circle $\eta_k$, which is a standard solid torus.  The companion is
the knot~$J^i_k$.

\begin{figure}[H]
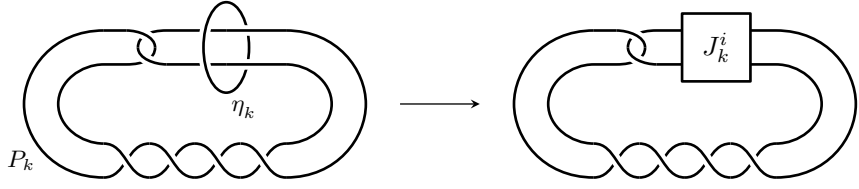

  \includestandalone{stevedore-pattern}
  \caption{The stevedore's pattern $(P_k,\eta_k)$ and the satellite knot
  $P_k(\eta_k,J^i_k)$.}
  \label{figure:stevedore-pattern}
\end{figure}

Now, let $K_i = R(J_{n-1}^i,D)$, where $D=\Wh^+(T)$ is the positive Whitehead
double of the right handed trefoil~$T$.  Since $D$ is topologically
slice~\cite{Freedman:1984-1}, $K_i$ is topologically concordant to
$R(J_{n-1}^i,U)$, where $U$ is the trivial knot.  Since $R(J,U)$ is (smoothly)
slice for any $J$, it follows that $K_i$ is topologically slice.
By~\cite[Section~2.2, Lemma~2.3]{Cha-Kim:2017-1},
property~\ref{item:negativity-of-J_0} implies each $K_i$ is $n$-negative, and
$k$-positive for all~$k\ge 0$.  Since $\Delta_{K_i} =
\lambda(t)\lambda(t^{-1})$, it follows that the class $[K_i]$ lies in
$\gr_n(\cT)_\lambda = (\cT_n\cap \cT_\lambda) / (\cT_{n+1}\cap \cT_\lambda)$.
Therefore, to prove Theorem~\ref{theorem:main-independence-bipolar}, it suffices
to show the following statement.

\begin{theorem}
  \label{theorem:linear-combination-non-negativity}
  Suppose $K=\bigl(\#_{i=1}^r \, a_i K_i\bigr) \# L$ \textup{(}$a_i\in
  \Z$\textup{)} is a linear combination of the knots~$K_i$ and a knot $L$ with
  $\Delta_L(t)$ relatively prime to~$\lambda(t)$.  If $a_i\ne 0$ for some $i$,
  then $K$ is not $(n+1)$-bipolar.
\end{theorem}

For the special case that $L$ is a trivial knot and $m=1$, the conclusion of
Theorem~\ref{theorem:linear-combination-non-negativity} was shown
in~\cite{Cha-Kim:2017-1}.  The general case of
Theorem~\ref{theorem:linear-combination-non-negativity} requires substantially
more sophisticated ideas and methods, which will be discussed in
Sections~\ref{section:L2-signatures} and~\ref{section:d-invariant-estimate}.

\subsection{Construction of a negaton}
\label{subsection:construction-X^-}

To prove Theorem~\ref{theorem:linear-combination-non-negativity} by
contradiction, suppose the knot $K$ is $(n+1)$-bipolar.  Recall that $K_i$ is
$n$-bipolar.  The following observation will be useful.  The knot $L$ in the
statement of Theorem~\ref{theorem:linear-combination-non-negativity} is
automatically $n$-bipolar, since $L$ is concordant to the sum of $K$, $-a_1K_1$,
\dots, $-a_{r-1}K_{r-1}$, and $-a_rK_r$ and each of $K$ and $K_i$ are
$n$-bipolar.

We may assume $a_i\ne 0$ for all $i$, by removing $K_i$ when $a_i=0$.  In
addition, by taking $-K$ instead of $K$, we may assume that $a_1>0$.  Under this
assumption, we will prove that $K$ is not $(n+1)$-negative.  To derive a
contradiction, suppose $K$ is $(n+1)$-negative.  As done
in~\cite[Section~2.3]{Cha-Kim:2017-1}, we will construct a certain $n$-negaton
for the first summand~$K_1$, which we will call~$X^-$ below.  Since the
generality of Theorem~\ref{theorem:linear-combination-non-negativity} does not
cause significant issues in this construction, we will closely
follow~\cite{Cha-Kim:2017-1}, with minor additional changes related to~$L$.

Use the $n$-bipolarity of $L$ to choose an $n$-negaton, say $Z^-_L$, bounded
by~$-M(L)=M(-L)$.  Use the $n$-bipolarity of $K_i$, choose an $n$-negaton
$Z^-_i$ bounded by $-M(K_i)=M(-K_i)$ for each $i$ for which $a_i>0$, and choose
$n$-negaton $Z^-_i$ bounded by $M(K_i)$ for each $i$ for which $a_i<0$. Let
$V^-$ be an $(n+1)$-negaton bounded by~$M(K)$.  
There is a standard cobordism $C$ bounded by the union of $\partial_-C := -M(K)$
and $\partial_+C := \bigl(\bigsqcup_i a_i M(K_i)\bigr) \sqcup M(L)$, which is
associated with the connected sum expression $K=\bigl(\#_{i=1}^r \, a_i K_i\bigr)
\# L$; $C$ is obtained by attaching, to $\bigl(\bigsqcup_{i=1}^r a_i M(K_i)
\times I\bigr) \sqcup M(L)\times I$, $N$ 1-handles that connects the component
and then attaching $N$ 2-handles which make meridians of the involved $N+1$
knots parallel, where $N=\sum_i |a_i|$.  A detailed description of $C$ can be
found, for instance, from~\cite[p.~113]{Cochran-Orr-Teichner:2002-1}. Define
\begin{equation}
  \label{equation:definition-of-negaton}
  X^- := V^- \cupover{\partial_-C} C \cupover{\partial_+C}
  \biggl((a_1-1)Z^-_1
  \sqcup \Bigl({\textstyle\bigsqcup\limits_{i>1} |a_i| Z^-_i} \Bigr)
  \sqcup Z^-_L \biggr).
\end{equation}
Figure~\ref{figure:negaton-construction} depicts the construction of~$X^-$. By
(the argument of)~\cite[Lemma~2.4]{Cha-Kim:2017-1}, $X^-$ is an $n$-negaton
bounded by~$M(K_1)$.

\begin{figure}[H]
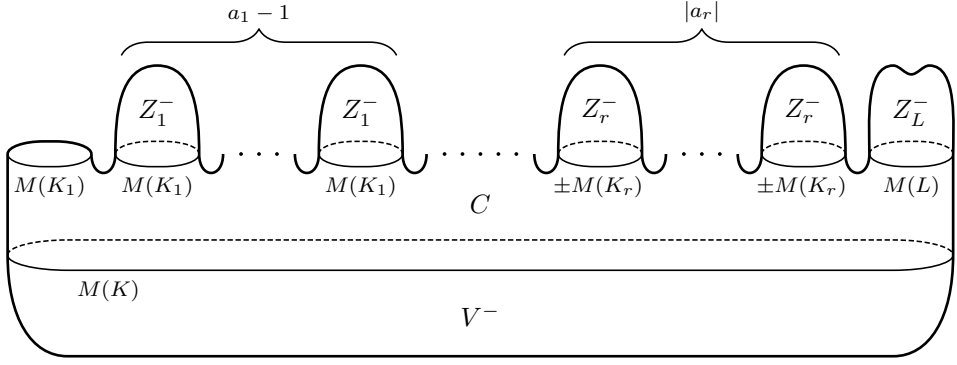

  \includestandalone{negaton-construction}
  \caption{The construction of~$X^-$.  The sign of $M(K_i)$ equals that
  of~$a_i$.}
  \label{figure:negaton-construction}
\end{figure}

Now, let
\begin{equation}
  \label{equation:definition-of-P}
  P := \Ker \{ H_1(M(K_1);\Q[t^{\pm1}]) \rightarrow H_1(X^-;\Q[t^{\pm1}]) \}.
\end{equation}
Since $X^-$ is an $n$-negaton with $n\ge 2$, $P$ is a metabolizer of the
Blanchfield pairing on $H_1(M(K_1);\Q[t^{\pm1}])$,
by~\cite[Theorem~5.8]{Cochran-Harvey-Horn:2012-1}.  (See also the statement
of~\cite[Lemma~2.5]{Cha-Kim:2017-1}.)  Since $K_1=R(J^1_{n-1},D)$ has exactly
two metabolizers $\langle\alpha_J\rangle$ and $\langle\alpha_D\rangle$, we have
the following two cases: $P=\langle\alpha_D\rangle$ or
$P=\langle\alpha_J\rangle$.  By deriving a contradiction for each case, the
proof of Theorem~\ref{theorem:linear-combination-non-negativity} will be
completed.

\section{\texorpdfstring{$L^2$}{L2}-signatures and localized mixed-type commutator series}
\label{section:L2-signatures}

In this section, we continue the proof of
Theorem~\ref{theorem:linear-combination-non-negativity}, for the case
$P=\langle\alpha_D\rangle$. 

Recall that we constructed an $n$-negaton $X^-$ in
\eqref{equation:definition-of-negaton} using the negatons $V^-$, $Z_L^-$
and~$Z_i^-$. To obtain a hyperbolic intersection form, take the connected sum of
the $(n+1)$-negaton $V^-$ and $b_2(V^-)$ copies of~$\C P^2$, and call the
result~$V^0$. Indeed, by~\cite[Proposition~5.5]{Cochran-Harvey-Horn:2012-1},
$V^0$ is a special type of 4-manifold called an \emph{integral $(n+1)$-solution}
in~\cite[Definition~3.1]{Cha:2010-1}, which particularly has a metabolic
intersection form over twisted coefficients.  We do not describe its definition
since we do not use it directly, but we will state and use its properties later.
For $Z_L^-$ and $Z_i^-$, define integral $n$-solutions $Z_L^0$ and $Z_i^0$ by
taking connected sum with copies of $\C P^2$ in the same way.

Repeat the construction of $X^-$, but now use $V^0$, $Z_L^0$ and $Z_i^0$ in
place of $V^-$, $Z_L^-$ and $Z_i^-$, to obtain a 4-manifold~$X^0$.  By the proof
of~\cite[Lemma~2.4]{Cha-Kim:2017-1}, $b_2(X^-)$ is equal to the sum of
$b_2(V^-)$, $b_2(Z_L^-)$, $(a_1-1)b_2(Z_1^-)$ and $|a_i| b_2(Z_i^-)$ ($i>1$).
Thus $X^0 = X^- \# (b_2(X^-) \C P^2)$.  Since $X^-$ is an $n$-negaton, it
follows that $X^0$ is an integral $n$-solution, again
by~\cite[Proposition~5.5]{Cochran-Harvey-Horn:2012-1}.

We will attach additional pieces to $X^0$, to obtain a sequence of 4-manifolds,
essentially following a technique first appeared
in~\cite{Cochran-Harvey-Leidy:2009-1}; see
also~\cites{Cochran-Harvey-Leidy:2009-2, Cochran-Harvey-Leidy:2009-3,
Cha:2010-1, Cha-Kim:2017-1}. The notation used below is close
to~\cites{Cha:2010-1,Cha-Kim:2017-1}. Consider the satellite construction
$J^1_{k+1} = P_k(\eta_k, J^1_k)$.  Due to~\cite{Cochran-Harvey-Leidy:2009-1},
there is a standard cobordism, which we denote by~$E_k$, from $M(J^1_{k+1})$ to
$M(J^1_k)\sqcup M(P_k)$ for $k=0$, \ldots,~$n-2$.  In
Section~\ref{subsection:localized-mixed-comm-series}, we will use an alternative
description given in Figure~\ref{figure:standard-cobordism}, which illustrates
that $E_k$ is obtained from $M(J^1_{k+1})\times [0,1]$ by attaching a 2-handle
and a 3-handle: start with $M(J^1_{k+1}) = M(P_k(\eta_k, J^1_k))$, attach a
2-handle along a zero-framed longitude of $J^1_k$ to obtain the second diagram,
and apply handle slide to obtain the last diagram, which is $M(J^1_k) \#
M(P_k)$.  Attach a 3-handle to obtain $M(J^1_k) \sqcup M(P_k)$.

\begin{figure}[H]
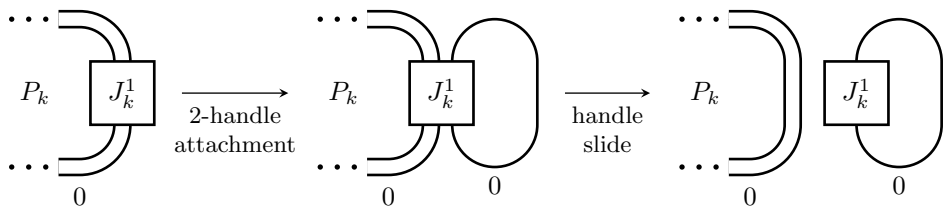

  \includestandalone{standard-cobordism}
  \caption{A handlebody description of the standard cobordism~$E_k$.}
  \label{figure:standard-cobordism}
\end{figure}

View $K_1=R(J^1_{n-1},D)$ as the satellite knot $R(U,D)(\alpha_J,J^1_{n-1})$,
and apply the same construction, to obtain a standard cobordism $E_{n-1}$ from
$M(K_1)$ to $M(J^1_{n-1})\sqcup M(R(U,D))$. Now, define 4-manifolds
$X_n$,~$X_{n-1}$, \dots,~$X_0=X$ as follows:
\begin{align*}
  X_n & := X^0
  \\
  X_{n-1} & := X_n \cupover{M(K_1)} E_{n-1} = X^0 \cupover{M(K_1)} E_{n-1}
  \\
  X_{n-2} & := X_{n-1} \cupover{M(J^1_{n-1})} E_{n-2} 
  = X^0 \cupover{M(K_1)} E_{n-1} \cupover{M(J^1_{n-1})} E_{n-2}
  \\
  &\hphantom{=}\vdots
  \\
  X=X_0 &:= X_1 \cupover{M(J^1_{1})} E_{0}
  = X^0 \cupover{M(K_1)} E_{n-1} \cupover{M(J^1_{n-1})} E_{n-2}
    \cupover{M(J^1_{n-2})} \cdots \cupover{M(J^1_{1})} E_{0}.
\end{align*}
See the schematic diagram in Figure~\ref{figure:higher-cobordism-diagram}.

\begin{figure}[ht]
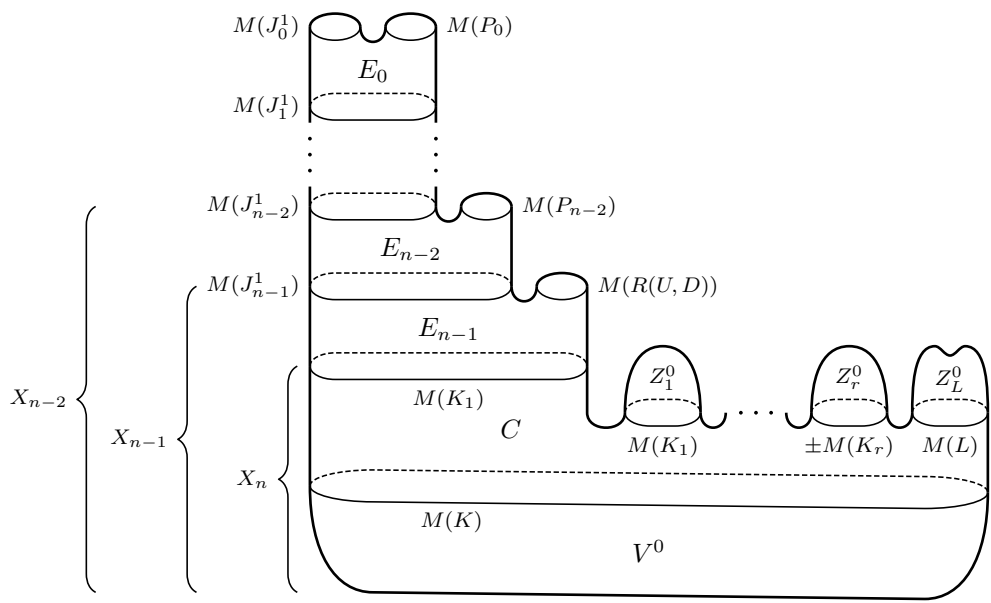

  \includestandalone{higher-cobordism-diagram}
  \caption{A schematic diagram of the 4-manifold~$X$.}
  \label{figure:higher-cobordism-diagram}
\end{figure}

We have that $H_1(X^0)=H_1(X^-)=\Z=\langle t \rangle$ where the generator $t$ is
represented by the meridian $\mu_K$ of $K$, since $X^-$ is a negaton bounded
by~$M(K)$.  See Definition~\ref{equation:definition-of-negaton}(1).  By a
Mayer-Vietoris argument using this, it follows that $H_1(X_k)=\Z$ for all $k=0$,
\dots,~$n$.

\subsection{A localized mixed-type commutator series associated
  with~\texorpdfstring{$X$}{X}}
\label{subsection:localized-mixed-comm-series}

Recall, from Section~\ref{subsection:construction-K_i}, that $p_1$ denotes the
first prime used in properties~\ref{item:signature-of-J_0-large-enough}
and~\ref{item:independence-of-signature-of-J_0}, and $\lambda(t)=(m+1)t-m$.
Let
\[
  \Sigma := \{f(t)\in \Q[t^{\pm1}]  \mid f(1)\ne 0, \,
  \gcd(f(t),\lambda(t)\lambda(t^{-1}))=1 \}.
\]
Obviously $\Sigma$ is a multiplicative subset.  Let $\Q[t^\pm]\Sigma^{-1}$ be
the localization.

\begin{definition}
  \label{definition:localized-commutator-series}
  Let $G$ be a group endowed with a homomorphism $G\to \pi_1(X)$ which induces
  an epimorphism $H_1(G) \twoheadrightarrow H_1(X)$.  For $i=0,1,\ldots,n+1$, define subgroups
  $\cP^i G$ of $G$ inductively as follows.  Let $\cP^0 G := G$, and $\cP^1 G$ be
  the kernel of the composition
  \[
     G\to \pi_1(X) \longtwoheadrightarrow H_1(X)=\Z=\langle t\rangle.
  \]

  Let $\cP^2G$ be the kernel of the composition
  \begin{align*}
    \cP^1 G
    & \longtwoheadrightarrow \frac{\cP^1 G}{[\cP^1 G,\cP^1 G]}
    \to \frac{\cP^1 G}{[\cP^1 G,\cP^1 G]} \otimesover{\Z}\Q
    = H_1\bigl(G;\Q[t^{\pm1}]\bigr) \to H_1\bigl(G;\Q[t^{\pm1}]\Sigma^{-1}\bigr)
    \\
    & 
    \to H_1\bigl(X;\Q[t^{\pm1}]\Sigma^{-1}\bigr)
    \longtwoheadrightarrow H_1\bigl(X;\Q[t^{\pm1}]\Sigma^{-1}\bigr) /
    \Im H_1\bigl(Z_L^0;\Q[t^{\pm1}]\Sigma^{-1}\bigr).
  \end{align*}
  Here, $\Im H_1(Z_L^0;\Q[t^{\pm1}]\Sigma^{-1})$ is the image of
  $H_1(Z_L^0;\Q[t^{\pm1}]\Sigma^{-1}) \to H_1(X;\Q[t^{\pm1}]\Sigma^{-1})$
  induced by the inclusion.
  
  For $i=2$, \dots,~$n-1$, let
  \[
    \cP^{i+1}G := \Ker \Bigl\{ \cP^i G
      \twoheadrightarrow \frac{\cP^i G}{[\cP^i G,\cP^i G]}
      \rightarrow \frac{\cP^i G}{[\cP^i G,\cP^i G]} \otimesover{\Z} \Q
      = H_1\bigl(G;\Q[G/\cP^i G]\bigr)
    \Bigr\}.
  \]
  Finally, define
  \[
    \cP^{n+1}G := \Ker \Bigl\{ \cP^n G
      \twoheadrightarrow \frac{\cP^n G}{[\cP^n G,\cP^n G]}
      \twoheadrightarrow \frac{\cP^n G}{[\cP^n G,\cP^n G]} \otimesover{\Z} \Z_{p_1}
      = H_1\bigl(G;\Z_{p_1}[G/\cP^n G]\bigr)
    \Bigr\}.
  \]
\end{definition}

It is straightforward to verify inductively that $\cP^k G$ is a normal subgroup
of~$G$ and the standard $k$th derived subgroup $G^{(k)}$ lies in $\cP^k G$ for
all~$k$.

We remark that the commutator series $\{\cP^k G\}_k$ in
Definition~\ref{definition:localized-commutator-series} is ``localized at
polynomials'' in the sense of~\cite[Sections~3
and~4]{Cochran-Harvey-Leidy:2009-2}, and is of ``mixed-coefficient'' type in the
sense that we use both $\Z_{p_1}$ and $\Q$ (see~\cite[Section~4.1]{Cha:2010-1}).

In particular, define the subgroups $\cP^i\pi_1(X_k)$, by applying
Definition~\ref{definition:localized-commutator-series} to the case
$G=\pi_1(X_k) \to \pi_1(X)$.  The following properties, which we will state as
Assertions~\ref{assertion:nontriviality-mu_J}
and~\ref{assertion:triviality-Z_0^L}, are essential for our purpose.  Let
$\mu_k\subset M(J^1_k)$ be the meridian of~$J^1_k$. By the construction of
$X_k$, $M(J^1_k)$ is a component of~$\partial X_k$ (see
Figure~\ref{figure:higher-cobordism-diagram}), and thus $\mu_k$ represents an
element in~$\pi_1(X_k)$ for $k\le n-1$.  For brevity, let $J^1_n := K_1$, so
that the previous sentence holds for $k=n$ as well.  Also, let
$(P_{n-1},\eta_{n-1}):=(R(U,D),\alpha_D)$, so that $J^1_k =
P_{k-1}(\eta_{k-1},J^1_{k-1})$ holds for $k=n$ as well.

\begin{assertion}
  \label{assertion:nontriviality-mu_J}
  The class of $\mu_k$ lies in $\pi_1(X_k)^{(n-k)}\subset\cP^{n-k}\pi_1(X_k)$
  and is nontrivial in $\cP^{n-k}\pi_1(X_k)/\cP^{n-k+1}\pi_1(X_k)$ for all
  $k=0$, $1$, \dots,~$n$.  In particular, the class $\mu_0$ is nontrivial
  in~$\cP^n\pi_1(X)/\cP^{n+1}\pi_1(X)$.
\end{assertion}

\begin{assertion}
  \label{assertion:triviality-Z_0^L}
  For $i= 1$, \dots,~$n$, the inclusion-induced map sends $\pi_1(Z_L^0)^{(i)}$
  to the subgroup $\cP^{i+1}\pi_1(X) \subset \pi_1(X)$.   In particular,
  $\pi_1(Z_L^0)^{(n)}$ maps to $\cP^{n+1}\pi_1(X)$.
\end{assertion}

\begin{remark}
  \phantomsection\leavevmode\Nopagebreak
  \begin{enumerate}
    \item  Analogues of Assertion~\ref{assertion:nontriviality-mu_J} for similar
    situations were established in earlier papers, for instance
    in~\cites{Cochran-Harvey-Leidy:2009-1, Cochran-Harvey-Leidy:2009-2,
    Cha:2010-1, Cha:2012-1, Cha-Kim:2017-1}. We will give a proof for our case,
    since our series $\cP^i$ is different (notably at $i=2$) from those in the
    literature. Assertion~\ref{assertion:triviality-Z_0^L} and its application
    are new, to the knowledge of the author.  
    \item
    Note that both $\mu_0$ and $\pi_1(Z_L^0)^{(n)}$ map to the $n$th subgroup
    $\cP^n \pi_1(X)$. Due to Assertions~\ref{assertion:nontriviality-mu_J}
    and~\ref{assertion:triviality-Z_0^L}, they have opposite nature in the next
    stage: $\pi_1(Z_L^0)^{(n)}$ lies in $\cP^{n+1} \pi_1(X)$, while $\mu_0$ does
    not.  This will be crucial in separating the contribution of the unknown
    knot $L$ from that of $K_1$ in the linear combination $K=\bigl(\#_{i=1}^r
    a_i K_i\bigr) \# L$. See Section~\ref{subsection:obstruction-from-rho},
    particularly the Cheeger-Gromov $\rho$-invariants
    in~\eqref{equation:computation-of-rho-for-J0}
    and~\eqref{equation:L2-signature-Z_L^0}.
  \end{enumerate}

\end{remark}

In the proof of Assertion~\ref{assertion:nontriviality-mu_J}, the following fact
will be useful.  Let $\lambda_{k-1}$ be the zero-linking longitude of
$J^1_{k-1}$, which lies in
\begin{equation*}
  E(J^1_{k-1}) \subset E(J^1_{k-1})\cup E(P_{k-1}\sqcup \eta_{k-1}) = E(J^1_k)
  \subset M(J^1_k) \subset \partial X_k.
\end{equation*}

\begin{assertion}
  \label{assertion:series-for-X_k-and_X_k+1}
  The inclusion $X_k \subset X_{k-1}$ induces an isomorphism
  $\cP^i\pi_1(X_k)/\langle \lambda_{k-1} \rangle \cong \cP^i\pi_1(X_{k-1})$
  for all $i\le n-k+2$.  Consequently, we have
  \begin{equation*}
    \cP^{n-k+1}\pi_1(X_k)/\cP^{n-k+2}\pi_1(X_k) \cong
    \cP^{n-k+1}\pi_1(X_{k-1})/\cP^{n-k+1}\pi_1(X_{k-1}).
  \end{equation*}
\end{assertion}

\begin{proof}[Proof of Assertion~\ref{assertion:series-for-X_k-and_X_k+1}]
  
  Since $E_{k-1}$ is obtained by attaching a 2-handle to $M(J^1_k)\times [0,1]$
  along $\lambda_{k-1}$ and then attaching a 3-handle (see
  Figure~\ref{figure:standard-cobordism}), $X_{k-1}=X_k\cup E_{k-1}$ is obtained
  from $X_k$ by the same handle attachments.  It follows that $\pi_1(X_{k-1})
  \cong \pi_1(X_k)/\langle \lambda_{k-1} \rangle$.  This shows the assertion for
  $i=0$.  Now, to proceed by induction, suppose $i\ge 1$ and suppose the
  assertion holds for $i-1$.  If $i=1$, we have the following commutative
  diagram with exact rows:
  \[
    \begin{tikzcd}
      \cP^i \pi_1(X_k) \ar[d,dotted,two heads] \ar[r,hook]
      & \cP^{i-1} \pi_1(X_k) \ar[d,two heads] \ar[r]
      & H_1(X) \ar[d,equal]
      \\
      \cP^i \pi_1(X_{k-1}) \ar[r,hook]
      & \cP^{i-1} \pi_1(X_{k-1}) \ar[r]
      & H_1(X)
    \end{tikzcd}
  \]
  Since the two rightmost horizontal arrows are to the \emph{same} target
  $H_1(X)$, there is an induced epimorphism $\cP^i\pi_1(X_k) \twoheadrightarrow
  \cP^i\pi_1(X_{k-1})$ and its kernel is equal to that of the epimorphism
  $\cP^{i-1}\pi_1(X_k) \twoheadrightarrow \cP^{i-1}\pi_1(X_{k-1})$.  So the the
  assertion holds for $i=1$.  For $i=2$, replace $H_1(X)$ in the above diagram
  by the quotient
  \[
    H_1\bigl(X;\Q[t^{\pm1}]\Sigma^{-1}\bigr) /
    \Im H_1\bigl(Z_L^0;\Q[t^{\pm1}]\Sigma^{-1}\bigr)
  \]
  and apply the same argument.
  
  For $i\ge 3$, we have
  \[
    \begin{tikzcd}
      \cP^i \pi_1(X_k) \ar[r,hook]
      & \cP^{i-1} \pi_1(X_k) \ar[d,two heads] \ar[r]
      & \dfrac{\cP^{i-1}\pi_1(X_k)}{(\cP^{i-1}\pi_1(X_k))^{(1)}} \otimes R
      \ar[d]
      \\[-1ex]
      \cP^i \pi_1(X_{k-1}) \ar[r,hook]
      & \cP^{i-1}\pi_1(X_{k-1}) \ar[r]
      & \dfrac{\cP^{i-1}\pi_1(X_{k-1})}{(\cP^{i-1}\pi_1(X_{k-1}))^{(1)}}\otimes R
    \end{tikzcd}
  \]
  where $R=\Q$ or $\Z_{p_1}$, depending on~$i$.  We claim that the rightmost
  vertical arrow is an isomorphism.  From the claim, it follows that the
  assertion holds for $i$, once again by the argument used above.  To show the
  claim, let $\gamma_{k-1}$ be the meridian of $J^1_{k-1}$ in the exterior
  $E(J^1_{k-1}) \subset M(J^1_k) \subset \partial X_k$.  Note that
  $\gamma_{k-1}$ is different from the meridian $\mu_{k-1}\subset
  M(J^1_{k-1})\subset \partial X_{k-1}$ used in the statement of
  Assertion~\ref{assertion:nontriviality-mu_J}, but $\gamma_{k-1}$ and
  $\mu_{k-1}$ are isotopic in the cobordism~$E_{k-1}$. The meridian
  $\gamma_{k-1}$ is identified with the curve $\eta_{k-1}$ which lies in the
  commutator subgroup $\pi_1(M(J^1_k))^{(1)}$.  Since $\pi_1(M(J^1_k))$ is
  normally generated by the meridian $\mu_k$, $\gamma_{k-1}$ lies in $\langle
  \mu_k\rangle^{(1)} = \langle \gamma_k\rangle^{(1)}$ in $\pi_1(X_k)$. By
  induction, it follows that $\gamma_{k-1}$ lies in
  $\langle\gamma_n\rangle^{(n-k+1)}$.  Therefore, the image of
  $\pi_1(E(J^1_{k-1}))$ lies in $\pi_1(X_k)^{(n-k+1)}$.  Since the longitude
  $\lambda_{k-1}$ lies in $\pi_1(E(J^1_{k-1}))^{(1)}$, it follows that
  \begin{equation}
    \label{equation:lambda-in-derived-subgroup}
    \lambda_{k-1} \in \pi_1(X_k)^{(n-k+2)}.
  \end{equation}
  Since $i\le n-k+2$, \eqref{equation:lambda-in-derived-subgroup}~implies that
  $\lambda_{k-1} \in \pi_1(X_k)^{(i)} \subset (\cP^{i-1}\pi_1(X_k))^{(1)}$.
  Also, by the induction hypothesis, $\cP^{i-1}\pi_1(X_k) \twoheadrightarrow
  \cP^{i-1}\pi_1(X_{k-1})$ is an epimorphism with kernel
  $\langle\lambda_{k-1}\rangle$.  It follows that the rightmost vertical arrow
  in the above diagram is an isomorphism. This completes the proof of the
  assertion.
\end{proof}

\begin{proof}[Proof of Assertion~\ref{assertion:nontriviality-mu_J}]

  In the proof of Assertion~\ref{assertion:series-for-X_k-and_X_k+1}, we already
  showed that $\mu_k$ lies in~$\pi_1(X_k)^{(n-k)}$.  This is the first part of
  Assertion~\ref{assertion:nontriviality-mu_J}.

  It remains to show that $\mu_k$ is nontrivial in
  $\cP^{n-k}\pi_1(X_k)/\cP^{n-k+1}\pi_1(X_k)$.  We will use reverse induction
  for $k=n$, $n-1$,~\dots,~$0$.  For $k=n$, $\pi_1(X_n)/\cP^1\pi_1(X_n) =
  H_1(X)$ by definition, and the meridian $\mu_n$ of $J^1_n = K_1$ is a
  generator of $H_1(X)$.  So the assertion holds.
  
  For the case $k=n-1$, let $\Lambda=\Q[t^{\pm1}]$ for brevity, and consider the
  following commutative diagram:
  \[
    \begin{tikzcd}[column sep=small]
      H_1(M(K_1);\Lambda) \ar[r] \ar[rd]
      &[-2ex] H_1(X_n;\Lambda\Sigma^{-1}) \ar[r,"\cong"]
      & H_1(X;\Lambda\Sigma^{-1}) \ar[r]
      &[-2ex] \dfrac{H_1(X;\Lambda\Sigma^{-1})}{\Im H_1(Z_L^0;\Lambda\Sigma^{-1})}
      \\
      & H_1(X_n\sm Z_L^0;\Lambda\Sigma^{-1}) \ar[u] \ar[r,"\cong"']
      & H_1(X\sm Z_L^0;\Lambda\Sigma^{-1}) \ar[u] \ar[ru,"\cong"']
    \end{tikzcd}
  \]
  Here, $H_1(X_n;\Lambda\Sigma^{-1}) \to H_1(X;\Lambda\Sigma^{-1})$ is an
  isomorphism, since $\pi_1(X)=\pi_1(X_0)$ is isomorphic to $\pi_1(X_n)/\langle
  \lambda_0,\ldots,\lambda_{n-1}\rangle$ by
  Assertion~\ref{assertion:series-for-X_k-and_X_k+1}, and the longitudes
  $\lambda_0,\ldots,\lambda_{n-1}$ lie in $\pi_1(X_n)^{(2)}$
  by~\eqref{equation:lambda-in-derived-subgroup}. The same argument shows that
  the bottom horizontal arrow is an isomorphism.  (Alternatively, one may use
  Mayer-Vietoris arguments to show that they are isomorphisms.) Also, consider
  the Mayer-Vietoris sequence for $X = \overline{X\sm Z_L^0} \cup_{M(L)} Z_L^0$:
  \[
    H_1(M(L);\Lambda\Sigma^{-1}) \to
    H_1(X\sm Z_L^0;\Lambda\Sigma^{-1}) \oplus H_1(Z_L^0;\Lambda\Sigma^{-1})
    \to H_1(X;\Lambda\Sigma^{-1}) \to 0.
  \]
  We have $H_1(M(L);\Lambda\Sigma^{-1}) = 0$ by the hypothesis that
  $\Delta_L(t)$ is relatively prime to the polynomial $\lambda(t)$.  It follows
  that the diagonal arrow on the right hand side of the above diagram is an
  isomorphism.

  By our hypothesis, the kernel $P$ of
  \[
    \langle \alpha_J\rangle \oplus \langle\alpha_D\rangle =
    H_1(M(K_1);\Lambda) \to H_1(X_n;\Lambda)
  \]
  is equal to the summand $\langle\alpha_D\rangle$.  So the other summand
  $\langle \alpha_J\rangle$ injects into $H_1(X_n;\Lambda)$.  Since $\langle
  \alpha_J\rangle \cong \Lambda/\langle\lambda\rangle$ is not annihilated by
  $\Sigma$, this implies that $\langle \alpha_J\rangle$ injects into
  $H_1(X_n;\Lambda\Sigma^{-1}) = H_1(X_n;\Lambda)\Sigma^{-1}$.  By the above
  diagram, it follows that $\langle \alpha_J\rangle$ injects into $H_1(X\sm
  Z_L^0;\Lambda\Sigma^{-1})$.  Thus $\alpha_J$ is nontrivial in
  $H_1(X;\Lambda\Sigma^{-1})/\Im H_1(Z_L^0;\Lambda\Sigma^{-1})$. Therefore, by
  the definition of $\cP^2\pi_1(X_n)$, $\alpha_J$ is nontrivial in the quotient
  \[
    \cP^1\pi_1(X_n)/\cP^2\pi_1(X_n) \subset
    H_1(X;\Lambda\Sigma^{-1})/\Im H_1(Z_L^0;\Lambda\Sigma^{-1}).
  \] 
  By Assertion~\ref{assertion:series-for-X_k-and_X_k+1},
  $\cP^1\pi_1(X_n)/\cP^2\pi_1(X_n) \cong
  \cP^1\pi_1(X_{n-1})/\cP^2\pi_1(X_{n-1})$.  Also, $\alpha_J$ is isotopic to the
  meridian $\mu_{n-1}$ in~$X_{n-1}$.  It follows that $\mu_{n-1}$ is nontrivial
  in the quotient $\cP^1\pi_1(X_{n-1})/\cP^2\pi_1(X_{n-1})$.  This is exactly
  the promised conclusion for $k=n-1$.

  Now, suppose $0\le k \le n-2$. The induction hypothesis is that $\mu_{k+1}$ is
  nontrivial in the quotient
  $\cP^{n-k-1}\pi_1(X_{k+1})/\cP^{n-k}\pi_1(X_{k+1})$.  To show that $\mu_k$ is
  nontrivial in $\cP^{n-k}\pi_1(X_k)/\cP^{n-k+1}\pi_1(X_k)$, we use an argument
  which is essentially the same as~\cite[Proof of Theorem~4.14]{Cha:2010-1},
  which was influenced by~\cite{Cochran-Harvey-Leidy:2009-1}.  Let $R=\Q$ if
  $k\ge 1$, and $R=\Z_{p_1}$ if $k=0$. Let
  \[
    B \colon
    H_1(M(J^1_{k+1});R[t^{\pm1}])\times H_1(M(J^1_{k+1});R[t^{\pm1}]) \to
    R(t)/R[t^{\pm1}]
  \]
  be the classical Blanchfield pairing of $J^1_{k+1}$ over $R$-coefficients. For
  brevity, let $G=\pi_1(X_{k+1})/\cP^{n-k}\pi_1(X_{k+1})$.  Consider the
  non-commutative Alexander module $\cA:= H_1(M(J^1_{k+1}); RG)$.  The
  non-commutative Blanchfield pairing $\cB \colon \cA\times \cA \to \cK/RG$ is
  defined following~\cite[Theorem~2.13]{Cochran-Orr-Teichner:1999-1}, where
  $\cK$ is the skew-quotient field of~$RG$. (For $R=\Z_{p_1}$, see
  also~\cite[Section~5]{Cha:2010-1}.)  In our case, since $2\le n-k\le n$, from
  Definition~\ref{definition:localized-commutator-series} it follows that
  $G=\pi_1(X_{k+1})/\cP^{n-k}\pi_1(X_{k+1})$ is poly-torsion-free-abelian
  (PTFA), and consequently $RG$ is an Ore domain due
  to~\cite{Cochran-Orr-Teichner:1999-1}*{Proposition~2.5}
  and~\cite{Cha:2010-1}*{Lemma~5.2}.  We will use the following known facts.
  
  \begin{enumerate}
    \item
    The nontriviality of $\mu_{k+1}$ in
    $\cP^{n-k-1}\pi_1(X_{k+1})/\cP^{n-k}\pi_1(X_{k+1}) \subset G$ implies that
    $\cA \cong RG \otimesover{R[t^{\pm1}]} H_1(M(J^1_{k+1});R[t^{\pm1}])$, and
    that $\cB(1\otimes x, 1\otimes y) = 0$ if and only if $B(x,y)=0$.  This is
    due to~\cite[Theorem~4.7]{Leidy:2006-1}, \cite[Theorem~5.16]{Cha:2003-1} and
    \cite[Lemma~6.5, Theorem~6.6]{Cochran-Harvey-Leidy:2009-1}.  See
    also~\cite[Theorem~5.4]{Cha:2010-1}.

    \item
    The 4-manifold $X_{k+1}$ endowed with $\pi_1(X_{k+1}) \to G$ is a
    Blanchfield bordism in the sense of~\cite[Definition~4.11]{Cha:2012-1}, due
    to an argument in~\cite[p.~3270]{Cha:2012-1} which
    uses~\cite[Theorem~4.13]{Cha:2012-1}.  (The 4-manifold $W_{k+1}$
    in~\cite{Cha:2012-1} plays the role of our~$X_{k+1}$.)  The only property of
    a Blanchfield bordism we need is the following: for all $z$ in $\Ker\{\cA
    \to H_1(X_{k+1};RG)\}$, $\cB(z,z)=0$ by~\cite[Theorem~4.12]{Cha:2012-1}.
  \end{enumerate}

  Recall that $E(P_k\sqcup \eta_k) \cup E(J^1_k) = E(J^1_{k+1}) \subset
  M(J^1_{k+1})\subset \partial X_{k+1}$.  Denote a zero-linking longitude of
  $\eta_k$ in $E(P_k\sqcup \eta_k)\subset M(J^1_{k+1})$ by~$\eta_k$, abusing
  notation.  If $\eta_k$ is trivial in
  $\cP^{n-k}\pi_1(X_{k+1})/\cP^{n-k+1}\pi_1(X_{k+1})$, then by the definition of
  $\cP^{n-k+1}$, $\eta_k=1\otimes \eta_k$ lies in the kernel of $\cA \to
  H_1(X_{k+1};RG)$.  By~(2), from this it follows that $\cB(1\otimes
  \eta_k,1\otimes \eta_k) = 0$, and consequently by~(1), $B(\eta_k,\eta_k)=0$.
  Since $J^1_{k+1} = P_k(\eta_k,J^1_k)$ with $k\le n-2$, the Alexander module
  $H_1(M(J^1_{k+1});R[t^{\pm1}])$ is isomorphic to that of stevedore's knot
  $P_k$, which is a cyclic module generated by~$\eta_k$.  It contradicts the
  non-singularity of the classical Blanchfield pairing~$B$. This shows that
  $\eta_k$ is nontrivial in $\cP^{n-k}\pi_1(X_{k+1})/\cP^{n-k+1}\pi_1(X_{k+1})$,
  which is isomorphic to $\cP^{n-k}\pi_1(X_k)/\cP^{n-k+1}\pi_1(X_k)$ by
  Assertion~\ref{assertion:series-for-X_k-and_X_k+1}\@.  Since $\eta_k$ is
  identified with $\mu_k$, it follows that $\mu_k$ is nontrivial in the quotient
  $\cP^{n-k}\pi_1(X_k)/\cP^{n-k+1}\pi_1(X_k)$.  This completes the proof of
  Assertion~\ref{assertion:nontriviality-mu_J}\@.
\end{proof}

\begin{proof}[Proof of Assertion~\ref{assertion:triviality-Z_0^L}]
  
  Recall that Assertion~\ref{assertion:triviality-Z_0^L} says that
  $\pi_1(Z_0^L)^{(i)}$ maps to $\cP^{i+1}\pi_1(X)$ for $1\le i \le n$.  To
  show this for $i=1$, observe that the composition
  \[
    \begin{aligned}
      \frac{\pi_1(Z_L^0)^{(1)}}{\pi_1(Z_L^0)^{(2)}} = H_1(Z_L^0;\Z[t^{\pm1}])
      & \to H_1(X;\Z[t^{\pm1}])
      \to H_1(X;\Q[t^{\pm1}]\Sigma^{-1})
      \\
      & \to H_1(X;\Q[t^{\pm1}]\Sigma^{-1})/\Im H_1(Z_L^0;\Q[t^{\pm1}]\Sigma^{-1})
    \end{aligned}
  \]
  is obviously zero.  By definition of $\cP^2 \pi_1(X)$, it follows that
  $\pi_1(Z_L^0)^{(1)}$ maps to $\cP^2\pi_1(X)$.  Therefore, for all $1\le i\le
  n$, $\pi_1(Z_L^0)^{(i)} = (\pi_1(Z_L^0)^{(1)})^{(i-1)}$ maps to
  $\cP^2\pi_1(X)^{(i-1)}$, which is a subgroup of~$\cP^{i+1}\pi_1(X)$.
\end{proof}

\subsection{Obstruction from Cheeger-Gromov \texorpdfstring{$\rho$}{rho}-invariants}
\label{subsection:obstruction-from-rho}

Now, we will use the Cheeger-Gromov $\rho$-invariant to derive a contradiction.
We begin with some background.
For a connected closed 3-manifold $M$ and a homomorphism $\phi\colon \pi_1(M) \to \Gamma$ with $\Gamma$ arbitrary, the Cheeger-Gromov invariant $\rhot(M,\phi)\in \R$ is defined~\cite{Cheeger-Gromov:1985-1}.
The value of $\rhot(M,\phi)$ is preserved under composition with automorphisms of $\pi_1(M)$, so one can view $\rhot(M,\phi)$ as an invariant of $M$ equipped with (the homotopy class of) a map $\phi\colon M\to B\Gamma=K(\Gamma,1)$, instead of $\pi_1(M) \to \Gamma$.
Even when $M$ is not connected, $\rhot(M,\phi)$ is defined for $\phi\colon M\to B\Gamma$, with additivity under disjoint union:
$\rhot(M,\phi)=\sum_i\rhot(M_i,\phi|_{M_i})$ where $M=\bigcup_i M_i$ with components~$M_i$.

In this paper, we do not use the definition of $\rhot(M,\rhot)$ given in \cite{Cheeger-Gromov:1985-1}.
Instead, for our purpose, the following $L^2$-signature defect interpretation is useful.
If $M=\bigcup M_i$ bounds a 4-manifold $W$ and $\phi\colon M\to B\Gamma$ factors through $W$, then $\rhot(M,\phi) = \sum_i\rhot(M_i,\phi|_{M_i})$ is equal to the $L^2$-signature defect $\smash{\osigmat_\Gamma(W)}:=\smash{\sign^{(2)}_\Gamma(W)} -\sign(W)$, where $\sign(W)$ is the ordinary signature and $\smash{\sign^{(2)}_\Gamma(W)}$ is the $L^2$-signature of $W$ over the group~$\Gamma$.
We note that this approach can also be used to provide an alternative definition of $\rhot(M,\phi)$ for arbitrary~$(M,\phi)$.
As references, see, for instance, \cite{Chang-Weinberger:2003-1}, \cite[Section~2]{Cochran-Teichner:2003-1}, \cite[Section~3]{Harvey:2006-1}, \cite[Section~2.1]{Cha:2014-1}. 

For our case, let $\Gamma:= \pi_1(X) / \cP^{n+1} \pi_1(X)$.  For a connected
3-dimensional submanifold $M$ in $X$, denote the composition $\pi_1(M)\to
\pi_1(X)\to \Gamma$ by $\phi$, abusing notation.  

Then, by the $L^2$-signature defect interpretation for $\rhot(\partial X,\phi)$,
we have
\begin{multline}
  \label{equation:L2-signature-of-X}
  \rhot(M(J^1_0),\phi) + \sum_{k=0}^{n-2} \rhot(M(P_k),\phi) +
  \rhot(M(R(U,D)),\phi) = \osigmat_\Gamma(X)
  \\
  = \osigmat_\Gamma(V^0) + \osigmat_\Gamma(C)
   + \osigmat_\Gamma(Z^0_L) + \sum_{i,j} \osigmat_\Gamma(Z^0_{i,j})
   + \sum_{k=0}^{n-1} \osigmat_\Gamma(E_k),
\end{multline}
where the 4-manifolds $Z^0_{i,j}$ are copies of $Z^0_i$ used in the construction
of~$X$.  (See Figure~\ref{figure:higher-cobordism-diagram}.)  The second
equality is obtained by Novikov additivity of $L^2$-signatures (for instance,
see~\cite{Cochran-Orr-Teichner:1999-1}*{Lemma~5.9}).

Recall that $\mu_0$ is the meridian of $J^1_0$ in~$M(J^1_0)$.  By
Assertion~\ref{assertion:nontriviality-mu_J}, $\phi(\mu_0)$ is nontrivial
in~$\Gamma$.  Since $\phi(\mu_{J^1_0})$ lies in the subgroup $\cP^n
\pi_1(X)/\cP^{n+1} \pi_1(X)$ of $\Gamma$, which is a vector space over
$\Z_{p_1}$ by Definition~\ref{definition:localized-commutator-series}, it
follows that $\mu_{J^1_0}$ has order $p_1$.  So the image of $\pi_1(M(J^1_0))$
in $\Gamma$ under $\phi$ is isomorphic to~$\Z_{p_1}$.
By~\cite[Lemma~8.7]{Cha-Orr:2009-1}, this implies that
\begin{equation}
  \label{equation:computation-of-rho-for-J0}
  \rhot(M(J^1_0),\phi) = \rho(J^1_0,\Z_{p_1})
\end{equation}
where $\rho(J^1_0,\Z_{p_1})$ is defined by~\eqref{equation:rho-J}.

By the explicit universal bound for the Cheeger-Gromov invariants given in~\cite[Theorem~1.9]{Cha:2014-1}, we have
\begin{equation}
  \label{equation:universal-bound-computation-1}
  \big|\rhot\big(M(P_k),\phi\big)\big| \le 6\cdot 69\,713\,280 
\end{equation}
since stevedore's knot $P_k$ has 6 crossings.  Similarly,
by~\cite[Theorem~1.9]{Cha:2014-1},
\begin{equation}
  \label{equation:universal-bound-computation-2}
  \big|\rhot\big(M(R(U,D)),\phi\big)\big| \le (8m+92)\cdot 69\,713\,280 
\end{equation}
since $R(U,D)$ has a diagram with $8m+92$ crossings (see
Figure~\ref{figure:R(J,D)}). 

By \cite[Lemma~2.4]{Cochran-Harvey-Leidy:2009-1},  the following holds for each~$k$.
\begin{equation}
  \label{equation:L2-signature-for-standard-cobordism}
  \osigmat_G(C)=0, \quad \osigmat_G(E_k)=0.
\end{equation}

To evaluate the terms $\osigmat_\Gamma(V^0)$ and $\osigmat_\Gamma(Z^0_L)$
in~\eqref{equation:L2-signature-of-X}, we will use the following result:

\begin{theorem}[Amenable Signature Theorem~{\cite[Theorem~3.2]{Cha:2010-1}}]
  \label{theorem:amenable-signature}
  Suppose that $W$ is an integral $(n+1)$-solution bounded by the zero surgery
  manifold $M(K)$ of a knot~$K$. Suppose $\Gamma$ is a group which satisfies
  $\Gamma^{(n+1)}=\{1\}$ and lies in Strebel's class $D(\Z_{p_1})$ in the sense
  of~\cite{Strebel:1974-1} \textup{(}or equivalently $\Gamma$ is a locally
  $p$-indicable group, due to~\cite{Howie-Schneebeli:1983-1}\textup{)}.  If
  $\phi\colon \pi_1(M(K)) \to \Gamma$ is a homomorphism which factors through
  $\pi_1(W)$ and sends the meridian of $K$ to an infinite order element
  in~$\Gamma$, then $\rhot(M(K),\phi) = \osigmat_\Gamma(W) = 0$.
\end{theorem}

In our case, $V^0$ is an integral $(n+1)$-solution bounded by~$M(K)$. Also, the
group $\Gamma$ lies in $D(\Z_{p_1})$ by \cite[Lemma~6.8]{Cha-Orr:2009-1}, and we
have $\Gamma^{(n+1)}=\{1\}$ since $\pi_1(X)^{(n+1)} \subset \cP^{n+1}\pi_1(X)$.
The meridian of $K$ has infinite order in~$\Gamma$ since $\Gamma$ surjects onto
$H_1(X)=\Z$ generated by the meridian.  By Amenable Signature
Theorem~\ref{theorem:amenable-signature}, it follows that
\begin{equation}
  \label{equation:L2-signature-for-V0}
  \osigmat_\Gamma(V^0) = \rhot(M(K),\phi)=0.
\end{equation}

Now we will evaluate $\osigmat_\Gamma(Z^0_L)$, using Amenable Signature Theorem
again.  Note $\partial Z^0_L=M(L)$.  An important difference from the above
paragraph is that the 4-manifold $Z^0_L$ is an integral $n$-solution, not $n+1$.
So, Amenable Signature Theorem does not apply directly over $\Gamma$, since
$\Gamma^{(n)}$ is not necessarily trivial.  Instead, we proceed as follows.

Note that the map $\phi\colon \pi_1(M(L)) \to \pi_1(Z_L^0) \to \Gamma$ factors
through $\pi_1(Z_L^0)/\pi_1(Z_L^0)^{(n)}$, by
Assertion~\ref{assertion:triviality-Z_0^L}\@. Let $G$ be the image of
$\pi_1(Z_L^0)/\pi_1(Z_L^0)^{(n)}$ in~$\Gamma$, and let $\psi\colon\pi_1(M(L))
\to G$ be the map induced by~$\phi$.  Since $G$ injects into $\Gamma$, we have
$\rhot(M(L),\phi) = \rhot(M(L),\psi)$, by the $L^2$-induction property (for
instance see~\cite[Eq.~2.3]{Cheeger-Gromov:1985-1}).  Now, since $G$ is a
subgroup of $\Gamma$ which is in $D(\Z_{p_1})$, $G$ is in Strebel's class
$D(\Z_{p_1})$ too. Also, $G^{(n)}$ is trivial since it is the image
of~$\pi_1(Z_L^0)/\pi_1(Z_L^0)^{(n)}$.  The meridian $\mu_L$ of $L$ has infinite
order in~$G$, since $G$ surjects onto $H_1(X)=\Z$ which $\mu_L$ generates.
Therefore Amenable Signature Theorem~\ref{theorem:amenable-signature} (with $n$
in place of $n+1$) applies to $(M(L),\psi)$ to conclude that
\begin{equation}
  \label{equation:L2-signature-Z_L^0}
  \osigmat_\Gamma(Z^0_L) = \rhot(M(L),\phi) = \rhot(M(L),\psi) = 0.
\end{equation}

By~\cite[Lemma~3.3]{Cha-Kim:2017-1}, we may assume that each $Z^0_{i,j}$ has the
property that $\osigmat_\Gamma(Z^0_{i,j})$ is equal to either $0$ or
$\rho(J^i_0,\Z_{p_1})$.  (Indeed, \cite[Lemma~3.3]{Cha-Kim:2017-1} applies when
every $J^i_0$ is $0$-negative; it is the case by
property~\ref{item:negativity-of-J_0} in
Section~\ref{subsection:construction-K_i}.)  Moreover, by
property~\ref{item:independence-of-signature-of-J_0} in
Section~\ref{subsection:construction-K_i}, we have
\begin{equation}
  \label{equation:L2-signature-Z^0_i,j}
  \osigmat_\Gamma(Z^0_{i,j}) = \begin{cases}
    0 \text{ or } \rho(J^1_0,\Z_{p_1}) &\text{if }i=1,
    \\
    0 &\text{if }i>1.
  \end{cases}
\end{equation}

Now,
combine~\eqref{equation:L2-signature-of-X}--\eqref{equation:L2-signature-Z^0_i,j}
to obtain
\[
  N\cdot |\rhot(J^1_0,\Z_{p_1})| \le (6n+8m+86)\cdot 69\;713\;280
\]
where $N$ is one plus the number of the 4-manifolds $Z^0_{1,j}$ such that
$\osigmat(Z^0_{1,j})\ne 0$.  Since $N\ge 1$, it contradicts
property~\ref{item:signature-of-J_0-large-enough} in
Section~\ref{subsection:construction-K_i}.  This completes the proof that $P$
cannot be equal to~$\langle \alpha_D \rangle$.  That is, $P$ must be~$\langle
\alpha_J \rangle$.

\section{Computing and estimating \texorpdfstring{$d$}{d}-invariants}
\label{section:d-invariant-estimate}

In this section, we will continue the proof of
Theorem~\ref{theorem:linear-combination-non-negativity}, to reach a
contradiction under the hypothesis that $P=\langle\alpha_J\rangle$.

\subsection{Finite cyclic covers and their \texorpdfstring{$d$}{d}-invariants}
\label{subsection:from-infinite-cover-to-finite}

We begin by applying a trick introduced in~\cite{Cochran-Harvey-Horn:2012-1},
which we describe below.  Let $K_1=K_0(\alpha,J)$ is a satellite knot, where the
pattern satisfies $\lk(K_0,\alpha)=0$.  Then, the identity on $E(K\sqcup
\alpha)$ extends to a map $E(K_1)=E(K\sqcup \alpha)\cup E(J) \to
E(K\sqcup\alpha) \cup E(U)=E(K_0)$ which induces an isomorphism
$H_1(M(K_1);\Q[t^{\pm1}]) \to H_1(M(K_0);\Q[t^{\pm1}])$, under which we will
identify the Alexander modules.  Essentially,
\cite[Lemma~8.2]{Cochran-Harvey-Horn:2012-1} says the following (see also
\cite[Lemma~5.1]{Cha-Kim:2017-1}): if $K_1=K_0(\alpha,J)$ admits a $1$-negaton
$X_1$ bounded by $M(K_1)$ and if $J$ is unknotted by changing some positive
crossings to negative, then $K_0$ has a $1$-negaton $X_0$ bounded by $M(K_0)$
such that the two maps
\[
  H_1(M(K_i);\Q[t^{\pm1}]) \to H_1(X_i;\Q[t^{\pm1}]) \quad (i=0,1)
\]
have the identical kernel.  In particular, this applies to the satellite
knot $K_1 = R(U,D)(\alpha_D,J^1_{n-1})$ defined in
Section~\ref{subsection:construction-K_i}, since $J^1_{n-1}=P_{n-2}(\eta_n,
J^1_{n-2})$ is unknotted by changing a single positive crossing (see
Figure~\ref{figure:stevedore-pattern}).  Note that here we use that $n\ge 2$.
Therefore, in our case, the knot $K_0 := R(U,D)$ admits a $1$-negaton bounded
by~$M(K_0)$, say $W$, such that
\begin{equation}
  \label{equation:kernel-condition-after-removing-J}
  \langle\alpha_J\rangle = \Ker\{H_1(M(K_0);\Q[t^{\pm1}])
    \rightarrow H_1(W;\Q[t^{\pm1}])\}.
\end{equation}
We will derive a contradiction from the existence of this 1-negaton~$W$
for~$K_0$.

The next step is to pass to finite degree branched cyclic covers, to which
Heegaard Floer homology machinery applies,
following~\cite[Section~5.1]{Cha-Kim:2017-1}. Let $\Sigma_r$ be the $r$-fold
branched cyclic cover of~$(S^3,K_0)$. The curves $\alpha_J$ and $\alpha_D$ in
Figure~\ref{figure:R(J,D)} represent homology classes in $\Sigma_r$, say $x_1$
and $x_2\in H_1(\Sigma_r)$ respectively.  (The classes $x_1$ and $x_2$ are
defined up to covering transformation, but it will cause no problem in our
argument.)  Due to~\cite{Milnor:1968-1}, we have $H_1(\Sigma_r) =
H_1(M(K_0);\Z[t^{\pm1}])/\langle t^r-1\rangle$.  Recall that
$H_1(M(K_0);\Z[t^{\pm1}])$ is given
by~\eqref{equation:integral-alexander-module}. From this, by elementary
computation, it follows that that $H_1(\Sigma_r)=\Z_{(m+1)^r-m^r} \oplus
\Z_{(m+1)^r-m^r}$ and the summands are generated by $x_1$ and $x_2$.  In
particular, $\Sigma_r$ is a $\Z_2$-homology sphere.

For a rational homology 3-sphere $Y$ and a \spinc\ structure $\mathfrak{t}$
on~$Y$, Ozsv\'ath and Szab\'o defined a correction term invariant
$d(Y,\mathfrak{t})$ using the Heegaard Floer chain
complex~\cite{Ozsvath-Szabo:2003-2}.  In case of a $\Z_2$-homology sphere $Y$,
the unique spin structure determines a canonical \spinc\ structure on~$Y$, which
we denote by~$\mathfrak{s}_Y$, and all \spinc\ structures of $Y$ are given in
the form $\mathfrak{s}_Y+c$, where $c\in H^2(Y)$ and $+$ designates the action
of $H^2(Y)$ on the set of \spinc\ structures.  For $x\in H_1(Y)$, let $\widehat
x\in H^2(Y)$ be the Poincar\'e dual of~$x$.  Techniques used
in~\cites{Cochran-Harvey-Horn:2012-1,Cha-Kim:2017-1} give us the following
$d$-invariant obstruction.

\begin{lemma}
  \label{lemma:d-invariant-obstruction}
  Suppose $M(K_0)$ bounds a $1$-negaton which
  satisfies~\eqref{equation:kernel-condition-after-removing-J}.  If $r$ is
  sufficiently large, then $d(\Sigma_r, \mathfrak{s}_{\Sigma_r}+ k\cdot
  \widehat{x_1})\ge 0$ for all $k\in \Z$.
\end{lemma}

\begin{proof}
  Attach a 2-handle to a given $1$-negaton~$W$, along the zero-framed meridian
  of $K_0$, to obtain a 4-manifold which we temporarily call~$V$.  Note that
  $\partial V = S^3$ and the cocore of the 2-handle is a slicing disk
  $\Delta\subset V$ bounded by~$K_0$.  Take the $r$-fold branched cyclic cover
  of $(V,\Delta)$, and call it~$V_r$.  Here the $r$-fold branched cyclic cover
  is defined since, using that $W$ is a $1$-negaton, we see that $H_1(W\sm
  \Delta)=\Z$ generated by a meridian of~$\Delta$.  Indeed, if we denote by
  $W_r$ the $r$-fold cyclic cover of $W$, then $V_r$ is obtained by attaching a
  2-handle to~$W_r$. The 4-manifold $V_r$ is bounded by~$\Sigma_r$.  

  The first key step is to relate the hypothesis $P=\langle\alpha_J\rangle$,
  which is associated with the infinite cyclic cover, to the kernel
  \[
    G:=\Ker\{H_1(\Sigma_r)\to H_1(V_r)\}
  \]
  associated with finite covers.  The following is a modified version
  of~\cite[Lemma~5.2]{Cha-Kim:2017-1}.

  \begin{assertion*}
    Under the hypothesis that $P=\langle\alpha_J\rangle$, $x_1\in G$ for all
    large primes~$r$.
  \end{assertion*}

  Although we do not use it, we remark that the assertion implies that
  $G=\langle x_1\rangle$, since it is known that $|G|$ is equal to
  $|H_1(\Sigma_m)|^{1/2} = (m+1)^r-m^r$.

  \begin{proof}[Proof of Assertion]
    
    Let $E_r$ be the $r$-fold cyclic cover of the exterior~$E(K_0)$ for $r \le
    \infty$.  Consider the following commutative diagram:
    \[
      \begin{tikzcd}[column sep=small]
        H_1(M(K_0);\Z[t^{\pm1}]) \ar[r,equal]
        &[-1ex] H_1(E_\infty) \ar[r]\ar[d]
        & H_1(W_\infty) \ar[r,equal]\ar[d]
        &[-1ex] H_1(W;\Z[t^{\pm1}]) \ar[r]
        & H_1(W;\Q[t^{\pm1}])
        \\
        & H_1(E_r) \ar[d] \ar[r]
        & H_1(W_r) \ar[d]
        \\
        & H_1(\Sigma_r) \ar[r]
        & H_1(V_r)
      \end{tikzcd}
    \]
    The vertical arrows are induced by coverings and inclusions.  Since
    $\alpha_J$ lies in the kernel of the composition of the top row by the
    hypothesis, $\alpha_J$ is $\Z$-torsion in $H_1(W;\Z[t^{\pm1}])$. That is,
    $a\cdot \alpha_j=0$ in $H_1(W;\Z[t^{\pm1}])$ for some nonzero~$a\in\Z$.  By
    the above diagram, it follows that $a\cdot x_1 \in H_1(\Sigma_r)$ lies in
    the kernel $G$ of $H_1(\Sigma_r)\to H_1(V_r)$.

    Suppose $r$ is a prime not smaller than any prime factor of~$a$.  Under this
    assumption, we claim that $\gcd(a, (m+1)^r - m^r)=1$. From this it follows
    that $x_1$ lies in $G$, since $x_1$ has order $(m+1)^r - m^r$ in
    $H_1(\Sigma_r)$. This proves the assertion, modulo the proof of the claim.

    To show the claim, it suffices to show that every prime factor $q$ of~$a$ is
    relatively prime to $(m+1)^r - m^r$.  It is obviously true, if $q\mid m$ or
    $q\mid m+1$.  So, suppose $q$ divides neither $m$ nor~$m+1$ but $q$ divides
    $(m+1)^r - m^r$.  Let $u=m^*(m+1)$, where $m^*$ is an arithmetic inverse of
    $m \bmod q$.  We have $u^r \equiv 1 \bmod q$ by the hypothesis, and $u^{q-1}
    \equiv 1 \bmod q$ by Fermat's little theorem.  Since $r$ is a prime and
    $u\not\equiv 1 \bmod q$, it follows that $r\mid q-1$.  This contradicts the
    assumption that $r\ge q$. This completes the proof of the claim.
  \end{proof}

  The assertion enables us to invoke a result of Cochran, Harvey and
  Horn~\cite[Theorem~6.5]{Cochran-Harvey-Horn:2012-1}, which says the following:
  if $W$ is a $1$-negaton bounded by $M(K_0)$, then $d(\Sigma_r,\widehat x)\ge
  0$ for all $x$ lying in $G=\Ker\{H_1(\Sigma_r) \to H_1(V_r)\}$.  Applying this
  to $x=k\cdot x_1$, the proof of Lemma~\ref{lemma:d-invariant-obstruction} is
  completed.
\end{proof}

\begin{theorem}
  \label{theorem:d-invariant-Sigma_r}
  Let $m\ge 1$ is an odd integer and $r\ge 1$ is an odd prime power.  Let $k$ be
  the arithmetic inverse of $2 \bmod (m+1)^r-m^r$.  Then $d(\Sigma_r,
  \mathfrak{s}_{\Sigma_r} + k\widehat x_1) \le -\frac32$.
\end{theorem}

On the other hand, Lemma~\ref{lemma:d-invariant-obstruction} says that
$d(\Sigma_r, \mathfrak{s}_{\Sigma_r} + k\widehat x_1)$ must be non-negative.
This contradiction implies that the kernel $P$ defined
in~\eqref{equation:definition-of-P} cannot be equal to~$\langle\alpha_J\rangle$.
This completes the proof of
Theorem~\ref{theorem:linear-combination-non-negativity}.

For $m=1$, Theorem~\ref{theorem:d-invariant-Sigma_r} was already shown
in~\cite[Theorem~5.4]{Cha-Kim:2017-1}.  (We remark that the symbol $m$
in~\cite{Cha-Kim:2017-1} denotes our~$r$.) So, in the remaining part of this
paper, we will assume that $m>1$.

\begin{proof}[Proof of Theorem~\ref{theorem:d-invariant-Sigma_r}] 
  
  Let $A$ and $B$ be the 3-manifolds given by the surgery presentations in
  Figures~\ref{figure:surgery-diagram-A-B}.  Let $Y_r$ be the connected sum of
  $A$, $B$ and $r-3$ copies of the lens space $L_{2m+1,1}$. (We use the
  orientation convention that $L_{p,1}$ is the $p$-framed surgery on the trivial
  knot $U$ in~$S^3$.)  Then, arguments in~\cite[Section~6.1]{Cha-Kim:2017-1}
  construct a negative definite 4-manifold  $W$ with $\partial W = Y_r \sqcup
  -\Sigma_r$ and $b_2(W)=(2m+1)r-4m+1$, and construct a \spinc\ structure
  $\mathfrak{t}$ on $W$ such that $c_1(\mathfrak{t})^2 = -r$,
  $c_1(\mathfrak{t}|_{\Sigma_r}) = \widehat{x_1}$ and $c_1(\mathfrak{t}|_{Y_r})
  = 0$.  Indeed, \cite[Section~6.1]{Cha-Kim:2017-1} is the case of $m=1$.  We do
  not repeat the details here, since exactly the same method works under our
  assumption that $m\ge 1$ is odd.  Perhaps the least obvious part is that the
  inverse matrix $P^{-1}$ in \cite[Eq.~(6.10)]{Cha-Kim:2017-1} should be
  replaced with a block matrix $P^{-1}=(R_{ij})_{1\le i,j\le r-1}$ with
  \[
    R_{ij} = \frac{((m+1)^i-m^i)((m+1)^{r-j}-m^{r-j})}{(m+1)^r-m^r}
    \begin{bmatrix} 0 & (m+1)^{i-j} \\ m^{i-j} & 0 \end{bmatrix}
    \text{ for } i\ge j,
  \]
  and $R_{ij}=R^T_{ji}$ for $i<j$.

  \begin{figure}[H]
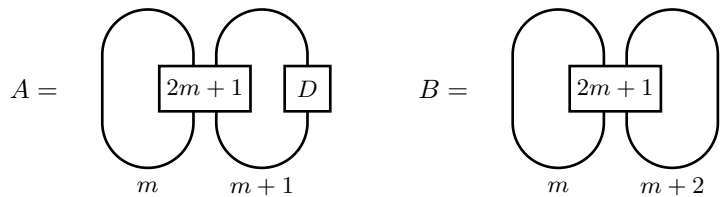

    \includestandalone{surgery-diagram-A-B}
    \caption{The 3-manifolds~$A$ and~$B$.
    The box \,{\fboxsep=2pt\fboxrule=.9pt\fbox{\footnotesize$2m+1$}}\,
    represents $2m+1$ right handed full
    twists between vertical strands.}
    \label{figure:surgery-diagram-A-B}
  \end{figure}
    
  By applying Ozsv\'ath-Szab\'o's $d$-invariant
  inequality~\cite[Theorem~9.6]{Ozsvath-Szabo:2003-2} to the negative definite
  4-manifold $W$, and by using additivity of the $d$-invariant under connected
  sum, we have
  \begin{equation*}
    \begin{aligned}
      d(\Sigma_r,\mathfrak{t}|_{\Sigma_r}) 
      & \le d(Y_r,\mathfrak{t}|_{Y_r}) - \frac{c_1(\mathfrak{t})^2 + b_2(W)}{4}
      \\
      & = d(A,\mathfrak{t}|_{A}) + d(B,\mathfrak{t}|_{B})
      + (r-3)d(L_{2m+1,1},\mathfrak{t}|_{L_{2m+1,1}}) - \frac{2mr-4m+1}{4}.
    \end{aligned}
  \end{equation*}  
  Since $c_1(\mathfrak{t}|_{Y_r})=0$, we have $\mathfrak{t}|_{L_{2m+1,1}}=
  \mathfrak{s}_{L_{2m+1,1}}$.  By a $d$-invariant formula for lens spaces given
  in \cite[Proposition~4.8]{Ozsvath-Szabo:2003-2},
  $d(L_{2m+1,1},\mathfrak{s}_{L_{2m+1,1}})=m/2$. By
  Lemmas~\ref{lemma:d-invariant-A} and~\ref{lemma:d-invariant-B}, which we will
  prove below, we have $d(A,\mathfrak{t}|_A) \le (m-7)/4$ and
  $d(B,\mathfrak{t}|_B) \le (m+2)/4$.  Combine them with the above inequality,
  to obtain $d(\Sigma_r, \mathfrak{t}|_{\Sigma_r}) \le -\frac32$.  Observe that
  $c_1(\mathfrak{t}|_{\Sigma_r})=\widehat{x_1} = 2k \widehat{x_1} =
  c_1(\mathfrak{s}_{\Sigma_r} + k \widehat{x_1})$, since $2k\equiv 1\bmod
  (m+1)^r-m^r$ and $x_1$ has order $(m+1)^r-m^r$.  It follows that
  $\mathfrak{t}|_{\Sigma_r} = \mathfrak{s}_{\Sigma_r}$, since $\Sigma_r$ is a
  $\Z_2$-homology sphere.
\end{proof}

Note that $|H_1(A)| = \bigl|\det\sbmatrix{m & -2m-1 \\ -2m-1 & m+1}\bigr| =
3m^2+3m+1$ is odd, so there is a unique \spinc\ structure $\mathfrak{s}$ of $A$
such that $c_1(\mathfrak{s}) = 0$.

\begin{lemma}
  \label{lemma:d-invariant-A}
  For the \spinc\ structure $\mathfrak{s}$ on $A$ with $c_1(\mathfrak{s}) = 0$,
  $d(A,\mathfrak{s}) \le (m-7)/4$.
\end{lemma}

For the case of $B$, since $|H_1(B)| = \bigl|\det\sbmatrix{m & -2m-1 \\ -2m-1 &
m+2} \bigr| = 3m^2+2m+1$ is even, there are exactly two \spinc\ structures
$\mathfrak{s}$ of $B$ such that $c_1(\mathfrak{s}) = 0$.

\begin{lemma}
  \label{lemma:d-invariant-B}
  If $\mathfrak{s}$ is a \spinc\ structure of $B$ such that $c_1(\mathfrak{s}) =
  0$, then $d(B,\mathfrak{s}) \le (m+2)/4$.
\end{lemma}

\begin{proof}[Proof of Lemma~\ref{lemma:d-invariant-A}] Let $A'$ be the
  3-manifold obtained by $(m+6)$-surgery on the knot $T\#-D$.  See
  Figure~\ref{figure:cobordism-from-A}.  Let $Y$ be the 3-manifold given by the
  last surgery diagram in Figure~\ref{figure:cobordism-from-A}. The four
  diagrams in Figure~\ref{figure:cobordism-from-A} describe a cobordism from
  $A\# A'$ to~$Y$.  More precisely, start by taking ${(A\# A') \times I}$.
  Attach a 2-handle, to obtain a cobordism, say $V$, from $A\#A'$ to the second
  surgery diagram in Figure~\ref{figure:cobordism-from-A}.  Apply handle slide,
  to change the second surgery diagram to the third.  Observe that, in the third
  surgery diagram, the two components with framing $m+6$ and $0$ give a $S^3$
  summand, while the two components with framing $m$ and $2m+7$ form a link
  concordant to the link in the last surgery diagram, which describes~$Y$.  It
  follows that there is a homology cobordism, say $V'$, from the third surgery
  diagram to~$Y$.  Now $V\cup V'$ is a cobordism from $A\#A'$ to~$Y$.

  \begin{figure}[t]
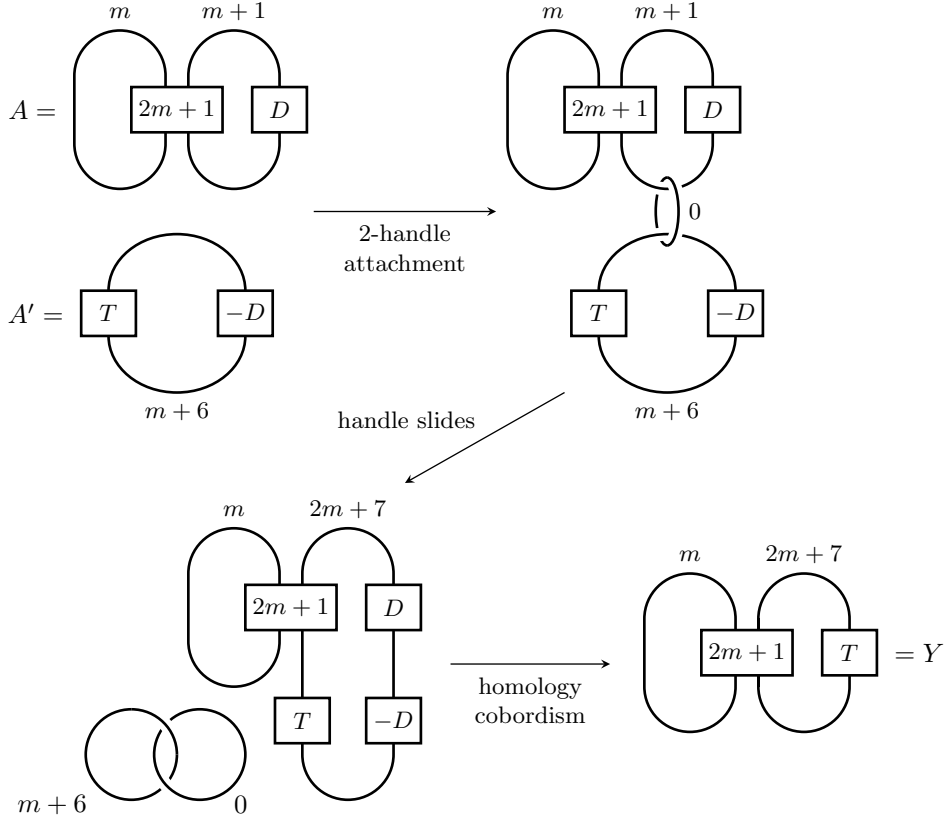

    \includestandalone{cobordism-from-A}
    \caption{A cobordism from~$A\# A'$ to~$Y$.}
    \label{figure:cobordism-from-A}
  \end{figure}

  We claim that $V$ is negative definite.  To see this, view the surgery diagram
  of the 3-manifold $A\# A'$ as a Kirby diagram of a 4-manifold~$V_0$.  That is,
  $V_0$ consists of one 0-handle and three 2-handles attached along the framed
  link diagram of $A\# A'$ in Figure~\ref{figure:cobordism-from-A}.  We have
  $\partial V_0=A\# A'$.  The second diagram in
  Figure~\ref{figure:cobordism-from-A}, or equivalently the third diagram, is a
  Kirby diagram of the 4-manifold $V_0\cup_{A\# A'} V$.  The Kirby diagram of
  $V_0$ has linking matrix
  \[
    L = \begin{bmatrix}
      m & -2m-1 & 0 \\ -2m-1 & m & 0 \\ 0 & 0 & m+6
    \end{bmatrix}
  \]
  which has signature~$1$ since $m>0$ and the top upper $2\times 2$ submatrix
  has negative determinant.
  It follows that $\sign V_0 = 1$.  The third diagram in
  Figure~\ref{figure:cobordism-from-A} has linking matrix
  \[
    \begin{bmatrix}
      m & -2m-1 & 0 & 0\\ -2m-1 & 2m+7 & 0 & 0 \\
      0 & 0 & 0 & -1 \\ 0 & 0 & -1 & m+6
    \end{bmatrix}
  \]
  which has vanishing signature.  So $\sign V_0\cup V = 0$.  By Novikov
  additivity, it follows that $\sign V=-1$.  Since $b_2(V)=1$, this proves the
  claim that $V$ is negative definite.

  Since $V'$ is a homology cobordism, $H_*(V\cup V')=H_*(V)$.  Consequently
  $V\cup V'$ is negative definite.
  
  We will construct a generator of $H_2(V\cup V') = H_2(V)$ and use it to
  describe a certain \spinc\ structure on~$V\cup V'$.  Let $\sigma$ be the core
  of the 2-handle of~$(V,A\#A')$, and let $\alpha=\partial\sigma$ be its
  attaching circle, which lies in~$A\#A'$. See the second diagram in
  Figure~\ref{figure:cobordism-from-A}, in which $\alpha$ is the zero-framed
  circle.  Since the linking matrix $L$ is a presentation for $H_1(A\#A')$, it
  is seen that $H_1(A\#A')=\Z_{3m^2+3m+1}\oplus \Z_{m+6}$, and $\alpha = (1,1)$
  in $H_1(A\#A')$.  So, the order of $\alpha$ in  $H_1(A\#A')$ is
  $(3m^2+3m+1)(m+6)/d$, where $d:=\gcd(3m^2+3m+1, m+6)$.  (Indeed, it can be
  seen that $d$ is either $91$ or~$1$.)  From this, it follows that there is a
  2-cycle $z$ in $A\#A'$ such that
  \[
    \partial z= \frac{(3m^2+3m+1)(m+6)}{d}\cdot \alpha.
  \]    
  Moreover, the 2-chain
  \[
    E:= \frac{(3m^2+3m+1)(m+6)}{d} \cdot \sigma - z
  \]
  is a generator of $H_2(V)=\Z$, by a standard Mayer-Vietoris argument for the
  2-handle attachment.

  The self-intersection number $E\cdot E$ is equal to the intersection number
  (in $A\#A'$) of $z$ and a pushoff of $\partial z$, say $\partial z'$, taken
  along the 2-handle attachment framing (which is the zero framing in
  Figure~\ref{figure:cobordism-from-A}). So $E\cdot E$ is equal to the linking
  number of $\partial z$ and its pushoff in the rational homology
  sphere~$A\#A'$.  In addition, the linking number can be computed using the
  linking matrix~$L$ (for instance, see~\cite[Theorem~3.1]{Cha-Ko:2000-1}):
  \[
    \begin{aligned}
      E\cdot E &= \lk_{A\#A'}(\partial z,\partial z')
      = \frac{(3m^2+3m+1)^2(m+6)^2}{d^2} \cdot
      \begin{bmatrix} 0 & -1 & 1 \end{bmatrix} L^{-1}
      \begin{bmatrix} 0 \\ -1 \\ 1 \end{bmatrix}
      \\
      & = \frac{(3m^2+3m+1)(m+6)(-2m^2+3m-1)}{d^2}.
    \end{aligned}
  \]
  Since the factor $-2m^2+3m-1$ of the numerator is even and $d$ is odd, $E\cdot
  E$ is even.  From this, it follows that $0\in H^2(V\cup V')$ is
  characteristic. Therefore there is a \spinc\ structure $\mathfrak{t}$ on
  $V\cup V'$ such that $c_1(\mathfrak{t})=0$.  By the Ozsv\'ath-Szab\'o
  inequality~\cite[Theorem~9.6]{Ozsvath-Szabo:2003-2}, we have
  \[
      d(Y,\mathfrak{t}|_Y) - d(A,\mathfrak{t}|_A) - d(A',\mathfrak{t}|_{A'})
      \ge \frac 14.
  \]
  Note that $A'$ is the $(m+6)$-surgery on the knot $T\# -D$, and
  $c_1(\mathfrak{t}|_{A'})=0$ since $c_1(\mathfrak{t})=0$.   By techniques
  of~\cite[p.~2150-2151]{Cochran-Harvey-Horn:2012-1} and
  \cite[Appendix~A]{Hedden-Kim-Livingston:2016-1}, $d(A',\mathfrak{t}|_{A'}) =
  d(L_{m+6},\mathfrak{s}_{L_{m+6}})$.  By Ozsv\'ath-Szab\'o's formula for lens
  spaces~\cite[Proposition~4.8]{Ozsvath-Szabo:2003-2}, we have
  $d(L_{m+6},\mathfrak{s}_{L_{m+6}}) = (m+5)/4$.
  By Lemma~\ref{lemma:d-invariant-Y}, which we will prove below, we have
  $d(Y,\mathfrak{t}|_Y) \le (2m-1)/4$.  (Note that $c_1(\mathfrak{t}|_Y)=0$
  since $c_1(\mathfrak{t})=0$.)  Combining these with the above inequality, it
  follows that $d(A,\mathfrak{s}) \le (m-7)/4$.
\end{proof}

\subsection{A quick summary of N\'emethi's method for Seifert 3-manifolds}
\label{subsection:nemethi-method}

For the reader's convenience, we provide a summary of N\'emethi's method to
compute $d$-invariants~\cite{Nemethi:2005-1}, which we will use in
Section~\ref{subsection:d-invariant-Y-B}.
We focus on the case of Seifert
3-manifolds, which is treated in~\cite[Section~11]{Nemethi:2005-1},
while~\cite{Nemethi:2005-1} provides techniques for a larger class of certain
plumbed 3-manifolds.

Let $Y$ be a Seifert 3-manifold.  It is well known that $Y$ admits a surgery
presentation of a specific form, which is shown in the left of
Figure~\ref{figure:plumbing-graph}.  The associated star-shaped plumbing graph,
which is shown in the right of Figure~\ref{figure:plumbing-graph}, is often used
to describe the 3-manifold~$Y$, as the boundary of a plumbed 4-manifold~$X$: for
each vertex, take a disk bundle over a 2-sphere whose Euler number is the
integer decoration of the vertex.  For each edge, perform $+1$ plumbing between
the two disk bundles corresponding to the endpoints.  The result is a 4-manifold
$X$ with $\partial X=Y$. 

Let $\nu$ be the number of branches of the star-shaped graph in
Figure~\ref{figure:plumbing-graph}. In this subsection, we assume
that $\nu\ge 3$, and that $X$ is negative definite.  (We remark that not all
Seifert 3-manifolds $Y$ are described by a plumbing graph satisfying this
assumption.)

\begin{figure}[H]
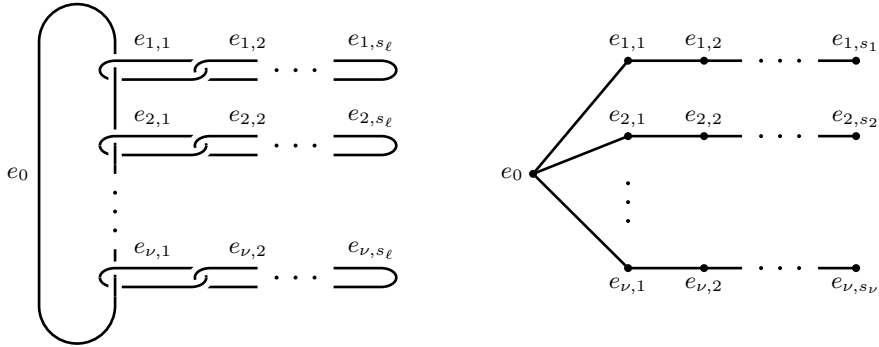

  \includestandalone{plumbing-graph}
  \caption{A Seifert 3-manifold and its plumbing graph.}
  \label{figure:plumbing-graph}
\end{figure}

We use the following notation. Let $e_0$ be the decoration of the root vertex.
Let $s_\ell$ be the number of (non-root) vertices on the $\ell$th branch, and
let $e_{\ell,1}$, \dots,~$e_{\ell,s_\ell}$ be the decorations of those $s_\ell$
vertices.  See Figure~\ref{figure:plumbing-graph}.  We may assume that
$e_{\ell,j}\le -2$ for all~$\ell$,~$j$ (for instance see~\cite{Neumann:1981-1}).
Let  $b_0$ and $b_{\ell,j} \in H_2(X)$ be the classes of 2-spheres corresponding
to vertices with decoration $e_0$ and $e_{\ell,j}$ respectively ($1\le \ell \le
\nu$, $1\le j\le s_\ell$).  They form a basis for the free abelian
group~$H_2(X)$. Let $b^*_0$ and $b^*_{\ell,j}\in H^2(X)$ be basis elements (hom)
dual to $b_0$ and~$b_{\ell,j}$. With respect to these bases, the intersection
form $\lambda\colon H_2(X)\times H_2(X) \to X$, or its adjoint $\lambda\colon
H_2(X) \to \Hom(H_2(X),\Z) = H^2(X)$ is given as follows:
$\lambda(b_0,b_0)=e_0$, $\lambda(b_{\ell,j}, b_{\ell,j}) = e_{\ell,j}$, and for
$b\ne b'$, $\lambda(b,b')=1$ if $\{b,b'\}=\{b_0,b_{\ell,1}\}$ or
$\{b_{\ell,j},b_{\ell,j+1}\}$, and $\lambda(b,b')=0$ otherwise.  For $1\le
\ell\le \nu$, define a continued fraction by
\[
  \frac{\alpha_\ell}{\omega_\ell} := [-e_{\ell,1},\ldots,-e_{\ell,s_\ell}] =  
  -e_{\ell,1} - \cfrac{1}{-e_{\ell,2} - \cfrac{1}{\,\,\,\vdots\,\,\,}}.
\]
The condition $e_{\ell,j}\le -2$ implies that $\alpha_\ell/\omega_\ell >1$.  So
we may assume $0\le \omega_\ell < \alpha_\ell$.  Then, a standard
diagonalization process applied to the intersection form $\lambda$ gives us a
diagonal matrix whose all diagonals are easily seen to be negative, possibly
except one which is equal to
\[
  e := e_0 + \sum_{\ell=1}^\nu \omega_\ell/\alpha_\ell.
\]
So, $X$ is negative definite if and only if $e<0$.

The set of characteristic elements in $H^2(X)$ is defined to be
\[
  \Char(X) := \{\xi \in H^2(X) \mid \xi(b_j)\equiv \lambda(b_j,b_j) \bmod 2
  \text{ for all } j\}.
\]
The Chern class $c_1\colon \Spinc(X) \to \Char(X)$ is bijective, since $H^2(X)$
does not have any 2-torsion.  Under the identification via $c_1$, the action of
an element $c\in H^2(X)$ on $\Spinc(X)$ is given by $\xi \mapsto \xi+2c$ for
$\xi\in \Char(X)=\Spinc(X)$. In particular, the action of $x\in H_2(X)$ on
$\Char(X)=\Spinc(X)$ via $\lambda\colon H_2(X)\to H^2(X)$ is given by $\xi
\mapsto \xi+2\lambda(x)$.

The Chern class $c_1\colon \Spinc(Y)\to H^2(Y)$ is not injective in general, so
the standard identification of \spinc\ structures of $Y$ is given indirectly
using~$X$: we have a bijection
\[
  \Spinc(Y) \approx \Char(X)/2\lambda(H_2(X)).
\]
Here, for a \spinc\ structure $\xi \in \Spinc(X)=\Char(X)$, the coset $[\xi] =
\xi+2\lambda(H_2(X))$ in $\Char(X)/2\lambda(H_2(X))$ corresponds to the
restriction of $\xi$ on~$Y$. Essentially, the bijectivity is a consequence of
the fact that $H^2(Y)$ is the cokernel of $\lambda\colon H_2(X) \to H^2(X)$.

In~\cite{Nemethi:2005-1}, the notion of a \emph{distinguished representative} is
used to express a \spinc\ structure of~$Y$.  Instead of the original definition
(see Section~5, especially Definition~5.1 of~\cite{Nemethi:2005-1}),  we will
use a characterization theorem as a definition.  We need the following notation.
For $1\le i\le j\le s_\ell$, let $n^\ell_{i,j}/d^\ell_{i,j} := [-e^\ell_i,
\ldots, -e^\ell_j]$ where $n^\ell_{i,j} > 0$ and $\gcd(n^\ell_{i,j},
d^\ell_{i,j}) = 1$. Note that $\alpha_\ell/\omega_\ell =
n^\ell_{1,s_\ell}/d^\ell_{1,s_\ell}$.  Define an element $K\in \Spinc(X) =
\Char(X)\subset H^2(X)$ by
\begin{equation}
  \label{equation:canonical-spinc-structure}
  K(b_0) = -e_0-2,\quad K(b_{\ell,j}) = -e_{\ell,j}-2
  \text{ for } 1\le \ell\le \nu,\, 1\le j\le s_\ell. 
\end{equation}
The element $K$ is characteristic since $K(b_0) \equiv \lambda(b_0,b_0)$ and
$K(b_{\ell,j}) \equiv \lambda(b_{\ell,j},b_{\ell,j}) \bmod 2$.
In~\cite{Nemethi:2005-1}, $K$ is called the \emph{canonical \spinc\ structure}.
We have $\Char(X) = K + 2H^2(X)$.  Consider a class $k_r\in \Char(X)$ of the
form
\begin{equation}
  \label{equation:k_r}
  k_r = K - 2 \biggl( a_0^\nstrut b^*_0
  + \sum_{\ell,j} a_{\ell,j}^\nstrut b_{\ell,j}^* \biggr)
\end{equation}
where $a_0$ and $a_{\ell,j}$ are integers.  Let
\begin{equation}
  \label{equation:a_ell}
  a_\ell = \sum_{t=1}^{s_\ell-1} n^\ell_{t+1,s_\ell} a_{\ell,t}.  
\end{equation}

\begin{definition}[{\cite[Corollary~11.7]{Nemethi:2005-1}}]
  \label{definition:distinguished-representative}
  The class $k_r$ in~\eqref{equation:k_r} called a \emph{distinguished
  representative} if $0\le a_\ell < \alpha_\ell$ for all~$\ell$, and if
  \begin{equation}
    \label{equation:cond-for-distinguished-repn}    
    0\le a_0 \le -1 -i e_0
    - \sum_{\ell=1}^\nu \Bigl[ \frac{i\omega_\ell + a_\ell}{\alpha_\ell} \Bigr]
  \end{equation}
  for all $i>0$.  Here $[x]$ is the largest integer not greater than~$x$.
\end{definition}

To state N\'emethi's formula for the $d$-invariant, we need one more notation.
Let
\begin{equation}
  \label{equation:tau(i)}
  \tau(i) = \sum_{t=0}^{i-1} \biggl(a_0+1-te_0 + \sum_{\ell=1}^\nu
  \Bigl[ \frac{-t\omega_\ell + a_\ell}{\alpha_\ell} \Bigr] \biggr).
\end{equation}
In particular, $\tau(0)=0$.

\begin{theorem}[{N\'emethi~\cite[p.~1038]{Nemethi:2005-1}}]
  \label{theorem:nemethi-formula}
  Suppose $X$ is negative definite and $\nu\ge 3$.  Suppose the class $k_r$
  given in~\eqref{equation:k_r} is a distinguished
  representative. Then, for the \spinc\ structure $[k_r]$ of $Y$, the
  $d$-invariant is given by
  \[
    d(Y,[k_r]) = \frac{k_r^2+b_2(X)}{4} - 2\cdot \min \{ \tau(i) \mid i\ge 0\}.
  \] 
\end{theorem}

\begin{remark}
  \phantomsection\label{remark:min-tau(i)}
  \leavevmode\Nopagebreak
  \begin{enumerate}
    \item 
    For each $k_r$ given by~\eqref{equation:k_r}, the minimum in
    Theorem~\ref{theorem:nemethi-formula} can be found in finite steps.  To see
    this, let $\Delta_i = \tau(i+1)-\tau(i)$.  Then we have
    \begin{align*}
      \Delta_i & \ge a_0 + 1 - ie_0 + \sum_{\ell=1}^\nu
      \frac{- i \omega_\ell + a_\ell + \alpha_\ell - 1}{\alpha_\ell} 
      \\
      & = -e\cdot i
      + \biggl( 1+a_0 +
      \sum_{\ell=1}^\nu \frac{a_\ell-\alpha_\ell+1}{\alpha_\ell} \biggr) \ge 0
    \end{align*}
    if $i$ is not smaller than
    \[
      R := \biggl( 1+a_0 +
      \sum_{\ell=1}^\nu \frac{a_\ell-\alpha_\ell+1}{\alpha_\ell} \biggr) \Big/
      (-e).
    \]
    Here we use that $e$ is negative since $X$ is negative definite.  So, the
    minimum in Theorem~\ref{theorem:nemethi-formula} can be taken over $0\le i \le
    \max\{0,R\}$. 

    \item A similar argument shows that it can be determined in finite steps
    whether a class $k_r$ given by~\eqref{equation:k_r} satisfies
    Definition~\ref{definition:distinguished-representative}.  Indeed, the right
    hand side of~\eqref{equation:cond-for-distinguished-repn} is bounded from
    below by $-(1+\sum a_\ell/\alpha_\ell)-ei$.  Since $e<0$,
    \eqref{equation:cond-for-distinguished-repn} is satisfied for all large~$i$,
    and thus it suffices to check~\eqref{equation:cond-for-distinguished-repn}
    for only finitely many~$i$.
  \end{enumerate}
  Using (1) and (2), it is straightforward to write a practically efficient
  algorithm (and computer code) to find distinguished representatives of all
  \spinc\ structures of $Y$ and compute the associated $d$-invariants.
\end{remark}

\subsection{\texorpdfstring{$d$}{d}-invariants of the 3-manifolds
  \texorpdfstring{$Y$}{Y} and~\texorpdfstring{$B$}{B}}
\label{subsection:d-invariant-Y-B}

Let $Y$ be the 3-manifold given by the last surgery diagram in
Figure~\ref{figure:cobordism-from-A}, or equivalently by the first surgery
diagram in Figure~\ref{figure:surgery-diagram-Y}.  The following lemma gives an
estimate of the $d$-invariant of $Y$, which is used to complete the proof of
Lemma~\ref{lemma:d-invariant-A}.   Note that there are two \spinc\ structures on
$Y$ satisfying $c_1(\mathfrak{s})=0$, since $|H_1(Y)| = 2m^2-3m+1$ is even.

\begin{lemma}
  \label{lemma:d-invariant-Y}
  If $\mathfrak{s}$ is a \spinc\ structure on $Y$ such that
  $c_1(\mathfrak{s})=0$, then $d(Y,\mathfrak{s}) \le (2m-1)/4$.
\end{lemma}

\begin{proof}
  The surgery diagram calculus in Figure~\ref{figure:surgery-diagram-Y} shows
  that $Y$ is a Seifert 3-manifold.  The last plumbing graph in
  Figure~\ref{figure:surgery-diagram-Y} describes a plumbed 4-manifold $X$ with
  $\partial X=Y$.

  \begin{figure}[ht]
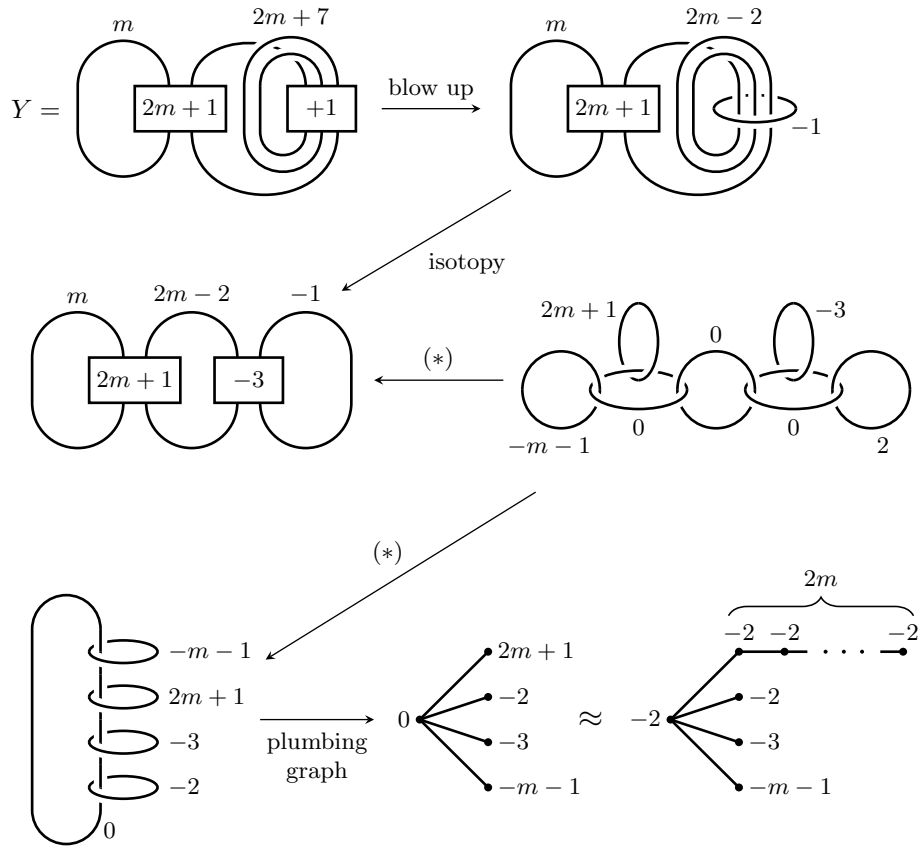

    \includestandalone{surgery-diagram-Y}
    \caption{Surgery diagram calculus which gives a plumbing tree for~$Y$. The symbol $(*)$ means handle slides and elimination of components with 0-framed meridians.}
    \label{figure:surgery-diagram-Y}
  \end{figure}

  To compute the $d$-invariant, we will apply the method discussed in
  Section~\ref{subsection:nemethi-method}. Using the notation in
  Section~\ref{subsection:nemethi-method}, denote the basis of $H_2(X)$ by
  $\{b_0$, $b_{1,1}$, \dots,~$b_{1,2m}$, $b_{2,1}$, $b_{3,1}$, $b_{4,1}\}$, and
  the dual basis of $H^2(X)$ by  $\{b^*_0$, $b^*_{1,1}$, \dots,~$b^*_{1,2m}$,
  $b^*_{2,1}$, $b^*_{3,1}$, $b^*_{4,1}\}$. The intersection form $\lambda\colon
  H_2(X)\times H_2(X)\to \Z$ is computed straightforwardly from the plumbing
  graph:
  \[
    \lambda = \left[
      \begin{array}{c|ccccc|c|c|c}
        -2 & 1 & & & & & 1 & 1 & 1 \\
        \hline
        1 & -2 & 1 & & & & & & \\
        & 1 & -2 & 1 & & & & & \\
        & & \ddots & \ddots & \ddots & & & \\
        & & & 1 & -2 & 1 & & & \\
        & & & & 1 & -2 & & & \\
        \hline
        1 & & & & & & -2 & & \\
        \hline
        1 & & & & & & & -3 & \\
        \hline
        1 & & & & & & & & -m-1
      \end{array}
    \right]_{(2m+4)\times(2m+4)}
  \]
  Also, using the definition in Section~\ref{subsection:nemethi-method}, it is
  routine to compute the following.
  \[
    \Bigl( \frac{\alpha_1}{\omega_1},\, \frac{\alpha_2}{\omega_2},\,
    \frac{\alpha_3}{\omega_3},\, \frac{\alpha_4}{\omega_4} \Bigr)
    = \Bigl( \frac{2m+1}{2m},\,\frac21,\,\frac31,\,\frac{m+1}{1} \Bigr).
  \]
  So, the orbifold Euler number is given by
  \[
    e = -2 + \frac{2m}{2m+1} + \frac{1}{2} + \frac{1}{3} + \frac{1}{m+1}
    = -\frac{(m-1)(2m-1)}{6(m+1)(2m+1)}.
  \]
  For $m>1$, we have $e<0$, so $X$ is negative definite.

  We will describe two \spinc\ structures $[k_1]$ and $[k_2]\in
  \Char(X)/2\lambda(H_2(X))$.  Let
  \begin{align*}
    k_1 &:= -2 b^*_{1,2m} - b^*_{3,1} + 2b^*_{4,1},
    \\ 
    k_2 &:= -2 b^*_{1,2m-1} + b^*_{3,1} - 2b^*_{4,1}.
  \end{align*}
  It is straightforward to show that $k_1$ and $k_2$ are distinguished
  representatives in the sense of
  Definition~\ref{definition:distinguished-representative}.  Indeed, in our
  case, the canonical \spinc\ structure described
  in~\eqref{equation:canonical-spinc-structure} is given by $K= b^*_{3,1} -(m+1)
  b^*_{4,1}$, and $k_1$ is of the form~\eqref{equation:k_r} where
  \[
    a_0 = 0,\quad (a_{1,1},\ldots,a_{1,2m}) = (0,\ldots,0,1), \quad
    a_{2,1}=0,\quad a_{3,1}=1,\quad a_{4,1}=1.
  \]
  By~\eqref{equation:a_ell}, we have $(a_0,a_1,a_2,a_3,a_4) =
  (0,1,0,1,\tfrac{m-3}{2})$.  This satisfies the conditions in
  Definition~\ref{definition:distinguished-representative}, so $k_1$ is a
  distinguished representative.  The class $k_2$ is shown to be a distinguished
  representative too, by similar computation.  In this case, we have
  \[
    a_0 = 0,\quad (a_{1,1},\ldots,a_{1,2m}) = (0,\ldots,0,1,0), \quad
    a_{2,1}=0,\quad a_{3,1}=0,\quad a_{4,1}=m-2
  \]
  and $(a_0,a_1,a_2,a_3,a_4) = (0,2,0,0,m-2)$.

  Under the adjoint $\lambda\colon H_2(X) \to H^2(X)=\Hom(H_2(X),\Z)$ of the
  intersection form of~$X$, $k_1$ and $k_2$ are respectively the images of
  \begin{align*}
    x_1 &= 2b_0 + 2b_{1,1}+\cdots+2b_{1,2m} + b_{2,1} + b_{3,1}, \\
    x_2 &= 4b_1 + 4b_{1,1}+\cdots+4b_{1,2m-1}+2b_{1,2m}
    + 2b_{2,1} + 2b_{3,1} + 2b_{4,1}.
  \end{align*}
  Recall that $\Char(X)\subset H^2(X)$ is identified with $\Spinc(X)$ via~$c_1$,
  and thus for a \spinc\ structure $[k] \in \Char(X)/2\lambda(H_2(X))$ of $Y$,
  we have $c_1([k]) = k|_Y$.  So, $c_1([k]) = 0$ if and only if $k$ lies in the
  kernel of $H^2(X) \to H^2(Y)$, or equivalently $k$ lies in the image of
  $\lambda\colon H_2(X) \to H^2(X)$.  From this observation, it follows that
  $c_1([k_i]) = 0$ for $i=1,2$, since $k_i = \lambda(x_i)$. Also, $[k_1] \ne
  [k_2]$ since $x_1-x_2\notin 2H_2(X)$.  Therefore, to complete the
  proof, it suffices to show $d(Y,[k_i]) \le (2m-1)/4$ for $i=1,2$.

  We have
  \[
    k_1^2 = \lambda(x_1,x_1) = -3, \quad k_2^2 = \lambda(x_2,x_2) = -m-4.
  \]
  The last thing we need is the minimum of the values of $\tau(i)$ defined
  in~\eqref{equation:tau(i)}.  Recall the notation $\Delta_i =
  \tau(i+1)-\tau(i)$ from Remark~\ref{remark:min-tau(i)}.
  
  \begin{assertion*}
    For both $k_1$ and $k_2$ and for all $i\ge 0$, $\Delta_i \ge 0$.
  \end{assertion*}

  We will provide a proof of the assertion for $m\ge 23$.  For $m<23$, the
  assertion is verified by direct inspection using
  Remark~\ref{remark:min-tau(i)} (indeed the author used a computer program), so
  we omit details for $m<23$.  Note that it suffices to use $m\ge 23$, to prove
  the main results of this paper, Theorems~\ref{theorem:main-top-slice}, \ref{theorem:main-independence-top-slice}, \ref{theorem:main-bipolar}
  and~\ref{theorem:main-independence-bipolar}\@.

  For the case of $k_1$, by Remark~\ref{remark:min-tau(i)}, we have $\Delta_i\ge
  0$ for $i\ge R$, where
  \[
    R=8+(48m-6)/(2m-1)(m-1).
  \]
  Since $m\ge 23$, $R\le 10$.  So, $\Delta_i\ge 0$ for $i\ge 10$. For $i\le 10$,
  using~\eqref{equation:tau(i)}, we have
  \begin{align*}
    \Delta_i &= 1 + 2i + \biggl[ \frac{-2mi+1}{2m+1} \bigg]
    + \biggl[ \frac{-i}{2} \bigg]
    + \biggl[ \frac{-i+1}{3} \bigg]
    + \biggl[ \frac{-i+(m-3)/2}{m+1} \bigg]
    \\
    & \ge 1+2i-i+\frac{-i-1}{2} + \frac{-i-1}{3} = \frac{i+1}{6} \ge 0.
  \end{align*}
  This shows the claim for $k_1$.  For the case of $k_2$, we proceed in the same
  way.  We have 
  \[
    R=7+(48m-6)/(2m-1)(m-1).
  \]
  Since $m\ge 23$, $R\ge 10$, and thus $\Delta_i\ge 0$ for $i\ge 10$ by
  Remark~\ref{remark:min-tau(i)}.  For $1\le i\le 10$,
  using~\eqref{equation:tau(i)}, we have
  \begin{align*}
    \Delta_i &= 1 + 2i + \biggl[ \frac{-2mi+1}{2m+1} \bigg]
    + \biggl[ \frac{-i}{2} \bigg]
    + \biggl[ \frac{-i}{3} \bigg]
    + \biggl[ \frac{-i+m-2}{m+1} \bigg]
    \\
    & \ge 1+2i-i+\frac{-i-1}{2} + \frac{-i-2}{3} = \frac{i-1}{6} \ge 0.
  \end{align*}
  For $i=0$, a direct computation using~\eqref{equation:tau(i)} gives
  $\Delta_0=1$.  This completes the proof of the assertion.

  From the assertion, it follows that $\min\{\tau(i)\mid i\ge 0\}=0$ since
  $\tau(0)=0$. So, by N\'emethi's Theorem~\ref{theorem:nemethi-formula},
  \[
    d(Y,[k_i]) = \frac{k_i^2 + 2m+4}{4} =
    \begin{cases}
      (2m-1)/4 &\text{for } i=1, \\
      m/4 &\text{for } i=2.
    \end{cases}
  \]
  Therefore $d(Y,[k_i]) \le (2m-1)/4$ holds for $i=1,2$.
\end{proof}

Now, to complete the proof of
Theorem~\ref{theorem:linear-combination-non-negativity}, it only remains to
prove of Lemma~\ref{lemma:d-invariant-B}, which estimates $d$-invariants of the
3-manifold~$B$ in Figure~\ref{figure:surgery-diagram-A-B}.  In the proof below,
we will use that $B$ is a Seifert 3-manifold as well.

\begin{proof}[Proof of Lemma~\ref{lemma:d-invariant-B}]
  
  Recall that $B$ is the 3-manifold described in
  Figure~\ref{figure:surgery-diagram-A-B}.  Let $\mathfrak{s}$ be a \spinc\
  structure of $B$ satisfying $c_1(\mathfrak{s})=0$.  The goal is to show that
  $d(B,\mathfrak{s})\le (m+2)/4$.

  As we will see in what follows, it will be useful to consider $-B$ instead
  of~$B$.  That is, we will show $d(-B,\mathfrak{s})\ge -(m+2)/4$.
  Figure~\ref{figure:plumbing-diagram-B} shows that $-B$ is a Seifert
  3-manifold. We use the notation of Section~\ref{subsection:nemethi-method}.
  The associated plumbed 4-manifold $X$ has $H_2(X)=\Z^{2m}$, with basis
  elements $b_0$, $b_{1,1}$, \dots,~$b_{1,m}$, $b_{2,1}$,
  \dots,~$b_{2,m-2}$,~$b_{3,1}$. The associated decorations are $e_0=-2$,
  $e_{1,j} = -2$ for all $j$, $e_{2,j}=-2$ for all $j$, and $e_{3,1}=-2m-1$. Let
  $b^*_0$, $b^*_{\ell,j}$ be the dual basis elements in~$H^2(X)$.
  Since
  \[
    e = -2 + \frac{m}{m+1} + \frac{m-2}{m-1} + \frac{1}{2m+1}
    = \frac{-3m^2 -2m -1}{(m^2-1)(2m+1)} < 0,
  \]
  the 4-manifold $X$ is negative definite.  This is why we use $-B$ instead
  of~$B$.

  \begin{figure}[H]
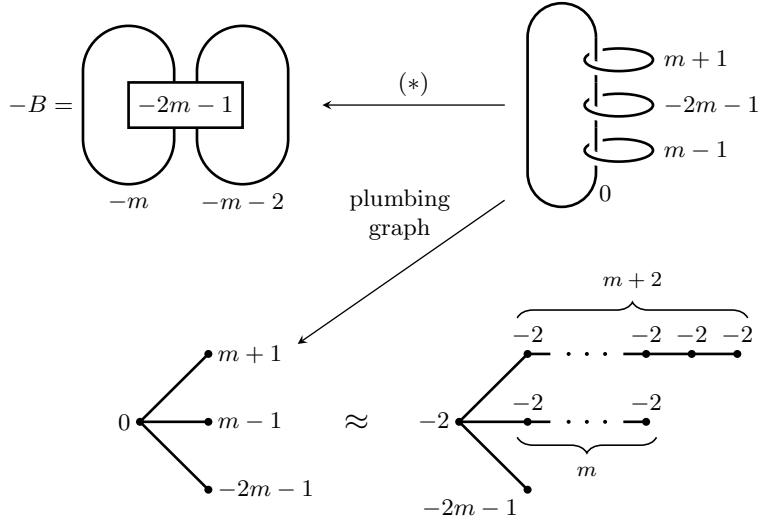

    \includestandalone{plumbing-diagram-B}
    \caption{Surgery diagram calculus showing that $-B$ is a Seifert 3-manifold. The symbol $(*)$ means handle slides (or Rolfsen twist~\cite{Rolfsen:1976-1}*{p.~264}) followed by elimination of a component together with a 0-framed meridian.}
    \label{figure:plumbing-diagram-B}
  \end{figure}

  The canonical \spinc\ structure $K$ of $X$ is given by $K=(2m-1)b^*_{3,1}$.
  Let
  \begin{align*}
    x_1 &= b_{1,1} + b_{1,3} + \cdots + b_{1,m} + b_{3,1}, \\
    x_2 &= b_{2,1} + b_{2,3} + \cdots + b_{2,m-2} + b_{3,1},
  \end{align*}
  and let $k_i = \lambda(x_i) \in H^2(X)$ for $i=1$,~$2$.  Then $k_i \in
  \Char(X)=K+2H^2(X)$, so that $[k_i] \in \Char(X)/2\lambda(H_2(X))=\Spinc(-B)$
  is a \spinc\ structure of~$-B$.  Similarly to the proof of
  Lemma~\ref{lemma:d-invariant-Y}, $c_1([k_i]) = k_i|_{-B} = 0$.  Also $[k_1]
  \ne [k_2]$ in $\Spinc(-B)$ since $x_1-x_2 \notin 2H_2(X)$.  It follows that
  $[k_1]$ and $[k_2]$ are the two \spinc\ structures of $-B$ with $c_1=0$.  So,
  it suffices to show that $d(-B,[k_i])\ge -(m+2)/4$ for $i=1$,~$2$.  Instead of
  determining the values exactly, we will present a simpler argument which gives
  the promised estimate. We have $k_1^2 = \lambda(x_1,x_1) = -3m-2$. So, by
  Ozsv\'ath-Szab\'o's inequality~\cite[Theorem~9.6]{Ozsvath-Szabo:2003-2},
  \begin{equation}
    \label{equation:-B-k_1-estimate}  
    d(-B,[k_1]) \ge \frac{k_1^2 + b_2(X)}{4} = \frac{-m-2}{4}.
  \end{equation}
  Similarly, since $k_2^2 = \lambda(x_2,x_2) = -3m$, we have
  \begin{equation}
    \label{equation:-B-k_2-estimate}  
    d(-B,[k_2]) \ge \frac{k_2^2 + b_2(X)}{4} = \frac{-m}{4}.
  \end{equation}
  So, we have $d(-B,[k_i]) \ge -(m+2)/4$ for $i=1$,~$2$.   (Indeed, it can be
  shown that the equality holds in~\eqref{equation:-B-k_1-estimate}
  and~\eqref{equation:-B-k_2-estimate}, by using the technique described in
  Section~\ref{subsection:nemethi-method}, as we did in the proof of
  Lemma~\ref{lemma:d-invariant-Y}.) 
\end{proof}

\appendix
\section*{Appendix. General primary decomposition}
\setcounter{section}{0}\stepcounter{section}
\label{section:general-primary-decomposition}

The goal of this appendix is to present an abstract formulation of the notion of
primary decomposition along an invariant with values in an unique factorization
domain.  We also discuss questions in specific cases and earlier related results
in the literature from our viewpoint.  The organization is as follows.  In
Section~\ref{subsection:pd-general-definition}, we describe the definition of
general primary decomposition and present basic observations. In
Section~\ref{subsection:pd-extensions}, we investigate primary decomposition of
extensions.  In Sections~\ref{subsection:pd-in-knot-concordance}, we discuss
specializations to various knot concordance groups (e.g. smooth/topological) and
related filtrations.  In Section~\ref{subsection:pd-in-rational-H-cob}, we
discuss the case of rational homology cobordism group of rational homology
3-spheres.

\subsection{Definitions and basic observations}
\label{subsection:pd-general-definition}

Let $\bK$ be an abelian monoid and $\sim$ is an equivalence relation on~$\bK$.
Suppose that the monoid structure on $\bK$ descends to an abelian group
structure on the set $\cC := \bK/{\sim}$ of equivalence classes. Denote the
equivalence class of $K\in \bK$ by $[K] \in \cC$.

Let $R$ be a unique factorization domain with involution $r \mapsto r^*$. Main
examples are $\Z$ with a trivial involution, and the Laurent polynomial ring
$\Q[t^{\pm1}]$ with the standard involution $\bigl(\sum a_i t^i\bigr)^* = \sum
a_i t^{-i}$.  For $r$, $s\in R$, write $r\doteq s$ if $r$ and $s$ in $R$ are
associates; that is, $r=us$ for some unit $u$ in~$R$. We say that two
irreducibles $\lambda$ and $\mu$ in $R$ are \emph{$*$-associates} if either
$\lambda\doteq \mu$ or $\lambda^*\doteq \mu$.  We say that $r\in R$ is
\emph{self-dual} if $r\doteq r^*$.

Suppose $\chi \colon \bK \to (R\sm\{0\}) / {\doteq}$ is a function.  We will
denote a representative of $\chi(K)$ by~$\Delta_K\in R\sm\{0\}$.  Suppose the
following hold for all $K$, $K'$ in~$\bK$:
\begin{enumerate}[label=($\Delta$\arabic*)]
  \item\label{item:delta-symmetry}
  $\Delta_K$ is self-dual.
  \item\label{item:delta-addivitity}
  $\Delta_{K+K'} \doteq \Delta_K\cdot \Delta_{K'}$.
  \item\label{item:delta-negation}
  $-[K] = [J]$ for some $J$ in $\bK$ such that $\Delta_J
  \doteq \Delta_K$.
\end{enumerate}
In particular, writing $\cS(R)=\{r\in R\mid r\ne 0$ is self-dual$\}$, $\bK
\xrightarrow{\chi} \cS(R)/{\doteq}$ is a homomorphism between abelian
monoids.

Of course, the main example which one may keep in mind is the case that $\bK$ is
the monoid of knots under connected sum, ${\sim}$ is concordance and $\Delta_K$
is the Alexander polynomial in $R=\Q[t^{\pm1}]$.  Including this, we will
discuss various specific examples in
Sections~\ref{subsection:pd-in-knot-concordance}
and~\ref{subsection:pd-in-rational-H-cob}.

Let $\Delta = \{[K] \in \cC \mid \Delta_K \doteq 1 \}$.  For an irreducible
$\lambda$ in $R$, let $S(\lambda)=\lambda$ if $\lambda$ is self-dual, and
$S(\lambda)=\lambda\lambda^*$ otherwise.   Let
\begin{align*}
  \cC_\lambda &:= \{[K] \in \cC \mid \Delta_K \doteq
  S(\lambda)^k \text{ for some } k\ge 0\},
  \\
  \cC^\lambda &:= \{[K] \in \cC \mid \Delta_K
  \text{ is relatively prime to } \lambda\}.
\end{align*}
It is straightforward to verify that $\cC_\lambda$, $\cC^\lambda$ and $\Delta$
are subgroups of $\cC$, using~\ref{item:delta-addivitity}
and~\ref{item:delta-negation}. Also, using~\ref{item:delta-symmetry}, it is seen
that $[K]\in \cC_\lambda$ if and only if $K\sim J$ for some $J$ with $\Delta_J$
lying in the multiplicative subset generated by $\lambda$, $\lambda^*$ and the
units in~$R$. We have $\Delta\subset \cC_\lambda$, $\cC_\lambda=\cC_{\lambda^*}$
and $\cC^\lambda=\cC^{\lambda^*}$. Use $S(\lambda)=S(\lambda^*)$
and~\ref{item:delta-symmetry} to verify the two equalities respectively.

Note that while $\Delta_K$ is self-dual, irreducible factors of $\Delta_K$ are
not necessarily self-dual. For instance, in $R=\Q[t^{\pm1}]$, we have
$-2t+5-2t^{-1}=(t-2)(t^{-1}-2)$.  By the above definition, a class $[K]$ with
$\Delta_K=-2t+5-2t^{-1}$ lies in $\cC_\lambda$ for $\lambda=t-2$.

\begin{definition}
  \label{definition:weak-pd}
  Let $\bP$ be the set of $*$-associate classes of irreducibles in~$R$.

  \begin{enumerate}

    \item 
    We say that $(\bK,\sim,\chi)$ is \emph{left primary decomposable} if the sum
    \[
      \Phi_L\colon
      \bigoplus_{[\lambda]\in \bP} \cC_\lambda/\Delta \to \cC/\Delta
    \]
    of the inclusions $\cC_\lambda/\Delta \hookrightarrow \cC/\Delta$ is an
    isomorphism.

    \item
    We say that $(\bK,\sim,\chi)$ is \emph{right primary decomposable} if the
    surjections $\cC/\Delta \twoheadrightarrow \cC/\cC^\lambda$ induce an
    isomorphism
    \[
      \Phi_R\colon
      \cC/\Delta \to \bigoplus_{[\lambda]\in\bP} \cC/\cC^\lambda.
    \]

  \end{enumerate}  
\end{definition}

Since each $\Delta_K$ has finitely many irreducible factors, it follows that the
product $\cC/\Delta \to \prod_{[\lambda]\in\bP} \cC/\cC^\lambda$ of the
surjections $\cC/\Delta \twoheadrightarrow \cC/\cC^\lambda$ has image in the
direct sum~$\bigoplus_{[\lambda]\in\bP}\cC/\cC^\lambda$.  That is, $\Phi_R$ is
always a well-defined homomorphism.

We remark that taking the quotient by $\Delta$ may be viewed as an
analogue of ignoring units in the primary decomposition in a unique
factorization domain. 

For brevity, when the choice of $(\bK,\sim,\chi)$ is clearly understood from the
context, we will simply say that $\cC$ is left or right primary decomposable.

From the definition, it is straightforward to see that $\cC$ is left primary
decomposable if only if the following two conditions
\ref{item:left-pd-existence} and \ref{item:left-pd-uniqueness} hold. More
precesly, \ref{item:left-pd-existence} and \ref{item:left-pd-uniqueness} are
respectively equivalent to the surjectivity and injectivity of~$\Phi_L$.

\begin{enumerate}[label=(LP\arabic*)]
  \item\label{item:left-pd-existence}
  Existence: for every $K\in \bK$, there exist irreducibles $\lambda_1$,
  \dots,~$\lambda_n$ and $[K_1]\in \cC_{\lambda_1},\ldots,[K_n]\in
  \cC_{\lambda_n}$ such that $[K] \equiv [K_1]+\cdots+[K_n]\bmod \Delta$.
  \item\label{item:left-pd-uniqueness}
  Uniqueness: if $[K_1]+\cdots+[K_n] \equiv 0 \bmod \Delta$ and $[K_i]\in
  \cC_{\lambda_i}$ for some pairwise distinct $[\lambda_1]$, \dots,~$[\lambda_n]
  \in \bP$, then $[K_i]\equiv 0 \bmod \Delta$ for all~$i$.
\end{enumerate}

The following examples illustrate that the left and right primary
decomposabilities are independent of each other.

\begin{example}
  \label{example:left-pd-but-not-right}
  Let
  \[
    \bK = \{aK+bK'+cJ+dJ'\mid a,b,c,d\ge 0\}\cong (\Z_{\ge 0})^4
  \]
  be the free abelian monoid generated by four generators $K$, $K'$, $J$
  and~$J'$.  Define $\sim$ on $\bK$ by
  \[
    aK+bK'+cJ+dJ' \sim pK+qK'+rJ+sJ' \Longleftrightarrow
    a-b+c-d = p-q+r-s
  \]
  and let $\cC = \bK/{\sim}$.  Then $\cC$ is the infinite cyclic group, and
  $[K]=-[K']=[J]=-[J']$ is a generator.  Fix $R$ which has three distinct
  self-dual irreducibles $\lambda$, $\mu$ and~$\nu$.  Define
  \[
    \Delta_{aK+bK'+cJ+dJ'} = \lambda^{a+b}(\mu\nu)^{c+d}.
  \]
  In particular $\Delta_K = \lambda$ and $\Delta_J=\mu\nu$.  It is
  straightforward to verify that \ref{item:delta-symmetry},
  \ref{item:delta-addivitity} and~\ref{item:delta-negation} are satisfied, and
  that the subgroup $\Delta$ is trivial.

  We claim that $\cC$ is left primary decomposable.  Indeed, it is
  straightforward to see that for $L\in \bK$, $\Delta_L$ is a power of an
  irreducible if and only if $L=aK+bK'$ (and thus $\Delta_L=\lambda^{a+b}$).  It
  follows that for any irreducible $\zeta$ in $R$, $\cC_\zeta=\cC$ if
  $\zeta=\lambda$, and $\cC_\zeta=0$ otherwise.  So $\Phi_L\colon
  \bigoplus_{[\zeta]} \cC_\zeta/\Delta \to \cC/\Delta$ is an isomorphism. 

  On the other hand, $\cC$ is not right primary decomposable.  To see this,
  observe that $[K]\in \cC^\zeta$ for all $\zeta\ne \lambda$,  since
  $\Delta_K=\lambda$.  Also $[K]=[J]\in \cC^\lambda$ since $\Delta_J=\mu\nu$.
  Since $[K]$ generates $\cC$, it follows that $\cC^\zeta=\cC$ for all~$\zeta$.
  So $\Phi_R\colon \cC/\Delta \to \bigoplus_{[\zeta]} \cC/\cC^\zeta=0$ is
  not injective.
\end{example}

\begin{example}
  \label{example:right-pd-but-not-left}
  Let $\bK$, $\sim$, $\cC$, $R$, $\lambda$, $\mu$ and $\nu$ be as in
  Example~\ref{example:left-pd-but-not-right}, but define
  \[
    \Delta_{aK+bK'+cJ+dJ'} = (\lambda\mu)^{a+b}(\mu\nu)^{c+d}.
  \]
  Note that \ref{item:delta-symmetry}, \ref{item:delta-addivitity}
  and~\ref{item:delta-negation} are satisfied, $\Delta_K = \lambda\mu$,
  $\Delta_J=\mu\nu$ and $\Delta$ is trivial.

  First we will show that $\cC$ is not left primary decomposable.  For $L\in
  \bK$, either $L=0$ or $\Delta_L$ is not a power of $S(\zeta)$ for any
  irreducible $\zeta$.  Therefore $\cC_\zeta=0$ for all~$\zeta$.  It follows
  that $\Phi_L\colon \bigoplus_{[\zeta]} \cC_\zeta/\Delta = 0 \to \cC/\Delta$ is
  not surjective.

  On the other hand, $\cC$ is right primary decomposable.  To prove this,
  observe that $\Delta_L$ is relatively prime to $\mu$ if and only if $L=0$.
  That is, $\cC^\mu=0$.  Also, $[K] \in \cC^\zeta$ for $\zeta\ne \lambda,\mu$,
  and $[J] \in \cC^\zeta$ for $\zeta\ne\mu,\nu$.  Since $\cC$ is generated by
  $[K]=[J]$, it follows that $\cC^\zeta=\cC$ for $\zeta\ne \mu$.  Therefore
  $\Phi_R\colon \cC/\Delta \to \bigoplus_{[\zeta]} \cC/\cC^\zeta$ is an
  isomorphism.
\end{example}

While the above examples tell us that the bijectivity of $\Phi_L$ does not imply
the bijectivity of $\Phi_R$ nor vice versa, the following observations provide
partial relationships between the surjectivity and injectivity of $\Phi_L$
and~$\Phi_R$.

\begin{lemma}
  \phantomsection\leavevmode\Nopagebreak
  \label{lemma:surj-inj-Phi_L-Phi_R}
  \begin{enumerate}
    \item\label{item:surj-Phi_L-Phi_R}
    If $\Phi_L$ is surjective, then $\Phi_R$ is surjective.
    \item\label{item:inj_Phi_L-Phi_R}
    Suppose $\Phi_R$ is injective. Then the following splitting property holds:
    if $\Delta_K$ and $\Delta_J$ are relatively prime and $[K]+[J]\equiv0 \bmod
    \Delta$, then $[K]\equiv[J]\equiv0 \bmod \Delta$. In particular, $\Phi_L$ is
    injective.
  \end{enumerate}
\end{lemma}

\begin{proof}
  \ref{item:surj-Phi_L-Phi_R} The homomorphism $\Phi_R$ is surjective if and
  only if for every given $[K]\in \cC$ and $[\lambda]\in\bP$, there exists $[J]
  \in \cC$ such that $[J] \equiv [K] \bmod \cC^{\lambda}$ and $[J] \equiv 0
  \bmod \cC^\mu$ for all $[\mu] \ne [\lambda]$. By the surjectivity of $\Phi_L$
  or equivalently~\ref{item:left-pd-existence}, $[K] \equiv [K_1]+\cdots+[K_n]
  \bmod \Delta$ for some $[K_i]\in \cC_{\lambda_i}$, where $[\lambda_i]\ne
  [\lambda_j]$ for $i\ne j$.  If $[\lambda]=[\lambda_i]$ for some~$i$, then let
  $[J] = [K_i]$. Otherwise, let $[J]=0$.  Then, since $\cC_{\lambda_i} \subset
  \cC^{\mu}$ whenever $[\mu] \ne [\lambda_i]$, it follows that $[J]$ has the
  desired property.

  \ref{item:inj_Phi_L-Phi_R} Suppose $\gcd(\Delta_K,\Delta_J)=1$ and
  $[K]+[J]\equiv 0 \bmod \Delta$. If an irreducible $\lambda$
  divides~$\Delta_J$, then $\gcd(\lambda,\Delta_K)=1$ and so $[K] \equiv 0 \bmod
  \cC^\lambda$.  If $\gcd(\lambda,\Delta_J)=1$, then $[K] \equiv [K]+[J] \equiv
  0 \bmod \cC^\lambda$. So $[K]\equiv 0\bmod \cC^\lambda$ for all $\lambda$, and
  thus $[K]\equiv 0\bmod \Delta$ by the injectivity of~$\Phi_R$.  The same
  argument shows $[J]\equiv 0 \bmod \Delta$.  From this,
  \ref{item:left-pd-uniqueness} immediately follows.  That is, $\Phi_L$ is
  injective.
\end{proof}

\begin{remark}
  As seen in the last sentence of the above proof of
  Lemma~\ref{lemma:surj-inj-Phi_L-Phi_R}\ref{item:surj-Phi_L-Phi_R}, the
  composition $\cC_\lambda/\Delta \hookrightarrow \cC/\Delta \twoheadrightarrow
  \cC/\cC^\mu$ is zero if $[\lambda]\ne [\mu]$. Consequently, the composition
  \[
    \Phi_R \circ \Phi_L \colon \bigoplus_{[\lambda]\in \bP} \cC_\lambda/\Delta
    \to \bigoplus_{[\lambda]\in \bP} \cC/\cC^\lambda
  \]
  is the orthogonal direct sum of the compositions $\cC_\lambda/\Delta
  \hookrightarrow \cC/\Delta \twoheadrightarrow \cC/\cC^\lambda$. It follows
  that $\cC_\lambda/\Delta \hookrightarrow \cC/\Delta \twoheadrightarrow
  \cC/\cC^\lambda$ is an isomorphism for each irreducible~$\lambda$ if and only
  if $\Phi_R\circ \Phi_L$ is an isomorphism.  Consequently $\cC_\lambda/\Delta
  \hookrightarrow \cC/\Delta \twoheadrightarrow \cC/\cC^\lambda$ is an
  isomorphism if $\cC$ is left primary decomposable and right primary
  decomposable.
\end{remark}

Recall that $\cC$ is left primary decomposable if and only if
\ref{item:left-pd-existence} and \ref{item:left-pd-uniqueness} hold.  The
following is a stronger variation of~\ref{item:left-pd-existence}.

\begin{enumerate}[label=(LP\arabic*$'$)]
  \item\label{item:left-pd-strong-existence}
  Strong existence: for every $K\in \bK$, there exist irreducible factors
  $\lambda_1$, \dots,~$\lambda_n$ of $\Delta_K$ and $[K_1]\in
  \cC_{\lambda_1},\ldots,[K_n]\in \cC_{\lambda_n}$ such that $[K] \equiv
  [K_1]+\cdots+[K_n]\bmod \Delta$.
\end{enumerate}

\begin{definition}
  \label{definition:strong-pd}
  We say that $\cC=\bK/{\sim}$ is \emph{strongly primary decomposable}
  if~\ref{item:left-pd-strong-existence} and~\ref{item:left-pd-uniqueness} hold.
\end{definition}

\begin{proposition}
  \label{proposition:strong-pd-implies-2-sided}
  If $\cC$ is strongly primary decomposable, then $\cC$ is left primary
  decomposable and right primary decomposable.
\end{proposition}

\begin{proof}
  If $\cC$ is strongly primary decomposable,  $\cC$ is obviously left primary
  decomposable.  Also, $\Phi_R$ is surjective by
  Lemma~\ref{lemma:surj-inj-Phi_L-Phi_R}\ref{item:surj-Phi_L-Phi_R}.  So it
  remains to show that $\Phi_R$ is injective.
  
  Suppose that $[K]$ lies in $\cC^\lambda$ for all irreducibles~$\lambda$.  The
  goal is to show that $[K]\in \Delta$.  We claim that the given $[K]$ can be
  assumed to lie in~$\cC_{\mu}$ for some~$[\mu]\in\bP$.  To prove this, first
  use~\ref{item:left-pd-existence} (or the surjectivity of $\Phi_L$) to write $[K] \equiv
  [K_1]+\cdots+[K_n] \bmod \Delta$ for some $K_i\in \cC_{\lambda_i}$, where
  $[\lambda_i]\ne [\lambda_j]$ for $i\ne j$. Fix~$i$. We have that $[K_i]
  \equiv [K]-\sum_{j\ne i} [K_j] \bmod \Delta$, $[K_j]\in \cC_{\lambda_j}
  \subset \cC^{\lambda_i}$ for all $j\ne i$, and $[K]\in \cC^{\lambda_i}$.  It
  follows that $[K_i] \in \cC^{\lambda_i}$.  Since $[K_i]\in \cC_{\lambda_i}$,
  $[K_i]\in \cC^{\lambda}$ for $[\lambda]\ne [\lambda_i]$.  So, $[K_i]$ lies
  in~$\cC^\lambda$ for all~$\lambda$. Note that it suffices to show that
  $[K_i]\in\Delta$ for all~$i$, to conclude that $[K]\in \Delta$. This proves
  the claim.
    
  Now, fix an irreducible~$\mu$.  Suppose $[K] \in \cC_\mu$ and $[K]\in
  \cC^\lambda$ for all~$\lambda$. Since $[K]\in \cC^{\mu}$, we have $[K]=[J]$
  for some $J\in \bK$ such that $\gcd(\Delta_J,\mu)=1$.
  By~\ref{item:left-pd-strong-existence}, $[J] \equiv [J_1]+\cdots[J_m] \bmod
  \Delta$ for some $[J_k]\in \cC_{\mu_k}$, where the $\mu_k$ are irreducible
  factors of~$\Delta_J$. In particular, every $\mu_k$ is relatively prime
  to~$\mu$.  Also, $\mu_k^*$ is relatively prime to $\mu$ too, since $\mu_k^*$
  divides $\Delta_J^* \doteq \Delta_J$ which is relatively prime to~$\mu$. So
  $[\mu_k] \ne [\mu]$ in $\bP$ for all~$k$. Since $-[K] + [J_1] + \cdots +
  [J_m]\in \Delta$ and $[K]\in \cC_\mu$, it follows that $[K]$ lies in $\Delta$
  by~\ref{item:left-pd-uniqueness}.
\end{proof}

\subsection{Extensions}
\label{subsection:pd-extensions}

Suppose that $(\bK,\sim,\chi)$ is as in
Section~\ref{subsection:pd-general-definition}, and $\cA$ is a subgroup
of~$\cC=\bK/{\sim}$. We have $\cA=\bA/{\sim}$ for the submonoid $\bA:=\{K\in
\bK\mid [K]\in \cA\}$.  Or conversely, if $\bA$ is a submonoid of $\bK$ such
that $K+J\sim 0$ and $K\in \bA$ imply $J\in \bA$, then $\cA := \bA/{\sim}$ is a
subgroup of~$\cC$.  Let $\cG=\cC/\cA$ be the quotient group.  We have $\cG =
\bK/{\approx}$, where the equivalence relation $\approx$ is defined by $K\approx
J$ if and only if $K\sim J+L$ for some $L\in \bA$.

We will relate primary decompositions of $\cA$, $\cC$ and~$\cG$.  To avoid
confusion, for $\cA$, $\cC$ and~$\cG$ respectively, denote by $\Delta(\cA)$,
$\Delta(\cC)$ and~$\Delta(\cG)$ their subgroups of classes represented by $K$
with $\Delta_K\doteq 1$.  We have $\Delta(\cA)=\Delta(\cC)\cap \cA$ and
$\Delta(\cG) = (\Delta(\cC)+\cA)/\cA \subset \cG=\cC/\cA$.  
It is straightforward to verify that the exact sequence
\[
  0 \to \cA \to \cC \to \cG \to 0
\]
gives rise to an exact sequence
\begin{gather*}
  0 \to \Delta(\cA) \to \Delta(\cC) \to \Delta(\cG) \to 0
\end{gather*}
and induces the following exact sequences for all irreducibles~$\lambda$
in~$\Q[t^{\pm1}]$.
\begin{gather*}
  0 \to \cA_\lambda \to \cC_\lambda \to \cG_\lambda \to 0
  \\
  0 \to \cA^\lambda \to \cC^\lambda \to \cG^\lambda \to 0
\end{gather*}
Consequently, rows of the following commutative diagram are exact.
\begin{equation}
  \label{equation:extension-pd-diagram}
  \def\bop{\bigoplus\limits_{[\lambda]\in\bP}}
  \def\sbop{\smash{\bop}}
  \begin{tikzcd}[
    row sep={9.5ex,between origins},
    column sep={9.5ex,between origins}
    ]
    0 \ar[rr]
    &&[-2.8em] \bop \cA_\lambda/\Delta(\cA)
    \ar[rr] \ar[d,"\Phi^\cA_L",shorten <=-2ex,pos=0]
    && \bop \cC/\Delta(\cC)
    \ar[rr] \ar[d,"\Phi^\cC_L",shorten <=-2ex,pos=0]
    && \bop \cG_\lambda/\Delta(\cG)
    \ar[rr] \ar[d,"\Phi^\cG_L",shorten <=-2ex,pos=0]
    &&[-3em] 0 
    \\
    0 \ar[rr]
    && \cA/\Delta(\cA) \ar[rr] \ar[d,"\Phi^\cA_R"]
    && \cC/\Delta(\cC) \ar[rr] \ar[d,"\Phi^\cC_R"]
    && \cG/\Delta(\cG) \ar[rr] \ar[d,"\Phi^\cG_R"]
    && 0
    \\
    0 \ar[rr]
    && \bop \cA/\cA^\lambda \ar[rr]
    && \bop \cC/\cC^\lambda \ar[rr]
    && \bop \cG/\cG^\lambda \ar[rr]
    && 0
  \end{tikzcd}
\end{equation}
For $\bullet=L$ and $R$, \eqref{equation:extension-pd-diagram} gives rise to an
exact sequence
\begin{equation}
  \label{equation:extension-snake}
  \begin{aligned}
  0 \to \Ker \Phi^\cA_\bullet
  & \to \Ker \Phi^\cC_\bullet \to \Ker \Phi^\cG_\bullet 
  \\
  &\to \Coker \Phi^\cA_\bullet
  \to \Coker \Phi^\cC_\bullet \to \Coker \Phi^\cG_\bullet \to 0
  \end{aligned}
\end{equation}
by the snake lemma.  It follows that $\Phi^\cC_\bullet$ is an isomorphism if
and only if $\Phi^\cA_\bullet$ is injective, $\Phi^\cG_\bullet$ is surjective
and the connecting map $\Ker \Phi^\cG_\bullet \to \Coker \Phi^\cA_\bullet$ is an
isomorphism.  The following is an immediate consequence.

\begin{theorem}
  \label{theorem:extension-pd}
  The group $\cC$ is left (respectively right) primary decomposable if both
  $\cA$ and $\cG$ are left (respectively right) primary decomposable.
\end{theorem}


\begin{theorem}
  \label{theorem:extension-strong-pd}
  The group $\cC$ is strongly primary decomposable if both $\cA$ and $\cG$ are
  strongly primary decomposable.
\end{theorem}

\begin{proof}
  Suppose $\cA$ and $\cG$ are strongly primary decomposable.  Then by
  Proposition~\ref{proposition:strong-pd-implies-2-sided} and
  Theorem~\ref{theorem:extension-pd}, $\cC$ satisfies the uniqueness
  condition~\ref{item:left-pd-uniqueness}. So, it remains to verify the strong
  existence condition~\ref{item:left-pd-strong-existence} for~$\cC$.
  
  Fix $[K]\in \cC$.  Since $\cG$ satisfies~\ref{item:left-pd-strong-existence},
  we have $[K] \equiv [J_1]+\cdots+[J_n] \bmod \Delta(\cC)+\cA$ for some
  $[J_1]\in \cC_{\lambda_1}$,~\dots, $[J_n]\in \cC_{\lambda_n}$, where each
  $\lambda_i$ is an irreducible factor of~$\Delta_K$ and $\Delta_{J_i}$ is a
  power of~$S(\lambda_i)$. So, for some $[J]\in \cC$ with $\Delta_J \doteq 1$,
  we have
  \begin{equation}
    \label{equation:first-decomp}
    [K] -([J_1]+\cdots+[J_n]) +[J] \in \cA.
  \end{equation}
  Using property~\ref{item:delta-negation}, choose $J_1'$,~\dots, $J_n'$ such
  that $[J_i']=-[J_i]$ and $\Delta_{J_i'} \doteq \Delta_{J_i}$, and let
  $L=K+J_1'+\cdots+J_n'+J$. Then $[L]\in \cA$ by~\eqref{equation:first-decomp},
  and $\Delta_L \doteq \Delta_K\cdot \prod \Delta_{J_i}$
  by~\ref{item:delta-addivitity}.  In particular, irreducibles dividing
  $\Delta_L$ divides~$\Delta_K$. Since $\cA$
  satisfies~\ref{item:left-pd-strong-existence}, there is a decomposition
  \begin{equation*}
    [L] \equiv [L_1]+\cdots+[L_m] \mod \Delta(\cA) = \cA \cap \Delta(\cC),
  \end{equation*}
  where $[L_j] \in \cA_{\mu_j}$ and each $\mu_j$ is an irreducible factor
  of~$\Delta_K$.  It follows that
  \[
    [K] \equiv [J_1]+\cdots+[J_n] + [L_1]+\cdots+[L_m] \mod \Delta(\cC).
  \]
  Recall that $[J_i] \in \cC_{\lambda_i}$, $[L_j]\in \cA_{\mu_j} \subset
  \cC_{\mu_j}$ and $\lambda_i$ and $\mu_j$ are factors of~$\Delta_K$.
  So,~\ref{item:left-pd-strong-existence} is satisfied for the given~$[K]\in
  \cC$.
\end{proof}

\subsection{Knot concordance and primary decomposition}
\label{subsection:pd-in-knot-concordance}

\subsubsection*{Algebraic concordance over $\Q$}

Levine's work on knot concordance provides an algebraic analogue of the knot
concordance group, which is now called the algebraic concordance
group~\cites{Levine:1969-1,Levine:1969-2}.  The algebraic concordance group over
$\Q$ is known to be left and right primary decomposable in our sense.  We
describe this using Blanchfield linking forms over~$\Q[t^{\pm1}]$, while
Levine's original papers~\cites{Levine:1969-1,Levine:1969-2} use Seifert
matrices.

Let $\Q(t)$ be the rational function field.  A $(\Q(t)/\Q[t^{\pm1}])$-valued
\emph{linking form} is defined to be a map $B\colon V\times V\to
\Q(t)/\Q[t^{\pm1}]$, with $V$ a finitely generated $\Q[t^{\pm1}]$-module such
that $V\otimes_{\Q[t^{\pm1}]} \Q(t)=0$, which is sesquilinear and nonsingular.
That is, $B(x,y)=B(y,x)^*$, $y\mapsto B(x,-)$ is $\Q[t^{\pm1}]$-linear for all
$x\in V$, and the adjoint $V\to \Hom_{\Q[t^{\pm1}]}(V,\Q(t)/\Q[t^{\pm1}])$ is an
isomorphism. 
Let $\bK$ be the monoid of $(\Q(t)/\Q[t^{\pm1}])$-valued linking forms (under direct sum operation), and define $\chi\colon \bK \to \Q[t^{\pm1}]/{\doteq}$ by
$\chi(B)=\Delta_B(t)$, where the Alexander polynomial $\Delta_B(t)$ is defined
to be the order of the torsion module $V$ over~$\Q[t^{\pm1}]$.
A linking form $B$ is \emph{metabolic} if there is a submodule $P\subset V$ such
that $P=\{x\in V\mid B(x,P)=0\}$.  We say that $B$ and $B'$ are Witt equivalent
if there are metabolic linking forms $H$ and $H'$ such that the orthogonal sums
$B\oplus H$ and $B'\oplus H$ are isomorphic. The set $\cG$ of Witt equivalence
classes of linking forms is an abelian group under orthogonal sum.\footnote{For
the study of knots, often the Witt group of linking forms $B$ satisfying
$\Delta_B(1)\Delta_B(-1)\ne 0$ is considered.  One can also consider
$(-1)$-linking forms, which satisfy $B(x,y)=-B(y,x)^*$ instead of
$B(x,y)=B(y,x)^*$. Variations of
Theorem~\ref{theorem:algebraic-concordance-primary-decomposition} holds for
these cases as well.} The conditions~\ref{item:delta-symmetry},
\ref{item:delta-addivitity} and \ref{item:delta-negation} stated in
Section~\ref{subsection:pd-general-definition} hold.  By definition,
$\Delta_B(t)$ is trivial if and only if $V$ is trivial.  It follows that the
subgroup $\Delta=\{[B]\in \cG \mid \Delta_B \doteq 1\}$ is trivial.

The following result is essentially due to Levine~\cite{Levine:1969-2}.

\begin{theorem}[Levine~\cite{Levine:1969-2}]
  \label{theorem:algebraic-concordance-primary-decomposition}
  The group $\cG$ is left and right primary decomposable.
  \[
    \bigoplus_{[\lambda]} \cG_\lambda \xrightarrow[\cong]{\Phi_L} 
    \cG \xrightarrow[\cong]{\Phi_R} \bigoplus_{[\lambda]} \cG/\cG^\lambda
  \]
\end{theorem}

Indeed, $\cG$ is strongly primary decomposable.  An elegant generalization of
Levine's Theorem~\ref{theorem:algebraic-concordance-primary-decomposition} to
the case of the group ring of a (noncommutative) free group was developed in
work of Sheiham~\cite[Section~3]{Sheiham:2006-1}.  See
also~\cite{Sheiham:2001-1}.

\subsubsection*{Algebraic concordance over~$\Z$}

For a knot $K$ in $S^3$, the Blanchfield form~\cite{Blanchfield:1957-1} $B_K$
defined on the Alexander module $H_1(S^3\sm K;\Q[t^{\pm1}])$ is a linking form
in the above sense.  We have $\Delta_K = \Delta_{B_K}$.  The association
$K\mapsto B_K$ induces a homomorphism of the (topological and smooth) knot
concordance group into $\cG$.  The image of this homomorphism is characterized
as follows.  By replacing $(\Q[t^{\pm1}],\Q(t))$ with
$(\Z[t^{\pm1}],S^{-1}\Z[t^{\pm1}])$ where $S=\{f(t)\in \Z[t^{\pm1}] \,|\,
f(1)=\pm1\}$, one constructs an integral analogue of $\cG$, say~$\cG(\Z)$.  It
is known that the natural map $\cG(\Z)\to \cG$ is
injective~\cite{Levine:1969-2}, and the subgroup $\cG(\Z)$ of $\cG$ is exactly
the image of the concordance group of knots in~$S^3$.  A knot $K$ such that
$[B_K]=0$ in $\cG(\Z) \subset \cG$ is said to be \emph{algebraically slice}.

Levine showed that both $\cG$ and $\cG(\Z)$ are isomorphic to $\Z^\infty \oplus
(\Z_2)^\infty \oplus (\Z_4)^\infty$~\cite{Levine:1969-2}.  The difference of the
structures of $\cG$ and $\cG(\Z)$ was studied in work of
Stoltzfus~\cite{Stoltzfus:1977-1}. Among his main results, it was shown that
$\Phi_L\colon \bigoplus_{[\lambda]} \cG(\Z)_\lambda \to \cG(\Z)$ is not
surjective (while it is injective due to Levine's work~\cite{Levine:1969-1}).
In particular, $\cG(\Z)$ is not left primary decomposable. (To the knowledge of
the author, it was not addressed in the literature whether $\cG(\Z)$ is right
primary decomposable.)

\subsubsection*{Topological knot concordance and algebraically slice knots}

Let $\bK$ be the set of isotopy classes of oriented knots in~$S^3$, with
connected sum as a monoid operation.  If $K$, $K' \in \bK$ are topologically
concordant, write ${K\sim K'}$. The set of equivalence classes $\cC^\top =
\bK/{\sim}$ is the topological knot concordance group. Define $\chi\colon\bK \to
\Q[t^{\pm1}]/{\doteq}$ by $\chi(K) = \Delta_K(t)$, the Alexander polynomial
of~$K$.  The assumptions~\ref{item:delta-symmetry}, \ref{item:delta-addivitity}
and \ref{item:delta-negation} in Section~\ref{subsection:pd-general-definition}
are standard properties of the Alexander polynomial.  By work of
Freedman~\cite{Freedman:1984-1}, the subgroup $\Delta=\{[K]\in \cC^\top \mid
\Delta_K \doteq 1\}$ is trivial.

Let $\cA^\top$ be the topological concordance group of algebraically slice
knots.  Then, since
\[
  0\to \cA^\top \to \cC^\top \to \cG^\Z \to 0
\]
is exact, $\cC^\top$ is right primary decomposable if so are both $\cA^\top$
and~$\cG^\Z$.  Regarding left primary decomposability, since
$\bigoplus_{[\lambda]} \cG^\Z_\lambda \to \cG^\Z$ is not surjective due to
Stoltzfus~\cite{Stoltzfus:1977-1}, $\Phi_L^\top\colon \bigoplus_{[\lambda]}
\cC^\top_\lambda \to \cC^\top$ is not surjective by the exact
sequence~\eqref{equation:extension-snake}. In particular, $\cC^\top$ is not left
primary decomposable.  The following appear to be interesting.

\begin{question}
  \phantomsection\leavevmode\Nopagebreak
  \label{question:pd-top-conc}
  \begin{enumerate}
    \item\label{item:top-slice-injectivity}
    Is $\Phi_L^\top\colon \bigoplus_{[\lambda]} \cC^\top_\lambda \to \cC^\top$
    injective?
    \item\label{item:right-pd-top-conc} Is $\cC^\top$ right primary
    decomposable?
  \end{enumerate}
\end{question}

By Lemma~\ref{lemma:surj-inj-Phi_L-Phi_R}\ref{item:inj_Phi_L-Phi_R}, an
affirmative answer to
Question~\ref{question:pd-top-conc}\ref{item:right-pd-top-conc} implies that the
answer to Question~\ref{question:pd-top-conc}\ref{item:top-slice-injectivity} is
affirmative and that the splitting property in
Lemma~\ref{lemma:surj-inj-Phi_L-Phi_R}\ref{item:inj_Phi_L-Phi_R} holds.

In the literature, there are related results which provide affirmative answers
to the following splitting question for certain concordance invariants (or
obstructions): if $K$ and $J$ have relatively prime Alexander polynomials and if
the invariant vanishes for $K\# J$, then does the invariant vanish for each of
$K$ and~$J$? See work of S.-G. Kim~\cite{Kim:2005-2} which addresses the case of
Casson-Gordon invariants, and work of and S.-G. Kim and T.
Kim~\cites{Kim-Kim:2008-1,Kim-Kim:2014-1} on $L^2$-signature defects.  These may
be viewed as evidence supporting an affirmative answer to the injectivity part
for $\Phi_R^\top$ in
Question~\ref{question:pd-top-conc}\ref{item:right-pd-top-conc}, which implies
the splitting property stated in
Lemma~\ref{lemma:surj-inj-Phi_L-Phi_R}\ref{item:inj_Phi_L-Phi_R}. Also, T.
Kim~\cite{Kim:2017-1} proved results related to primary decomposition structures
in~$\cA^\top$.

\subsubsection*{Smooth concordance and topologically slice knots}

Let $\bK$ and $\chi$ be as above, but now write $K\sim K'$ if $K$ and $K'$ are
smoothly concordant.  Then $\cC^\smooth=\bK/{\sim}$ is the smooth knot
concordance group.

Similarly to the case of $\cC^\top$, the homomorphism
$
  \Phi^\smooth_L\colon
  \bigoplus_{[\lambda]} \cC^\smooth_\lambda/\Delta \to \cC^\smooth/\Delta
$
is not surjective since $\bigoplus_{[\lambda]} \cG^\Z_\lambda \to \cG^\Z$ is not
surjective.  So $\cC^\smooth$ is not left primary decomposable.

\begin{question}
  \phantomsection\leavevmode\Nopagebreak
  \label{question:pd-smooth-conc}
  \begin{enumerate}
    \item\label{item:smooth-slice-injectivity}
    Is $\Phi_L^\smooth \colon \bigoplus_{[\lambda]} \cC^\smooth_\lambda/\Delta
    \to \cC^\smooth/\Delta$ injective?
    \item\label{item:right-pd-smooth-conc} Is $\cC^\smooth$ right primary
    decomposable?
  \end{enumerate}
\end{question}

Again by Lemma~\ref{lemma:surj-inj-Phi_L-Phi_R}\ref{item:inj_Phi_L-Phi_R}, an
affirmative answer to
Question~\ref{question:pd-smooth-conc}\ref{item:right-pd-smooth-conc} implies
that the answer to
Question~\ref{question:pd-smooth-conc}\ref{item:smooth-slice-injectivity} is
affimative and that the splitting property stated in
Lemma~\ref{lemma:surj-inj-Phi_L-Phi_R}\ref{item:inj_Phi_L-Phi_R} holds.

The smooth concordance group of topologically slice knots $\cT$ is the kernel of
the natural homomorphism $\cC^\smooth \to \cC^\top$.  By
Theorem~\ref{theorem:extension-pd}, $\cC^\smooth$ is right primary decomposable
if both $\cT$ and $\cC^\top$ are right primary decomposable.

Question~\ref{question:pd-top-slice} in the introduction asks whether $\cT$ is
left and right primary decomposable in the sense of
Definition~\ref{definition:weak-pd}, and Theorem~\ref{theorem:main-top-slice}
provides supporting evidence for affirmative answers.

Recall that the definition in Section~\ref{subsection:pd-general-definition}
says that $[K]\in \cT_\lambda$ if $\Delta_K \doteq S(\lambda)^k$ for some $k\ge
0$. For topologically slice~$K$, $\Delta_K \doteq ff^*$ for some $f\in
\Q[t^{\pm1}]$, due to Fox and Milnor~\cite{Fox-Milnor:1966-1}.  From this it
follows that $\Delta_K \doteq S(\lambda)^k$ for some $k\ge 0$ if and only if
$\Delta_K \doteq (\lambda\lambda^*)^\ell$ for some $\ell\ge 0$.  Thus the
definition of $\cT_\lambda$ in Section~\ref{subsection:pd-general-definition}
agrees with that in Section~\ref{subsection:pd-top-slice}.

Livingston informed us that the techniques
of~\cite{Hedden-Livingston-Ruberman:2010-1} can be used to show the following:
there are non-associate self-dual irreducible polynomials $\lambda_1$,
$\lambda_2$ and $\lambda_3$ and a topologically slice knot $K$ with
$\Delta_K=(\lambda_1\lambda_2\lambda_3)^2$ which is not smoothly concordant to a
connected sum $K_1\#K_2\#K_3$ for any knots $K_1$, $K_2$ and $K_3$ with
$\Delta_{K_i}$ a power of~$\lambda_i$.  It says that $\cT$ does not satisfy
the strong existence condition~\ref{item:left-pd-strong-existence}.  This method
does not provide counterexamples to the left/right primary decomposability
of~$\cT$.

\subsubsection*{Filtrations of the knot concordance groups}

For the bipolar filtration $\{\cT_n\}$ of $\cT$ defined by Cochran, Harvey and
Horn~\cite{Cochran-Harvey-Horn:2012-1}, Question~\ref{question:pd-bipolar} in
the introduction asks whether the associated graded $\gr_n(\cT)=\cT_n/\cT_{n+1}$
is left and right primary decomposable in the sense of
Definition~\ref{definition:weak-pd}.  Theorem~\ref{theorem:main-bipolar}
supports an affirmative answer, by presenting a large subgroup which is left and
right (indeed strongly) primary decomposable into infinitely many infinite rank
primary parts.

In addition, whether $\cT_n$ is left/right primary decomposable appears to be an
interesting problem.  By Theorems~\ref{theorem:extension-pd}
and~\ref{theorem:extension-strong-pd}, $\cT_n$ is left/right primary
decomposable if so are both $\cT_{n+1}$ and~$\gr_n(\cT)$.

In~\cite{Cochran-Orr-Teichner:1999-1}, Cochran, Orr and Teichner introduced a
descending filtration
\[
  \{0\} \subset \cdots \subset \cF_{n.5} \subset \cF_n \subset \cdots \subset
  \cF_1 \subset \cF_{0.5} \subset \cF_{0} \subset \cC^\top
\]
of the topological knot concordance group~$\cC^\top$.  A knot $K$ represents an
element of $\cF_h$ ($h\in\frac12\Z_{\ge0}$) if $K$ is $h$-solvable in the sense
of~\cite[Definitions~8.5 and~8.7]{Cochran-Orr-Teichner:1999-1}.  

\begin{question}
Are $\cF_h$ and $\cF_h/\cF_{h+0.5}$ left and/or right primary decomposable?
\end{question}

Once again by Theorem~\ref{theorem:extension-pd}, $\cF_h$ is left/right primary
decomposable if so are both $\cF_{h+0.5}$ and~$\cF_n/\cF_{h+0.5}$. For integers
$n>0$, it is unknown whether~$\cF_{n.5}/\cF_{n+1}$ is nontrivial. Recently
Davis, Martin, Otto and Park showed that elements in $\cF_{0.5}$ represented by
a genus one knot are contained in~$\cF_1$~\cite{Davis-Martin-Otto-Park:2016-1}.
For the other half of the associated graded $\cF_n/\cF_{n.5}$, Cochran, Harvey
and Leidy provided strong evidence which support the conjecture that the
associated graded $\cF_n/\cF_{n.5}$ is right primary decomposable, for all
integers~${n\ge 0}$~\cites{Cochran-Harvey-Leidy:2009-2,
Cochran-Harvey-Leidy:2009-3}.  Indeed, they proposed a highly refined primary
decomposition conjecture for the associated graded $\cF_n/\cF_{n.5}$, using
non-commutative localizations~\cite[p.~444]{Cochran-Harvey-Leidy:2009-2}, and
they showed that refined primary parts reveal interesting
structures~\cites{Cochran-Harvey-Leidy:2009-2,Cochran-Harvey-Leidy:2009-3}. We
remark that aforementioned earlier works of S.-G. Kim and T. Kim
~\cites{Kim:2005-2,Kim-Kim:2008-1} also provide supporting evidence for the
case of~$\cF_1/\cF_{1.5}$.

\subsection{Rational homology 3-spheres and primary decomposition}
\label{subsection:pd-in-rational-H-cob}

Let $\Theta^\top_\Q$ and $\Theta^\smooth_\Q$ be the topological and smooth
rational homology cobordism groups of rational homology 3-spheres, respectively.
For a rational homology 3-sphere $Y$, let $\Delta_Y = |H_1(Y)|\in \Z$, the order
of the first homology with integral coefficients. The association $\chi\colon Y
\mapsto \Delta_Y = |H_1(Y)|$ satisfies conditions~\ref{item:delta-symmetry},
\ref{item:delta-addivitity} and~\ref{item:delta-negation}.  So, our general
definition of primary decomposition applies to $\Theta^\top_\Q$
and~$\Theta^\smooth_\Q$. 

For the topological case, the subgroup $\Delta=\{[Y] \in \Theta^\top_\Q \mid
H_1(Y)=0\}$ is trivial, since every integral homology 3-sphere bounds a
contractible compact 4-manifold~\cite[Section~9.3C]{Freedman-Quinn:1990-1}\@.
So, the left and right primary decompositions concern the homomorphisms
\[
  \Phi_L\colon \bigoplus_p (\Theta^\top_\Q)_p \to \Theta^\top_\Q, \quad
  \Phi_R\colon \Theta^\top_\Q \to \bigoplus_p \Theta^\top_\Q/(\Theta^\top_\Q)^p,
\]
where $p$ varies over primes in~$\Z$.  Here, the primary parts
$(\Theta^\top_\Q)_p$ is generated by $\Z[\frac1p]$-homology spheres, and
$(\Theta^\top_\Q)^p$ is generated by $\Z_p$-homology spheres.

The linking form $L_Y\colon H_1(Y)\times H_1(Y)\to \Q/\Z$ of a rational homology
3-sphere $Y$ gives fundamental information.  Algebraically, a $(\Q/\Z)$-valued
linking form is a nonsingular symmetric bilinear form $L\colon A \times A\to
\Q/\Z$ with $A$ a finite abelian group.  The Witt group of $(\Q/\Z)$-valued
linking forms, which we denote by $W(\Q/\Z)$, is defined in the standard manner.
There is a homomorphism $\Theta^\top_\Q \to W(\Q/\Z)$, which sends the class of
a rational homology sphere $Y$ to the Witt class of the associated linking
form~$L_Y$.  This is surjective due to~\cite{Kawauchi-Kojima:1980-1}.  The
following is a well known conjecture~\cite{Kim-Livingston:2014-1}*{p.~4}:

\begin{conjecture}
  \label{conjecture:QHcob-linking-form}
  The homomorphism $\Theta^\top_\Q \to W(\Q/\Z)$ is an isomorphism.
\end{conjecture}

Also, it is a standard fact that the Witt group $W(\Q/\Z)$ is left and right
primary decomposable.  (Indeed the strong existence
condition~\ref{item:left-pd-strong-existence} is also satisfied.)  The primary
part $W(\Q/\Z)_p$ is the Witt group of linking forms defined on $p$-torsion
finite abelian groups.  So, an affirmative answer to the above conjecture
implies that $\Theta^\top$ is left and right primary decomposable.

For the smooth case, the left
and right primary decompositions concern the homomorphisms
\[
  \Phi_L\colon \bigoplus_p (\Theta^\smooth_\Q)_p/\Delta 
  \to \Theta^\smooth_\Q/\Delta, \quad
  \Phi_R\colon \Theta^\smooth_\Q/\Delta 
  \to \bigoplus_p \Theta^\smooth_\Q/(\Theta^\smooth_\Q)^p,
\]
where $\Delta\subset \Theta^\smooth_\Q$ is the subgroup generated by the classes
of integral homology 3-spheres.  A result of S.-G. Kim and
Livingston~\cite[Proposition, p.~184]{Kim-Livingston:2014-1} says that this
$\Phi_L$ is not surjective, and thus $\Theta^\smooth_\Q$ is not left primary
decomposable.  We have the following questions.

\begin{question}
  \phantomsection\leavevmode\Nopagebreak
  \begin{enumerate}
    \item Is $\Phi_L\colon \bigoplus_p (\Theta^\smooth_\Q)_p/\Delta \to
    \Theta^\smooth_\Q / \Delta$ injective?
    \item Is $\Theta^\smooth_\Q$ right primary decomposable?
  \end{enumerate}
\end{question}

By Theorem~\ref{theorem:extension-pd}, $\Theta^\smooth_\Q$ is right primary
decomposable if Conjecture~\ref{conjecture:QHcob-linking-form} is true and the
rational homology cobordism group of rational homology 3-spheres bounding a
topological rational homology 4-ball is right primary decomposable.

\bibliography{research}

\end{document}